\documentclass[11pt,a4paper]{article}

\usepackage{geometry}
\usepackage[T1]{fontenc}
\usepackage[utf8]{inputenc}
\usepackage{authblk}

\usepackage{graphics}
\usepackage{tikz,pgf}
\usetikzlibrary{decorations.markings, decorations.pathmorphing}
\usetikzlibrary{arrows,decorations,backgrounds,scopes,plotmarks}
\usepackage{caption}
\usepackage{subfigure}
\usepackage{float}

\usepackage[english]{babel}
\usepackage[all]{xy}
\usepackage{marvosym}
\usepackage{stmaryrd,mathrsfs,amsmath,amssymb,bm,amsfonts}
\usepackage{enumerate}
\usepackage{amsmath,amsthm,amssymb,amscd}

\newtheorem{thm}{Theorem}[section]
\newtheorem{defn}[thm]{Definition}
\newtheorem{prop}[thm]{Proposition}
\newtheorem{coro}[thm]{Corollary}
\newtheorem{lem}[thm]{Lemma}
\newtheorem{rem}[thm]{Remark}
\newtheorem{ex}[thm]{Example}

\numberwithin{equation}{subsection}

\begin{document}

\title{A graphical calculus for semi-groupal categories}

\author[1,2]{Xuexing Lu}

\author[1,2]{Yu Ye}

\author[1,2]{Sen Hu}

\affil[1]{\small School of Mathematical Sciences, University of Science and Technology of China}
\affil[2]{Wu Wen-Tsun Key Laboratory of Mathematics, Chinese Academy of Sciences}

\renewcommand\Authands{ and }

\maketitle

\begin{abstract}
Around the year 1988, Joyal and Street established a graphical calculus for monoidal categories, which provides a firm foundation for many explorations of graphical notations in mathematics and physics. For a deeper understanding of their work, we consider a similar graphical calculus for semi-groupal categories.
We introduce two frameworks to formalize this graphical calculus, a topological one based on the notion of a processive plane graph and a combinatorial one based on the notion of a planarly ordered processive graph, which serves as a combinatorial counterpart of a deformation class of processive plane graphs. We demonstrate the equivalence of Joyal and Street's graphical calculus  and the theory of upward planar drawings. We introduce the category of semi-tensor schemes, and give a construction of a free monoidal category on a semi-tensor scheme. We deduce the unit convention as a kind of quotient construction, and show an idea to generalize the unit convention. Finally, we clarify the relation of the unit convention and Joyal and Street's construction of a free monoidal category on a tensor scheme.
\end{abstract}



\tableofcontents

\section{Introduction}

In \cite{[JS88],[JS91]}, Joyal and Street established a graphical calculus for monoidal categories (also called tensor categories), which provides a firm foundation for many explorations of graphical notations in mathematics and physics. They introduced the notion of a \textbf{progressive plane graph} (commonly known as a \textbf{string diagram}) and showed that the value of a \textbf{diagram} (labelled progressive plane graph) in a monoidal category is invariant under deformations of the underlying progressive plane graph. They also provided a construction of a free monoidal category on a \textbf{tensor scheme} by deformation classes of diagrams. See \cite{[S11]} and Section $2$ of \cite{[W13]} for good introductions.

Their framework of progressive plane graphs is topological, which encodes the abstract laws of tensor calculus in monoidal categories into the topology of progressive plane graphs such that all algebraic constructions relating to tensor calculus depend only on deformation classes. The topological nature is mainly manifested in the following two conventions:  one is the \textbf{identity convention} that drawing an identity morphism as an edge (see Fig \ref{iden}),
\begin{figure}[H]
\centering
$$
\begin{matrix}
\begin{matrix}
  \begin{tikzpicture}[scale=0.7]
\draw [dashed] (-2.5,2) rectangle (-0.5,-0.5);
\node (v1) at (-1.5,2) {};
\node (v2) at (-1.5,-0.5) {};
\node[scale=0.7] at (-1.3,1.5) {$X$};
\node [scale=0.7]at (-1.3,0) {$X$};
\draw  (-1.5,2) --(-1.5,-0.5)[postaction={decorate, decoration={markings,mark=at position .25 with {\arrow[black]{stealth}}}}][postaction={decorate, decoration={markings,mark=at position .75 with {\arrow[black]{stealth}}}}];
\draw[fill] (-1.5,0.75) circle [radius=0.07];
\node [scale=0.7]at (-2,0.7) {$Id_X$};
\end{tikzpicture}
\end{matrix}&=&
\begin{matrix}
 \begin{tikzpicture}[scale=0.7]
\draw [dashed] (-2.5,2) rectangle (-0.5,-0.5);
\node (v1) at (-1.5,2) {};
\node (v2) at (-1.5,-0.5) {};
\draw  (-1.5,2) --(-1.5,-0.5)[postaction={decorate, decoration={markings,mark=at position .5 with {\arrow[black]{stealth}}}}];
\node [scale=0.7]at (-2,0.7) {$Id_X$};
\end{tikzpicture}
\end{matrix}
&=&
\begin{matrix}
 \begin{tikzpicture}[scale=0.7]
\draw [dashed] (-2.5,2) rectangle (-0.5,-0.5);
\node (v1) at (-1.5,2) {};
\node (v2) at (-1.5,-0.5) {};
\draw  (-1.5,2) --(-1.5,-0.5)[postaction={decorate, decoration={markings,mark=at position .5 with {\arrow[black]{stealth}}}}];
\node [scale=0.7]at (-2,0.7) {$X$};
\end{tikzpicture}
\end{matrix}
\end{matrix}
$$
\caption{}
\label{iden}
\end{figure}
\noindent which, together with the \textbf{middle-four-interchange law} $(g\circ f)\otimes(g'\circ f')=(g\otimes g')\circ(f\otimes f'),$ implies the \textbf{level exchange property} (see Fig \ref{lep});
\begin{figure}[H]
\centering
$$
\begin{matrix}
\begin{matrix}
\begin{tikzpicture}[scale=1]
\draw [dashed] (-0.2,1.6) rectangle (1.8,-0.8);

\node (v5) at (0.3,1.6) {};
\node (v6) at (0.3,-0.8) {};
\draw  (0.3,1.6)  -- (0.3,-0.8)[postaction={decorate, decoration={markings,mark=at position .75 with {\arrow[black]{stealth}}}}][postaction={decorate, decoration={markings,mark=at position .1 with {\arrow[black]{stealth}}}}];
\node (v7) at (1.3,1.6) {};
\node (v8) at (1.3,-0.8) {};
\draw(1.3,1.6)  -- (1.3,-0.8)[postaction={decorate, decoration={markings,mark=at position .45 with {\arrow[black]{stealth}}}}][postaction={decorate, decoration={markings,mark=at position .9 with {\arrow[black]{stealth}}}}];

\node [left]at (0.3,0.9) {$f$};

\node[right]at (1.3,-0.2) {$g$};
\draw[fill] (0.3,0.9) circle [radius=0.05];
\draw[fill] (1.3,-0.2) circle [radius=0.05];
\node [scale=0.7]at (0.5,1.3) {$X$};
\node [scale=0.7]at (0.5,0.1) {$Y$};
\node [scale=0.7]at (1.1,0.9) {$X'$};
\node [scale=0.7]at (1.1,-0.5) {$Y'$};
\end{tikzpicture}
\end{matrix}&=&\begin{matrix}
\begin{tikzpicture}[scale=1]
\draw [dashed] (-0.2,1.6) rectangle (1.8,-0.8);

\node (v5) at (0.3,1.6) {};
\node (v6) at (0.3,-0.8) {};
\draw  (0.3,1.6)  -- (0.3,-0.8)[postaction={decorate, decoration={markings,mark=at position .45 with {\arrow[black]{stealth}}}}][postaction={decorate, decoration={markings,mark=at position .9 with {\arrow[black]{stealth}}}}];
\node (v7) at (1.3,1.6) {};
\node (v8) at (1.3,-0.8) {};
\draw(1.3,1.6)  -- (1.3,-0.8)[postaction={decorate, decoration={markings,mark=at position .75 with {\arrow[black]{stealth}}}}][postaction={decorate, decoration={markings,mark=at position .1 with {\arrow[black]{stealth}}}}];

\node [left]at (0.3,-0.2) {$f$};

\node[right]at (1.3,0.9) {$g$};
\draw[fill] (0.3,-0.2) circle [radius=0.05];
\draw[fill] (1.3,0.9) circle [radius=0.05];
\node [scale=0.7]at (0.5,0.7) {$X$};
\node [scale=0.7]at (0.5,-0.5) {$Y$};
\node [scale=0.7]at (1.1,1.3) {$X'$};
\node [scale=0.7]at (1.1,0.1) {$Y'$};
\end{tikzpicture}
\end{matrix}
\end{matrix}
$$
\caption{}
\label{lep}
\end{figure}
\noindent the other is the \textbf{unit convention} (Section $2.3$ of \cite{[BS10]}, Section $1.3$ of \cite{[S12]}) that drawing the unit object and its identity morphism as a blank space (see Fig \ref{unit}),

\begin{figure}[H]
\centering
$$
\begin{matrix}
\begin{matrix}
  \begin{tikzpicture}[scale=0.7]
\draw [dashed] (-2.5,2) rectangle (-0.5,-0.5);
\node (v1) at (-1.5,2) {};
\node (v2) at (-1.5,-0.5) {};
\draw  (-1.5,2) --(-1.5,-0.5)[postaction={decorate, decoration={markings,mark=at position .5 with {\arrow[black]{stealth}}}}];
\node [scale=0.7]at (-2,0.7) {$I$};
\end{tikzpicture}
\end{matrix}&=&
\begin{matrix}
  \begin{tikzpicture}[scale=0.7]
\draw [dashed] (-2.5,2) rectangle (-0.5,-0.5);
\node (v1) at (-1.5,2) {};
\node (v2) at (-1.5,-0.5) {};
\draw [dashed] (-1.5,2) --(-1.5,-0.5)[postaction={decorate, decoration={markings,mark=at position .5 with {\arrow[black]{stealth}}}}];
\node [scale=0.7]at (-2,0.7) {$I$};
\end{tikzpicture}
\end{matrix}&=&
\begin{matrix}
 \begin{tikzpicture}[scale=0.7]
\draw [dashed] (-2.5,2) rectangle (-0.5,-0.5);
\node [scale=0.7]at (-1.5,0.7) {$I$};
\end{tikzpicture}
\end{matrix}&=&\begin{matrix}
 \begin{tikzpicture}[scale=0.7]
\draw [dashed] (-2.5,2) rectangle (-0.5,-0.5);
\draw[fill] (-1.5,0.75) circle [radius=0.07];
\node [scale=0.7]at (-2,0.7) {$Id_I$};
\end{tikzpicture}
\end{matrix}
\end{matrix}
$$
\caption{}
\label{unit}
\end{figure}
\noindent under which the \textbf{unit axiom}  can be represented as Fig \ref{unax}, which, on the level of morphisms, can be represented by the equations $(f\circ Id_I )\otimes( Id_I\circ g)=( f\otimes Id_I)\circ(Id_I\otimes g)=f\circ g=f\otimes g$ and $( Id_I\circ g)\otimes(f\circ Id_I )=( Id_I\otimes f)\circ( g\otimes Id_I)=f\circ g=g\otimes f$.
\begin{figure}[H]
\centering
\begin{tikzpicture}
\node at (-1,0) {=};
\node [left](v2) at (-2,0.5) {$f$};
\draw[fill] (-2,0.5) circle [radius=0.055];
\node [left](v3) at (-2,-0.5) {$g$};
\node[right,scale=0.8] at (-2,0){$I$};
\draw[fill] (-2,-0.5) circle [radius=0.055];
\node [scale=0.8](v1) at (-2,1.5) {$X$};
\node [scale=0.8](v4) at (-2,-1.5) {$Y$};
\draw (v1) -- (-2,0.5)[postaction={decorate, decoration={markings,mark=at position 0.5 with {\arrow[black]{stealth}}}}];
\draw [dashed] (-2,0.5) --  (-2,-0.5)[postaction={decorate, decoration={markings,mark=at position 0.5 with {\arrow[black]{stealth}}}}];
\draw   (-2,-0.5) -- (v4)[postaction={decorate, decoration={markings,mark=at position 0.5 with {\arrow[black]{stealth}}}}];
\node [scale=0.8](v5) at (0,1) {$X$};
\node [left](v6) at (0,0) {$f$};
\draw[fill] (0,0) circle [radius=0.055];
\node [scale=0.8](v7) at (0,-1) {$I$};
\node [scale=0.8](v8) at (0.5,1) {$I$};
\node [right](v9) at (0.5,0) {$g$};
\draw[fill] (0.5,0) circle [radius=0.055];
\node [scale=0.8](v10) at (0.5,-1) {$Y$};
\draw  (v5) -- (0,0)[postaction={decorate, decoration={markings,mark=at position 0.5 with {\arrow[black]{stealth}}}}];
\draw [dashed] (0,0) -- (v7)[postaction={decorate, decoration={markings,mark=at position 0.5 with {\arrow[black]{stealth}}}}];
\draw  [dashed](v8) --(0.5,0)[postaction={decorate, decoration={markings,mark=at position 0.5 with {\arrow[black]{stealth}}}}];
\draw  (0.5,0) -- (v10)[postaction={decorate, decoration={markings,mark=at position 0.5 with {\arrow[black]{stealth}}}}];
\node at (-3,0) {=};
\node [scale=0.8](v14) at (-4,1) {$X$};
\node[right] (v15) at (-4,0) {$f$};
\draw[fill] (-4,0) circle [radius=0.055];
\node [scale=0.8](v16) at (-4,-1) {$I$};
\node[scale=0.8] (v11) at (-4.5,1) {$I$};
\node [left](v12) at (-4.5,0) {$g$};
\draw[fill] (-4.5,0) circle [radius=0.055];
\node[scale=0.8] (v13) at (-4.5,-1) {$Y$};
\draw  [dashed](v11) -- (-4.5,0)[postaction={decorate, decoration={markings,mark=at position 0.5 with {\arrow[black]{stealth}}}}];
\draw  (-4.5,0) -- (v13)[postaction={decorate, decoration={markings,mark=at position 0.5 with {\arrow[black]{stealth}}}}];
\draw  (v14) -- (-4,0)[postaction={decorate, decoration={markings,mark=at position 0.5 with {\arrow[black]{stealth}}}}];
\draw [dashed](-4,0)-- (v16)[postaction={decorate, decoration={markings,mark=at position 0.5 with {\arrow[black]{stealth}}}}];
\end{tikzpicture}
\caption{}
\label{unax}
\end{figure}
\noindent The unit convention also explains the appearance of isolated vertices in progressive plane graphs (see Fig \ref{isolated}).

\begin{figure}[h]
\centering
$$
\begin{matrix}
\begin{matrix}
  \begin{tikzpicture}[scale=0.7]
\draw [dashed] (-2.5,2) rectangle (-0.5,-0.5);
\node (v1) at (-1.5,2) {};
\node (v2) at (-1.5,-0.5) {};
\node[scale=0.7] at (-1.3,1.5) {$I$};
\node [scale=0.7]at (-1.3,0) {$I$};
\draw  (-1.5,2) --(-1.5,-0.5)[postaction={decorate, decoration={markings,mark=at position .25 with {\arrow[black]{stealth}}}}][postaction={decorate, decoration={markings,mark=at position .75 with {\arrow[black]{stealth}}}}];
\draw[fill] (-1.5,0.75) circle [radius=0.07];
\node [scale=0.7]at (-2,0.7) {$\alpha$};
\end{tikzpicture}
\end{matrix}&=&
\begin{matrix}
\begin{tikzpicture}[scale=0.7]
\draw [dashed] (-2.5,2) rectangle (-0.5,-0.5);
\node (v1) at (-1.5,2) {};
\node (v2) at (-1.5,-0.5) {};
\node[scale=0.7] at (-1.3,1.5) {$I$};
\node [scale=0.7]at (-1.3,0) {$I$};
\draw [dashed] (-1.5,2) --(-1.5,-0.5)[postaction={decorate, decoration={markings,mark=at position .25 with {\arrow[black]{stealth}}}}][postaction={decorate, decoration={markings,mark=at position .75 with {\arrow[black]{stealth}}}}];
\draw[fill] (-1.5,0.75) circle [radius=0.07];
\node [scale=0.7]at (-2,0.7) {$\alpha$};
\end{tikzpicture}
\end{matrix}
&=&
\begin{matrix}
 \begin{tikzpicture}[scale=0.7]
\draw [dashed] (-2.5,2) rectangle (-0.5,-0.5);
\draw[fill] (-1.5,0.75) circle [radius=0.07];
\node [scale=0.7]at (-2,0.7) {$\alpha$};
\end{tikzpicture}
\end{matrix}
\end{matrix}
$$
\caption{If $\alpha\not= Id_I$, then it is represented by an isolated vertex.}
\label{isolated}
\end{figure}
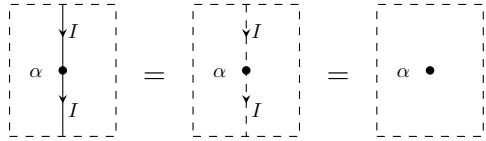

A \textbf{semi-groupal category} is a category equipped with an associative tensor product (but without specifying a unit object), or in other words, a semi-group object in the category $\mathbf{Cat}$ of categories.  A \textbf{semi-groupal functor} is a functor between two semi-groupal categories preserving the tensor products. The category $\mathbf{Semi.Gro}$ of semi-groupal categories and semi-groupal functors and the category $\mathbf{Mon.Cat}$ of monoidal categories and monoidal functors are closely related to each other. In fact, there is an adjunction $\mathfrak{F}:\mathbf{Semi.Gro} \rightleftharpoons \mathbf{Mon.Cat}:\mathfrak{U}$, where $\mathfrak{F}$ freely adjoins  a unit object to a semi-groupal category and $\mathfrak{U}$ treats a monoidal category just as a semi-groupal category.
(As a word of caution, in this paper, to avoid technical problems,  we suppose, when necessary, that all categories are small, and that all tensor products, units, semi-groupal functors and monoidal functors are strict.)

For a deeper understanding of Joyal and Street's work, we consider a graphical calculus for semi-groupal categories, which is same as that of Joyal and Street, except not referring to the unit convention. Same as in a monoidal category, the functorial property of the tensor product of a semi-groupal category can be equivalently represented as the middle-four-interchange law, which, under the identity convention, appears as the level change property in the graphical calculus.
It is easy to see that the graphical calculus for semi-groupal categories is also topological.

To formalize this graphical calculus, we introduce two frameworks, one is topological, similar as that of Joyal and Street, and the other is totally combinatorial.  In the topological framework, the central notion is that of a \textbf{processive plane graph} (\textbf{PPG} for short, Definition \ref{progre}); and in the combinatorial framework, the central notion is that of a \textbf{planarly ordered processive graph} (\textbf{POP-graph} for short, Definition \ref{pop-graph}), which servers as a combinatorial counterpart of a deformation class of processive plane graphs (\textbf{PPG-class}, for short). A concrete description of composition of POP-graphs (Definition \ref{com} and Theorem \ref{crux}) should be a key result in this framework. As two formalizations of one graphical calculus, the two frameworks have exactly the same structure.

To show the equivalence of the two frameworks, we introduce two monoidal categories, one is the monoidal category $\mathbf{PPG}$ of PPG-classes in the topological framework, and the other is the monoidal category $\mathcal{POP}$ of POP-graphs in the combinatorial framework. They are both free on the tensor scheme $\mathbf{PRM}$ of prime PPG-classes and the tensor scheme $\mathcal{PRM}$ of prime POP-graphs, respectively.  Then the equivalence of the two frameworks comes down to the equivalence of $\mathbf{PPG}$ and $\mathcal{POP}$ (Theorem \ref{main}), which follows from their freeness and the equivalence of $\mathbf{PRM}$ and $\mathcal{PRM}$ (by Lemma \ref{basic}). See Fig \ref{sum} for a summary.

\begin{figure}[h]
\centering
\begin{tikzpicture}
\draw  (-4.5,1.5) rectangle (6,-2.5);
\node (v7) at (-1,1.5) {};
\node (v8) at (-1,-2.5) {};
\node (v9) at (2.5,1.5) {};
\node (v10) at (2.5,-2.5) {};
\node (v1) at (-4.5,0.5) {};
\node (v2) at (6,0.5) {};
\node (v3) at (-4.5,-0.5) {};
\node (v4) at (6,-0.5) {};
\node (v5) at (-4.5,-1.5) {};
\node (v6) at (6,-1.5) {};
\draw  (-4.5,0.5) -- (6,0.5);
\draw  (-4.5,-0.5)-- (6,-0.5);
\draw   (-4.5,-1.5)-- (6,-1.5);
\draw  (-1,1.5) -- (-1,-2.5);
\draw  (2.5,1.5)-- (2.5,-2.5);
\node at (-2.8,1) {frameworks};
\node at (-2.8,0) {graphs};
\node at (-2.8,-1) {monoidal categories};
\node at (-2.8,-2) {tensor schemes};
\node at (0.8,1) {topology};
\node at (0.8,0) {PPG-class};
\node at (0.8,-1) {$\mathbf{PPG}$};
\node at (0.8,-2) {$\mathbf{PRM}$};
\node at (4.2,1) {combinatorics};
\node at (4.2,0) {POP-graph};
\node at (4.2,-1) {$\mathcal{POP}$};
\node at (4.2,-2) {$\mathcal{PRM}$};
\end{tikzpicture}
\caption{}
\label{sum}
\end{figure}

Although equivalent, the two frameworks are, in a sense, complementary to each other. In practice, the topological framework, as that of Joyal and Street, is effective and human-readable, which allows people to "see" the process of calculating or proving. While the combinatorial framework is formal and machine-processable, which makes it much easier for us to solve some problems about PPGs by a computer, for example, to enumerate all PPG-classes with fixed number of edges.

In Remark \ref{two}, \ref{three}, \ref{seven}, \ref{four}, \ref{five} and \ref{six}, we demonstrate the equivalence of  Joyal and Street's graphical calculus and the theory of \textbf{upward planar graphs} \cite{[GT95]}, where PPGs and progressive plane graphs (without isolated vertices) are essentially equivalent to upward \textbf{plane $st$ graphs} and upward plane graphs, respectively. The equivalence  sheds some light on the study of upward planarity, especially on how to develop a higher genus theory of upward planarity (or called a \textbf{topological order theory}, which is expected to be a directed version of \textbf{topological graph theory} \cite{[A96]}), where Joyal and Street's graphical calculus for symmetric monoidal categories (Chapter $2$ of \cite{[JS91]}) provides a natural background.  A detailed explanation of this higher genus theory will be given in other place.

As an application, we show a construction of free monoidal categories by POP-graphs. For this purpose, we introduce a category $\mathbf{Semi.Ten}$ of \textbf{semi-tensor schemes} and their morphisms, and define a commutative diagram of adjunctions (see Fig \ref{23}), where $\mathcal{F}^+$ and $\mathcal{F}$, when applying on a semi-tensor scheme, produce a free semi-groupal category and a free monoidal category, respectively; $\mathcal{U}^+$ and $\mathcal{U}$ are defined by the construction of prime POP-diagrams.
In this more general context, the unit convention turns out to be a kind of quotient construction and can be systematically generalized.
Finally, we extend Joyal and Street's construction of a free monoidal category on a tensor scheme into an adjunction, which clarifies the relation of their construction and the unit convention.

\begin{figure}[H]
\centering
\begin{tikzpicture}

\node at (-3.2,-2) {$\mathbf{Semi.Ten}$};
\node at (3.5,-2) {$\mathbf{Mon.Cat}$};
\node at (0,1.5) {$\mathbf{Semi.Gro}$};
\node (v1) at (-2.6,-1.3) {};
\node (v2) at (-0.4,1.3) {};
\node (v3) at (-0.2,1.1) {};
\node (v4) at (-2.4,-1.5) {};
\draw [-latex](v1) -- (v2);
\draw [-latex] (v3) -- (v4);
\node (v5) at (-2.3,-1.9) {};
\node (v6) at (2.6,-1.9) {};
\node (v7) at (2.6,-2.2) {};
\node (v8) at (-2.3,-2.2) {};
\draw [-latex] (v5)-- (v6);
\draw [-latex] (v7) -- (v8);
\node (v12) at (0.3,1) {};
\node (v11) at (2.6,-1.6) {};
\node (v10) at (2.8,-1.4) {};
\node at (0.6,1.2) {};
\node (v9) at (0.5,1.2) {};
\draw [-latex] (v9) -- (v10);
\draw [-latex] (v11)-- (v12);
\node at (-1.7,0.3) {$\mathcal{F}^+$};
\node at (-1,-0.5) {$\mathcal{U}^+$};
\node at (1.1,-0.5) {$\mathfrak{U}$};
\node at (1.8,0.2) {$\mathfrak{F}$};
\node at (0,-1.6) {$\mathcal{F}$};
\node at (0,-2.5) {$\mathcal{U}$};
\end{tikzpicture}
\caption{}
\label{23}
\end{figure}

The paper is organized as follows. In Section $2$, we introduce a topological framework for the graphical calculus.  In Section $3$, we introduce a combinatorial framework for the graphical calculus. In Section $4$, we first reformulate the notion of a processive plane graph, and then show that the two frameworks are equivalent (Theorem \ref{main}). In Section $5$, we study some basic properties of POP-graphs. Section $6$ is devoted to the proof of a key result, Theorem \ref{crux}, which justifies the definition of composition of POP-graphs (Definition \ref{com}). In Section $7$,  we study the decomposition and cancellation properties of the tensor product and composition of POP-graphs. In Section $8$, we prove Theorem \ref{free2}, which shows the freeness of $\mathcal{POP}$. In Section $9$, we introduce the category of semi-tensor schemes,  and using POP-graphs, give a construction of a free monoidal category on a semi-tensor scheme. In Section $10$, we give an algebraic explanation of the unit convention in our general context and show an idea to generalize it. In Section $11$, we extend Joyal and Street's construction of a free monoidal category on a tensor scheme into an adjunction and show that it is naturally compatible with the unit convention.

\section{A topological framework}
In this section, we show a topological framework for the graphical calculus.
We begin by introducing the key notion in this framework.

\begin{defn}\label{progre}
A \textbf{processive plane graph}, or \textbf{PPG}, is a non-empty directed graph drawn in a plane box such that $(1)$ all edges monotonically decrease in the vertical direction; $(2)$ all sinks and sources have degree one and $(3)$ all sources and sinks are on the horizontal boundaries of the plane box.
\end{defn}

Fig \ref{1} shows an example of PPG. The condition $(1)$ is called an \textbf{upward property}. A planar drawing of directed graph is called \textbf{upward} if all its edges monotonically decrease in the vertical direction (or other fixed direction). Clearly, a necessary condition for a directed graph to have an upward planar drawing is that it is acyclic. So a PPG is acyclic, and therefore has at least one source and at least one sink.

\begin{figure}[h]
\centering
\begin{tikzpicture}[scale=0.32]
\node (v2) at (-4,3) {};
\draw[fill] (-1.5,5.5) circle [radius=0.15];
\node (v1) at (-1.5,5.5) {};
\node (v7) at (-1.5,1) {};
\node (v9) at (1.5,5.5) {};
\node (v14) at (2,1.5) {};
\node (v3) at (-3,7.5) {};
\node (v4) at (-2,7.5) {};
\node (v5) at (-0.5,7.5) {};
\node (v6) at (-4.8,7.5) {};
\node (v11) at (-4.5,-1) {};
\node (v12) at (-2,-1) {};
\node (v13) at (0,-1) {};
\node (v15) at (2,-1) {};
\node (v8) at (1,7.5) {};
\node (v10) at (2.5,7.5) {};
\node  at (-2.5,3.5) {};
\node  at (-3,5.2) {};
\node  at (-1.2,3.3) {};
\node  at (0.5,3.25) {};
\node  at (2.2,3.7) {};
\node  at (-3,1.7) {};
\draw[fill] (-4,3) circle [radius=0.15];
\draw[fill] (v1) circle [radius=0.15];
\draw[fill] (v7) circle [radius=0.15];
\draw[fill] (v9) circle [radius=0.15];
\draw[fill] (v14) circle [radius=0.15];
\draw[fill] (v1) circle [radius=0.15];
\draw[fill] (v2) circle [radius=0.15];
\draw[fill] (v3) circle [radius=0.15];
\draw[fill] (v4) circle [radius=0.15];
\draw[fill] (v5) circle [radius=0.15];
\draw[fill] (v6) circle [radius=0.15];
\draw[fill] (v8) circle [radius=0.15];
\draw[fill] (v10) circle [radius=0.15];
\draw[fill] (v11) circle [radius=0.15];
\draw[fill] (v12) circle [radius=0.15];
\draw[fill] (v13) circle [radius=0.15];
\draw[fill] (v15) circle [radius=0.15];

\draw  plot[smooth, tension=1] coordinates {(v1) (-2.5,5)  (-3.5,4) (v2)}[postaction={decorate, decoration={markings,mark=at position .5 with {\arrow[black]{stealth}}}}];
\draw  plot[smooth, tension=1] coordinates {(v1) (-2,4.5)  (-3,3.5) (v2)}[postaction={decorate, decoration={markings,mark=at position .5 with {\arrow[black]{stealth}}}}];

\draw  (-3,7.5) -- (-1.5,5.5)[postaction={decorate, decoration={markings,mark=at position .5 with {\arrow[black]{stealth}}}}];
\draw  (-2,7.5) -- (-1.5,5.5)[postaction={decorate, decoration={markings,mark=at position .5 with {\arrow[black]{stealth}}}}];
\draw  (-0.5,7.5) -- (-1.5,5.5)[postaction={decorate, decoration={markings,mark=at position .5 with {\arrow[black]{stealth}}}}];

\draw  (-4.8,7.5)-- (-4,3)[postaction={decorate, decoration={markings,mark=at position .5 with {\arrow[black]{stealth}}}}];
\draw  (-1.5,5.5)  -- (-1.5,1)[postaction={decorate, decoration={markings,mark=at position .5 with {\arrow[black]{stealth}}}}];
\draw  (-4,3) -- (-1.5,1)[postaction={decorate, decoration={markings,mark=at position .5 with {\arrow[black]{stealth}}}}];

\draw (1,7.5)--(1.5,5.5)[postaction={decorate, decoration={markings,mark=at position .5 with {\arrow[black]{stealth}}}}];
\draw  (2.5,7.5) -- (1.5,5.5)[postaction={decorate, decoration={markings,mark=at position .5 with {\arrow[black]{stealth}}}}];
\draw  (1.5,5.5) -- (-1.5,1)[postaction={decorate, decoration={markings,mark=at position .5 with {\arrow[black]{stealth}}}}];
\draw  (-4,3) -- (-4.5,-1)[postaction={decorate, decoration={markings,mark=at position .5 with {\arrow[black]{stealth}}}}];
\draw  (-1.5,1) -- (-2,-1)[postaction={decorate, decoration={markings,mark=at position .65 with {\arrow[black]{stealth}}}}];
\draw  (0,-1) -- (-1.5,1)[postaction={decorate, decoration={markings,mark=at position .5 with {\arrowreversed[black]{stealth}}}}];
\draw  (1.5,5.5) -- (2,1.5)[postaction={decorate, decoration={markings,mark=at position .5 with {\arrow[black]{stealth}}}}];
\draw  (2,1.5) -- (2,-1)[postaction={decorate, decoration={markings,mark=at position .5 with {\arrow[black]{stealth}}}}];

\node (v17) at (6.5,7.5) {};
\node (v16) at (6.5,-1) {};
\draw[fill] (v16) circle [radius=0.15];
\draw[fill] (v17) circle [radius=0.15];
\draw  (6.5,-1) -- (6.5,7.5)[postaction={decorate, decoration={markings,mark=at position .5 with {\arrowreversed[black]{stealth}}}}];
\node (v18) at (4.5,7.5) {};
\node (v19) at (4.5,-1) {};
\node (v20) at (4.5,3.25) {};
\draw[fill] (v18) circle [radius=0.15];
\draw[fill] (v19) circle [radius=0.15];

\draw[fill] (v20) circle [radius=0.15];
\draw  (4.5,-1) -- (4.5,3.25)[postaction={decorate, decoration={markings,mark=at position .5 with {\arrowreversed[black]{stealth}}}}];
\draw  (4.5,3.24) -- (4.5,7.5)[postaction={decorate, decoration={markings,mark=at position .5 with {\arrowreversed[black]{stealth}}}}];

\draw [loosely dashed] (-6.5,7.5)--(-6.5,-1);
\draw [loosely dashed] (8,7.5)--(8,-1);
\draw [loosely dashed] (-6.5,7.5)--(8,7.5);
\draw [loosely dashed] (-6.5,-1)--(8,-1);

\end{tikzpicture}
\caption{A processive plane graph}
\label{1}
\end{figure}
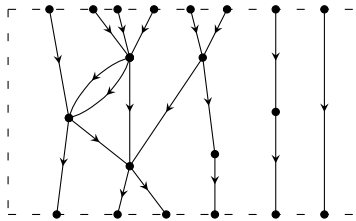

In this paper, we adopt the convention that an \textbf{isolated vertex} (with degree zero) is both an source (with indegree zero) and an sink (with outdegree zero). So the condition $(2)$ implies that a PPG has no isolated vertex.

Since a PPG is upward, then the condition $(3)$ can be replaced by the \textbf{boxed condition} that all source are drawn on one horizontal boundary of the plane box and all sinks are drawn on the other horizontal boundary of the plane box.
\begin{rem}
Strictly, a graph is a pure combinatorial object defined by a vertex set, an edge set and an incident relation, which does not refer to any geometric and topological notions, such as line segment, planar drawing, etc. However, we think that it is helpful to use some geometric languages, and we will not make a distinction between a graph and its \textbf{geometric representation} (or \textbf{drawing}). A plane graph is a planar drawing of a graph, or in other words, a geometric representation of a graph in the plane.
\end{rem}

\begin{rem}\label{two}
Note that Definition \ref{progre} is a restriction of that of a progressive plane graph introduced by Joyal and Street (Definition $1.1$ in \cite{[JS91]}, see also Definition $13$ in Chapter $2$ of \cite{[W13]}), which is an upward planar drawing of a non-empty directed graph (possibly with isolated vertices) in a plane box such that all vertices drawn on one horizontal boundary of the plane box are of degree one. It is easy to see that any progressive plane graph can be extended (in a non-unique way) into a PPG, see Fig \ref{25} for an example
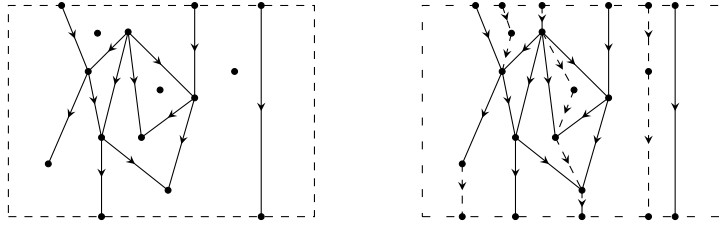
\begin{figure}[htbp]
\centering
$$
\begin{matrix}\begin{matrix}
\begin{tikzpicture}[scale=0.35]

\draw [dashed] (-3,1.5) rectangle (8.5,-6.5);
\node [above](v1) at (-1,1.5) {};
\node (v2) at (0,-1) {};
\node (v6) at (-1.5,-4.5) {};
\node [left] at (-1.5,-4.5) {};
\node (v3) at (1.5,0.5) {};
\node (v8) at (4,-2) {};
\node (v4) at (0.5,-3.5) {};
\node [below](v5) at (0.5,-6.5) {};
\node [above](v7) at (4,1.5) {};
\draw  (-1,1.5)  -- (0,-1)[postaction={decorate, decoration={markings,mark=at position .5 with {\arrow[black]{stealth}}}}];
\draw   (1.5,0.5) -- (0,-1)[postaction={decorate, decoration={markings,mark=at position .5 with {\arrow[black]{stealth}}}}];
\draw   (1.5,0.5)--  (0.5,-3.5)[postaction={decorate, decoration={markings,mark=at position .5 with {\arrow[black]{stealth}}}}];
\draw  (0.5,-3.5) -- (v5)[postaction={decorate, decoration={markings,mark=at position .5 with {\arrow[black]{stealth}}}}];
\draw  (0,-1) -- (-1.5,-4.5)[postaction={decorate, decoration={markings,mark=at position .5 with {\arrow[black]{stealth}}}}];
\draw  (0,-1) -- (0.5,-3.5)[postaction={decorate, decoration={markings,mark=at position .5 with {\arrow[black]{stealth}}}}];
\draw  (4,1.5)-- (4,-2)[postaction={decorate, decoration={markings,mark=at position .5 with {\arrow[black]{stealth}}}}];
\draw   (1.5,0.5)-- (4,-2)[postaction={decorate, decoration={markings,mark=at position .5 with {\arrow[black]{stealth}}}}];
\node (v9) at (3,-5.5) {};
\draw  (0.5,-3.5) -- (3,-5.5)[postaction={decorate, decoration={markings,mark=at position .5 with {\arrow[black]{stealth}}}}];
\draw  (4,-2) -- (3,-5.5)[postaction={decorate, decoration={markings,mark=at position .5 with {\arrow[black]{stealth}}}}];
\node (v10) at (2,-3.5) {};
\draw   (1.5,0.5) -- (2,-3.5)[postaction={decorate, decoration={markings,mark=at position .5 with {\arrow[black]{stealth}}}}];
\draw  (4,-2) -- (2,-3.5)[postaction={decorate, decoration={markings,mark=at position .5 with {\arrow[black]{stealth}}}}];

\draw [fill](v2) circle [radius=0.11];
\draw [fill](v3) circle [radius=0.11];
\draw [fill](v4) circle [radius=0.11];
\draw [fill](v6) circle [radius=0.11];
\draw [fill](v8) circle [radius=0.11];
\draw [fill](v9) circle [radius=0.11];
\draw [fill](v10) circle [radius=0.11];

\node at (-1,0.5) {};
\node at (4.5,0) {};
\node at (0,-5.5) {};
\node at (-1.5,-3) {};

\draw(6.5,1.5)--(6.5,-6.5)[postaction={decorate, decoration={markings,mark=at position .5 with {\arrow[black]{stealth}}}}];
\node [above]at(6.5,1.5) {};
\node [below]at (6.5,-6.5) {};
\node at (7,-3) {};

\draw [fill] (-1,1.5) circle [radius=0.11];
\draw [fill](0.5,-6.5) circle [radius=0.11];
\draw [fill](4,1.5) circle [radius=0.11];
\draw [fill](6.5,-6.5) circle [radius=0.11];
\draw [fill](6.5,1.5) circle [radius=0.11];

\draw [fill](5.5,-1) circle [radius=0.11];
\draw [fill](2.7,-1.7) circle [radius=0.11];

\draw [fill](0.35,0.45) circle [radius=0.11];
\end{tikzpicture}
\end{matrix}&&&&\begin{matrix}
  \begin{tikzpicture}[scale=0.35]

\draw [loosely dashed] (-3,1.5) rectangle (8.5,-6.5);
\node [above](v1) at (-1,1.5) {};
\node (v2) at (0,-1) {};
\node (v6) at (-1.5,-4.5) {};
\node [left] at (-1.5,-4.5) {};
\node (v3) at (1.5,0.5) {};
\node (v8) at (4,-2) {};
\node (v4) at (0.5,-3.5) {};
\node [below](v5) at (0.5,-6.5) {};
\node [above](v7) at (4,1.5) {};
\draw  (-1,1.5)  -- (0,-1)[postaction={decorate, decoration={markings,mark=at position .5 with {\arrow[black]{stealth}}}}];
\draw   (1.5,0.5) -- (0,-1)[postaction={decorate, decoration={markings,mark=at position .5 with {\arrow[black]{stealth}}}}];
\draw   (1.5,0.5)--  (0.5,-3.5)[postaction={decorate, decoration={markings,mark=at position .5 with {\arrow[black]{stealth}}}}];
\draw  (0.5,-3.5) -- (v5)[postaction={decorate, decoration={markings,mark=at position .5 with {\arrow[black]{stealth}}}}];
\draw  (0,-1) -- (-1.5,-4.5) node (v19) {}[postaction={decorate, decoration={markings,mark=at position .5 with {\arrow[black]{stealth}}}}];
\draw  (0,-1) node (v16) {} -- (0.5,-3.5)[postaction={decorate, decoration={markings,mark=at position .5 with {\arrow[black]{stealth}}}}];
\draw  (4,1.5)-- (4,-2)[postaction={decorate, decoration={markings,mark=at position .5 with {\arrow[black]{stealth}}}}];
\draw   (1.5,0.5)-- (4,-2)[postaction={decorate, decoration={markings,mark=at position .5 with {\arrow[black]{stealth}}}}];
\node (v9) at (3,-5.5) {};
\draw  (0.5,-3.5) -- (3,-5.5)[postaction={decorate, decoration={markings,mark=at position .5 with {\arrow[black]{stealth}}}}];
\draw  (4,-2) -- (3,-5.5) node (v22) {}[postaction={decorate, decoration={markings,mark=at position .5 with {\arrow[black]{stealth}}}}];
\node (v10) at (2,-3.5) {};
\draw   (1.5,0.5) node (v18) {} -- (2,-3.5)[postaction={decorate, decoration={markings,mark=at position .5 with {\arrow[black]{stealth}}}}];
\draw  (4,-2) -- (2,-3.5) node (v21) {}[postaction={decorate, decoration={markings,mark=at position .5 with {\arrow[black]{stealth}}}}];

\draw [fill](v2) circle [radius=0.11];
\draw [fill](v3) circle [radius=0.11];
\draw [fill](v4) circle [radius=0.11];
\draw [fill](v6) circle [radius=0.11];
\draw [fill](v8) circle [radius=0.11];
\draw [fill](v9) circle [radius=0.11];
\draw [fill](v10) circle [radius=0.11];

\node at (-1,0.5) {};
\node at (4.5,0) {};
\node at (0,-5.5) {};
\node at (-1.5,-3) {};

\draw(6.5,1.5)--(6.5,-6.5)[postaction={decorate, decoration={markings,mark=at position .5 with {\arrow[black]{stealth}}}}];
\node [above]at(6.5,1.5) {};
\node [below]at (6.5,-6.5) {};
\node at (7,-3) {};

\draw [fill] (-1,1.5) circle [radius=0.11];
\draw [fill](0.5,-6.5) circle [radius=0.11];
\draw [fill](4,1.5) circle [radius=0.11];
\draw [fill](6.5,-6.5) circle [radius=0.11];
\draw [fill](6.5,1.5) circle [radius=0.11];

\draw [fill](5.5,-1) circle [radius=0.11];
\draw [fill](2.7,-1.7) node (v24) {} circle [radius=0.11];

\draw [fill](0.35,0.45) node (v15) {} circle [radius=0.11];

\node (v11) at (5.5,1.5) {};
\draw [fill](5.5,1.5) circle [radius=0.11];
\node (v12) at (5.5,-1) {};

\node (v13) at (5.5,-6.5) {};
\draw [fill](5.5,-6.5) circle [radius=0.11];
\draw [dashed] (5.5,1.5) --(5.5,-1)[postaction={decorate, decoration={markings,mark=at position .5 with {\arrow[black]{stealth}}}}];
\draw  [dashed](5.5,-1) -- (5.5,-6.5)[postaction={decorate, decoration={markings,mark=at position .5 with {\arrow[black]{stealth}}}}];
\node (v14) at (0,1.5) {};
\draw [fill](v14) circle [radius=0.11];
\draw  [dashed](0,1.5)-- (0.35,0.45)[postaction={decorate, decoration={markings,mark=at position .55 with {\arrow[black]{stealth}}}}];
\draw  [dashed](0.35,0.45)-- (0,-1)[postaction={decorate, decoration={markings,mark=at position .65 with {\arrow[black]{stealth}}}}];
\node (v17) at (1.5,1.5) {};
\draw [fill](v17) circle [radius=0.11];

\draw[dashed]  (1.5,1.5)-- (1.5,0.5)[postaction={decorate, decoration={markings,mark=at position .65 with {\arrow[black]{stealth}}}}];
\node (v20) at (-1.5,-6.5) {};
\draw [fill](v20) circle [radius=0.11];
\draw [dashed] (-1.5,-4.5) -- (-1.5,-6.5)[postaction={decorate, decoration={markings,mark=at position .5 with {\arrow[black]{stealth}}}}];
\draw [dashed] (2,-3.5)-- (3,-5.5)[postaction={decorate, decoration={markings,mark=at position .5 with {\arrow[black]{stealth}}}}];
\node (v23) at (3,-6.5) {};
\draw [fill](v23) circle [radius=0.11];
\draw  [dashed](3,-5.5) -- (3,-6.5)[postaction={decorate, decoration={markings,mark=at position .65 with {\arrow[black]{stealth}}}}];
\draw [dashed] (2.7,-1.7) -- (2,-3.5)[postaction={decorate, decoration={markings,mark=at position .5 with {\arrow[black]{stealth}}}}];
\draw [dashed] (1.5,0.5) -- (2.7,-1.7)[postaction={decorate, decoration={markings,mark=at position .65 with {\arrow[black]{stealth}}}}];
\end{tikzpicture}
\end{matrix}
\end{matrix}
$$
\caption{A progressive plane graph and one of its PPG-extentions}
\label{25}
\end{figure}
\end{rem}

Following Joyal and Street \cite{[JS91]}, tensor product and composition of PPGs are defined as follows.
Fix the plane box to be $[-1,1]\times [-1,1]\subset \mathbb{R}^2$ and write $G:m\rightarrow n$ if $G$ has $m$ sources and $n$ sinks.
Define the functions $\gamma,\tau:\mathbb{R}^2\rightarrow \mathbb{R}^2$ as $$\gamma(x,t)=(x,\frac{1}{3}t),\ \ \ \tau(x,t)=(\frac{1}{2}x,t)$$
and the points $e_1=(1,0),\ e_2=(0,1)\in \mathbb{R}^2.$ Notation such as $\gamma(S+e_2),$ for $S\subset \mathbb{R}^2,$ denotes the set $$\left\{(x,\frac{1}{3}(t+1))\in \mathbb{R}^2\   \middle \vert\   (x,t)\in S\right\}.$$

Let $G_1$, $G_2$ be two PPGs. Their \textbf{tensor product} $G_1\otimes G_2$ is the PPG consisting of the space $$\tau\Big ((G_1-e_1)\sqcup(G_2+e_1)\Big)$$ with  $\tau\Big((V(G_1)-e_1)\sqcup(V(G_2)+e_1)\Big)$ as the set of vertices, where $V(G)$ denotes the vertex set of graph $G$. Ignoring translations, we depict this as Fig \ref{5}.

\begin{figure}[htbp]
\centering
$$\begin{matrix}
\begin{matrix}
\begin{tikzpicture}
\node (v1) at (-1,1) {};
\node (v4) at (-1,-1) {};
\node (v2) at (1,1) {};
\node (v3) at (1,-1) {};
\draw[dashed]  (-1,1)--(1,1);
\draw  [dashed](1,1) -- (1,-1);
\draw  [dashed](-1,1) -- (-1,-1);
\draw  [dashed](-1,-1) -- (1,-1);
\node at (0,0) {$G_1$};
\end{tikzpicture}\end{matrix}&\otimes&\begin{matrix}\begin{tikzpicture}
\node (v1) at (-1,1) {};
\node (v4) at (-1,-1) {};
\node (v2) at (1,1) {};
\node (v3) at (1,-1) {};
\draw[dashed]  (-1,1)--(1,1);
\draw[dashed]  (1,1) -- (1,-1);
\draw[dashed]  (-1,1) -- (-1,-1);
\draw[dashed]  (-1,-1) -- (1,-1);
\node at (0,0) {$G_2$};
\end{tikzpicture}\end{matrix}&=&\begin{matrix}\begin{tikzpicture}
\node (v1) at (-1,1) {};
\node (v4) at (-1,-1) {};
\node (v2) at (1,1) {};
\node (v3) at (1,-1) {};
\draw [dashed] (-1,1)--(1,1);
\draw[dashed]  (1,1) -- (1,-1);
\draw [dashed] (-1,1) -- (-1,-1);
\draw [dashed] (-1,-1) -- (1,-1);

\node (v5) at (0,1) {};
\node (v6) at (0,-1) {};
\draw [dashed] (0,1) edge (0,-1);
\node at (-0.5,0) {$\tau G_1$};
\node at (0.5,0) {$\tau G_2$};
\end{tikzpicture}
\end{matrix}
 \end{matrix}
 $$
\caption{}
\label{5}
\end{figure}

Suppose $G_1:l\rightarrow m$, $G_2:m\rightarrow n$ are PPGs. Let $a_1<\cdots<a_m$ be the sinks of $G_1-2e_2$, and let $b_1<\cdots<b_m$ be the sources of $G_2+2e_2$. The \textbf{composition} $G_2\circ G_1:l\rightarrow n$ is the PPG consisting of the space $$G_2\circ G_1=\gamma\Big((G^1-2e_2)\sqcup[a_1,b_1]\sqcup\cdots\sqcup[a_m,b_m]\sqcup(G^2+2e_2)\Big)$$
with $\gamma\Big(V(G_1-2e_2)-\{a_1,\cdots,a_m\})\sqcup\gamma((G_2+2e_2)-\{b_1,\cdots,b_n\}\Big)$ as the set of vertices, where $[a,b]\subset\mathbb{R}^2$ is the segment between the points $a$ and $b$. We depict this as Fig \ref{6}.

\begin{figure}[htbp]
\centering
$$\begin{matrix}\begin{matrix}\begin{tikzpicture}
\node (v1) at (-1,1) {};
\node (v4) at (-1,-1) {};
\node (v2) at (1,1) {};
\node (v3) at (1,-1) {};
\draw[dashed]  (-1,1)--(1,1);
\draw [dashed] (1,1) -- (1,-1);
\draw [dashed] (-1,1) -- (-1,-1);
\draw[dashed]  (-1,-1) -- (1,-1);
\node at (0,0) {$G_2$};
\end{tikzpicture}\end{matrix}&\circ&\begin{matrix}\begin{tikzpicture}
\node (v1) at (-1,1) {};
\node (v4) at (-1,-1) {};
\node (v2) at (1,1) {};
\node (v3) at (1,-1) {};
\draw[dashed]  (-1,1)--(1,1);
\draw [dashed] (1,1) -- (1,-1);
\draw [dashed] (-1,1) -- (-1,-1);
\draw [dashed] (-1,-1) -- (1,-1);
\node at (0,0) {$G_1$};
\end{tikzpicture}\end{matrix}&=&\begin{matrix}\begin{tikzpicture}
\node (v1) at (-1,1) {};
\node (v4) at (-1,-1) {};
\node (v2) at (1,1) {};
\node (v3) at (1,-1) {};
\draw[dashed]  (-1,1)--(1,1);
\draw [dashed] (1,1) -- (1,-1);
\draw [dashed] (-1,1) -- (-1,-1);
\draw [dashed] (-1,-1) -- (1,-1);

\node (v5) at (-1,0.5) {};
\node (v6) at (1,0.5) {};
\node (v7) at (-1,-0.5) {};
\node (v8) at (1,-0.5) {};
\draw [dashed] (-1,0.5) --(1,0.5);
\draw[dashed]  (-1,-0.5) --(1,-0.5);
\node at (0,0.75) {$\gamma G_1$};
\node at (0,-0.75) {$\gamma G_2$};
\node (v9) at (-0.5,0.5) {};
\node (v10) at (-0.5,-0.5) {};
\node (v11) at (0.5,0.5) {};
\node (v12) at (0,-0.5) {};
\node [scale=0.7] at (-0.5,0.3) {$\gamma a_1$};
\node [scale=0.7]at (-0.4,-0.3) {$\gamma b_1$};
\node [scale=0.7]at (0.25,0.3) {$\gamma a_m$};
\node [scale=0.7]at (0.45,-0.3) {$\gamma b_m$};
\draw  (-0.7,0.5) -- (-0.6,-0.5);
\draw  (0.5,0.5) -- (0.7,-0.5);
\node at (0,0) {$\cdots$};
\end{tikzpicture}\end{matrix} \end{matrix}$$
\caption{}
\label{6}
\end{figure}

\begin{defn}\label{eq1}
We say two PPGs are \textbf{equivalent} if they are connected by a planar isotopy.
\end{defn}
We will justify this definition in Section $4$. Such a planar isotopy in Definition \ref{eq1} is called a \textbf{deformation} of PPGs. We will call an equivalence/deformation class of PPGs shortly a \textbf{PPG-class}.

The tensor product and composition satisfy the middle-four-interchange law $(G\circ H)\otimes(G'\circ H')=(G\otimes G')\circ(H\otimes H').$
They are associative on deformation classes.

There are some special PPGs.
A PPG is called \textbf{elementary} if each of its connected components has at most one vertex which is neither a source nor a sink.

\begin{prop}\label{X}
Any PPG is equivalent to a composition of elementary ones.
\end{prop}
The PPG in Fig \ref{1} is equivalent to a composition of the three elementary PPGs in Fig \ref{7}, where for convenient we will freely omit the plane box, sources and sinks.
\begin{figure}[h]
\centering
$$
\begin{matrix}
\begin{matrix}
\begin{tikzpicture}[scale=0.5]
\node (v2) at (-4,3) {};
\node (v1) at (-1.5,5.5) {};
\node (v7) at (-1.5,1) {};
\node (v9) at (1.5,5.5) {};
\node (v14) at (2,1.5) {};
\node (v3) at (-3,7.5) {};
\node (v4) at (-2,7.5) {};
\node (v5) at (-0.5,7.5) {};
\node (v6) at (-4.8,7.4) {};
\node (v11) at (-4.5,-1) {};
\node (v12) at (-2,-1) {};
\node (v13) at (0,-1) {};
\node (v15) at (2,-1) {};
\node (v8) at (1,7.5) {};
\node (v10) at (2.5,7.5) {};
\node  at (-2.5,3.5) {};
\node  at (-3,5.2) {};
\node  at (-1.2,3.3) {};
\node  at (0.5,3.25) {};
\node  at (2.2,3.7) {};
\node  at (-3,1.7) {};
\draw[fill] (-4,3) circle [radius=0.11];
\draw[fill] (v1) circle [radius=0.11];
\draw[fill] (v7) circle [radius=0.11];
\draw[fill] (v9) circle [radius=0.11];
\draw[fill] (v14) circle [radius=0.11];
\draw  plot[smooth, tension=1] coordinates {(v1) (-2.5,5)  (-3.5,4) (v2)}[postaction={decorate, decoration={markings,mark=at position .5 with {\arrow[black]{stealth}}}}];
\draw  plot[smooth, tension=1] coordinates {(v1) (-2,4.5)  (-3,3.5) (v2)}[postaction={decorate, decoration={markings,mark=at position .5 with {\arrow[black]{stealth}}}}];

\draw  (v3) -- (-1.5,5.5)[postaction={decorate, decoration={markings,mark=at position .5 with {\arrow[black]{stealth}}}}];
\draw  (v4) -- (-1.5,5.5)[postaction={decorate, decoration={markings,mark=at position .5 with {\arrow[black]{stealth}}}}];
\draw  (v5) -- (-1.5,5.5)[postaction={decorate, decoration={markings,mark=at position .5 with {\arrow[black]{stealth}}}}];

\draw  (v6) -- (-4,3)[postaction={decorate, decoration={markings,mark=at position .5 with {\arrow[black]{stealth}}}}];
\draw  (-1.5,5.5) -- (-1.5,1)[postaction={decorate, decoration={markings,mark=at position .5 with {\arrow[black]{stealth}}}}];
\draw  (-4,3) -- (-1.5,1)[postaction={decorate, decoration={markings,mark=at position .5 with {\arrow[black]{stealth}}}}];

\draw  (v8)--(1.5,5.5)[postaction={decorate, decoration={markings,mark=at position .5 with {\arrow[black]{stealth}}}}];
\draw  (v10) -- (1.5,5.5)[postaction={decorate, decoration={markings,mark=at position .5 with {\arrow[black]{stealth}}}}];
\draw  (1.5,5.5) -- (-1.5,1)[postaction={decorate, decoration={markings,mark=at position .5 with {\arrow[black]{stealth}}}}];
\draw  (-4,3) -- (v11)[postaction={decorate, decoration={markings,mark=at position .5 with {\arrow[black]{stealth}}}}];
\draw  (-1.5,1) -- (v12)[postaction={decorate, decoration={markings,mark=at position .65 with {\arrow[black]{stealth}}}}];
\draw  (v13) -- (-1.5,1)[postaction={decorate, decoration={markings,mark=at position .5 with {\arrowreversed[black]{stealth}}}}];
\draw  (1.5,5.5) -- (2,1.5)[postaction={decorate, decoration={markings,mark=at position .5 with {\arrow[black]{stealth}}}}];
\draw  (2,1.5) -- (v15)[postaction={decorate, decoration={markings,mark=at position .5 with {\arrow[black]{stealth}}}}];

\node (v16) at (-6,4.3) {};
\node (v17) at (8,4.3) {};
\node (v18) at (-6,2.4) {};
\node (v19) at (8,2.4) {};
\draw [dotted]  (v16) edge (v17);
\draw [dotted]  (v18) edge (v19);


\node (v20) at (4.5,3.25) {};

\node (v21) at (4.5,7.5) {};

\node (v22) at (4.5,-1) {};
\draw[fill] (v20) circle [radius=0.11];
\draw  (v22) -- (4.5,3.25)[postaction={decorate, decoration={markings,mark=at position .5 with {\arrowreversed[black]{stealth}}}}];
\draw  (4.5,3.24) -- (v21)[postaction={decorate, decoration={markings,mark=at position .5 with {\arrowreversed[black]{stealth}}}}];

\node (v23) at (6.5,7.5) {};
\node (v24) at (6.5,-1) {};
\draw  (v23) -- (v24)[postaction={decorate, decoration={markings,mark=at position .5 with {\arrow[black]{stealth}}}}];
\end{tikzpicture}
\end{matrix}
&
\begin{matrix}
\begin{tikzpicture}[scale=0.5]
\node (v1) at (2.5,0) {};
\node (v2) at (-1.5,0) {};
\draw [->] (2.5,0) --(-1.5,0);
\node at (0.5,0.5) {composition};
\end{tikzpicture}
\end{matrix}
&
\begin{matrix}
\begin{matrix}
\begin{tikzpicture}[scale=0.5]
\node (v2) at (-2,0.5) {};
\node (v5) at (0.5,0.5) {};
\draw[fill] (0.5,0.5) circle [radius=0.11];
\node (v12) at (3,0.5) {};
\draw[fill] (3,0.5) circle [radius=0.11];
\node [scale=0.7](v1) at (-2,2) {};
\node (v3) at (-2,-1) {};
\node [scale=0.7] (v4) at (-0.5,2) {};
\node [scale=0.7](v6) at (0.5,2) {};
\node [scale=0.7](v7) at (1.5,2) {};
\node [scale=0.7](v8) at (-0.5,-1) {};
\node [scale=0.7](v9) at (0.5,-1) {};
\node [scale=0.7] (v10) at (1.5,-1) {};
\node [scale=0.7](v11) at (2.5,2) {};
\node [scale=0.7](v13) at (3.5,2) {};
\node [scale=0.7](v14) at (2.5,-1) {};
\node [scale=0.7] (v15) at (3.5,-1) {};
\draw  (v1) -- (v3)[postaction={decorate, decoration={markings,mark=at position .45 with {\arrow[black]{stealth}}}}];
\draw  (v4)  --  (0.5,0.5)[postaction={decorate, decoration={markings,mark=at position .45 with {\arrow[black]{stealth}}}}];
\draw  (v6)  --  (0.5,0.5)[postaction={decorate, decoration={markings,mark=at position .45 with {\arrow[black]{stealth}}}}];
\draw  (v7)  --  (0.5,0.5)[postaction={decorate, decoration={markings,mark=at position .45 with {\arrow[black]{stealth}}}}];
\draw  (0.5,0.5)  --  (v8)[postaction={decorate, decoration={markings,mark=at position .75 with {\arrow[black]{stealth}}}}];
\draw  (0.5,0.5)  --  (v9)[postaction={decorate, decoration={markings,mark=at position .75 with {\arrow[black]{stealth}}}}];
\draw  (0.5,0.5)  --  (v10)[postaction={decorate, decoration={markings,mark=at position .75 with {\arrow[black]{stealth}}}}];
\draw  (v11)  --  (3,0.5)[postaction={decorate, decoration={markings,mark=at position .45 with {\arrow[black]{stealth}}}}];
\draw  (v13)  --  (3,0.5)[postaction={decorate, decoration={markings,mark=at position .45 with {\arrow[black]{stealth}}}}];
\draw  (3,0.5)  --  (v14)[postaction={decorate, decoration={markings,mark=at position .75 with {\arrow[black]{stealth}}}}];
\draw  (3,0.5)  --  (v15)[postaction={decorate, decoration={markings,mark=at position .75 with {\arrow[black]{stealth}}}}];

\node [scale=0.7](v16) at (4.5,2) {};
\node [scale=0.7](v17) at (5.5,2) {};
\node [scale=0.7](v18) at (4.5,-1) {};
\node [scale=0.7] (v19) at (5.5,-1) {};
\draw  (v16)  --  (v18)[postaction={decorate, decoration={markings,mark=at position .45 with {\arrow[black]{stealth}}}}];
\draw  (v17)  --  (v19)[postaction={decorate, decoration={markings,mark=at position .45 with {\arrow[black]{stealth}}}}];
\end{tikzpicture}

\end{matrix}\\
\begin{matrix}
\begin{tikzpicture}[scale=0.5]
\node (v17) at (-0.5,-3.5) {};
\draw[fill] (-0.5,-3.5) circle [radius=0.11];
\node [scale=0.7](v16) at (-2,-2) {};
\node [scale=0.7](v18) at (-0.5,-2) {};
\node [scale=0.7](v19) at (0.5,-2) {};
\node (v23) at (1.5,-3.5) {};
\node (v26) at (2.5,-3.5) {};
\node (v29) at (3.5,-3.5) {};
\node [scale=0.7](v20) at (-1.5,-5) {};
\node [scale=0.7](v21) at (0.5,-5) {};
\node [scale=0.7](v22) at (1.5,-2) {};
\node (v24) at (1.5,-5) {};
\node [scale=0.7](v25) at (2.5,-2) {};
\node (v27) at (2.5,-5) {};
\node [scale=0.7](v28) at (3.5,-2) {};
\node (v30) at (3.5,-5) {};

\draw  (v16) -- (-0.5,-3.5)[postaction={decorate, decoration={markings,mark=at position .45 with {\arrow[black]{stealth}}}}];
\draw  (v18) -- (-0.5,-3.5)[postaction={decorate, decoration={markings,mark=at position .45 with {\arrow[black]{stealth}}}}];
\draw  (v19) -- (-0.5,-3.5)[postaction={decorate, decoration={markings,mark=at position .45 with {\arrow[black]{stealth}}}}];
\draw  (-0.5,-3.5) -- (v20)[postaction={decorate, decoration={markings,mark=at position .75 with {\arrow[black]{stealth}}}}];
\draw  (-0.5,-3.5) -- (v21)[postaction={decorate, decoration={markings,mark=at position .75 with {\arrow[black]{stealth}}}}];
\draw  (v22) -- (v24)[postaction={decorate, decoration={markings,mark=at position .45 with {\arrow[black]{stealth}}}}];
\draw  (v25) -- (v27)[postaction={decorate, decoration={markings,mark=at position .45 with {\arrow[black]{stealth}}}}];
\draw  (v28) -- (v30)[postaction={decorate, decoration={markings,mark=at position .45 with {\arrow[black]{stealth}}}}];

\node [scale=0.7](v66) at (4.5,-2) {};
\node [scale=0.7](v67) at (5.5,-2) {};
\node [scale=0.7](v68) at (4.5,-5) {};
\node [scale=0.7] (v69) at (5.5,-5) {};
\draw  (v66)  --  (v68)[postaction={decorate, decoration={markings,mark=at position .2 with {\arrow[black]{stealth}}}}][postaction={decorate, decoration={markings,mark=at position .9 with {\arrow[black]{stealth}}}}];
\draw  (v67)  --  (v69)[postaction={decorate, decoration={markings,mark=at position .45 with {\arrow[black]{stealth}}}}];
\draw[fill] (4.5,-3.5) circle [radius=0.11];
\end{tikzpicture}
\end{matrix}\\
\begin{matrix}
\begin{tikzpicture}[scale=0.5]

\node (v32) at (-1.5,-7.5) {};

\node [scale=0.7](v31) at (-1.5,-6) {};
\node (v33) at (-1.5,-9) {};
\node (v35) at (1.5,-7.5) {};
\draw [fill] (1.5,-7.5) circle [radius=0.11];
\node [scale=0.7] (v34) at (0.5,-6) {};
\node [scale=0.7](v36) at (1.5,-6) {};
\node [scale=0.7](v37) at (2.5,-6) {};
\node [scale=0.7](v38) at (1,-9) {};
\node [scale=0.7](v39) at (2,-9) {};
\node (v41) at (3.5,-7.5) {};
\draw [fill](3.5,-7.5) circle [radius=0.11];
\node [scale=0.7](v40) at (3.5,-6) {};
\node [scale=0.7](v42) at (3.5,-9) {};

\draw  (v31) -- (v33)[postaction={decorate, decoration={markings,mark=at position .45 with {\arrow[black]{stealth}}}}];
\draw  (v34) -- (1.5,-7.5)[postaction={decorate, decoration={markings,mark=at position .45 with {\arrow[black]{stealth}}}}];
\draw  (v36) -- (1.5,-7.5)[postaction={decorate, decoration={markings,mark=at position .45 with {\arrow[black]{stealth}}}}];
\draw  (v37) -- (1.5,-7.5)[postaction={decorate, decoration={markings,mark=at position .45 with {\arrow[black]{stealth}}}}];
\draw  (1.5,-7.5) -- (v38)[postaction={decorate, decoration={markings,mark=at position .75 with {\arrow[black]{stealth}}}}];
\draw  (1.5,-7.5) -- (v39)[postaction={decorate, decoration={markings,mark=at position .75 with {\arrow[black]{stealth}}}}];
\draw  (v40) -- (3.5,-7.5)[postaction={decorate, decoration={markings,mark=at position .45 with {\arrow[black]{stealth}}}}];
\draw  (3.5,-7.5) -- (v42)[postaction={decorate, decoration={markings,mark=at position .75 with {\arrow[black]{stealth}}}}];

\node [scale=0.7](v16) at (4.5,-6) {};
\node [scale=0.7](v17) at (5.5,-6) {};
\node [scale=0.7](v18) at (4.5,-9) {};
\node [scale=0.7] (v19) at (5.5,-9) {};
\draw  (v16)  --  (v18)[postaction={decorate, decoration={markings,mark=at position .45 with {\arrow[black]{stealth}}}}];
\draw  (v17)  --  (v19)[postaction={decorate, decoration={markings,mark=at position .45 with {\arrow[black]{stealth}}}}];
\end{tikzpicture}
\end{matrix}
\end{matrix}
\end{matrix}
$$
\caption{ Three elementary PPGs and their composition.}
\label{7}
\end{figure}
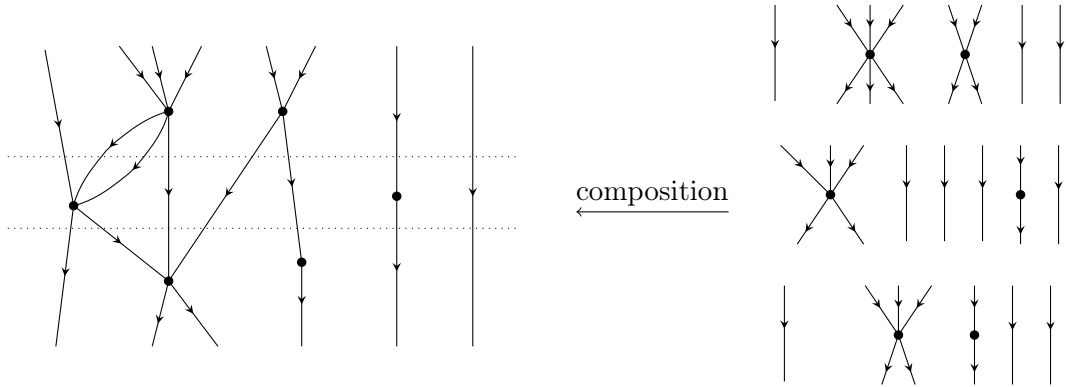

A PPG is called \textbf{prime} if it is connected and has exactly one vertex which is neither a source nor a sink, see the left of Fig \ref{8}. A PPG is called \textbf{unitary} if it has exactly one edge, see the middle of Fig \ref{8}.

\begin{figure}[htbp]
\centering
$$
\begin{matrix}
\begin{matrix}
\begin{tikzpicture}[scale=0.32]
\draw [dashed] (-4,3) rectangle (2.5,-2);
\node (v2) at (-0.5,0.5) {};
\draw[fill] (-0.5,0.5) circle [radius=0.15];
\node (v1) at (-3,3) {};
\node (v3) at (-2,3) {};
\node (v4) at (1.5,3) {};
\node (v5) at (-3,-2) {};
\node (v6) at (-2,-2) {};
\node (v7) at (2,-2) {};
\draw[fill] (-3.2,3) circle [radius=0.15];\draw[fill] (-2,3) circle [radius=0.15];\draw[fill] (1.5,3) circle [radius=0.15];
\draw  (-3.2,3) -- (-0.5,0.5)[postaction={decorate, decoration={markings,mark=at position .5 with {\arrow[black]{stealth}}}}];
\draw  (-2,3) -- (-0.5,0.5)[postaction={decorate, decoration={markings,mark=at position .5 with {\arrow[black]{stealth}}}}];
\draw (1.5,3) -- (-0.5,0.5)[postaction={decorate, decoration={markings,mark=at position .5 with {\arrow[black]{stealth}}}}];
\draw  (-0.5,0.5) -- (-3,-2)[postaction={decorate, decoration={markings,mark=at position .5 with {\arrow[black]{stealth}}}}];
\draw  (-0.5,0.5) -- (-1.8,-2)[postaction={decorate, decoration={markings,mark=at position .5 with {\arrow[black]{stealth}}}}];
\draw  (-0.5,0.5) -- (2,-2)[postaction={decorate, decoration={markings,mark=at position .5 with {\arrow[black]{stealth}}}}];
\draw[fill] (-3,-2) circle [radius=0.15];\draw[fill] (-1.8,-2) circle [radius=0.15];\draw[fill] (2,-2) circle [radius=0.15];
\node at (-0.2,2.4) {$\cdots$};
\node at (-0.2,-1.4) {$\cdots$};
\end{tikzpicture}
\end{matrix}&&&&
\begin{matrix}
\begin{tikzpicture}[scale=0.35]
\draw [dashed] (-2,3) rectangle (1,-1.5);
\node (v1) at (-0.5,3) {};
\node (v2) at (-0.5,-1.5) {};
\draw  (-0.5,3) -- (-0.5,-1.5)[postaction={decorate, decoration={markings,mark=at position .5 with {\arrow[black]{stealth}}}}];
\draw[fill] (-0.5,3) circle [radius=0.15];
\draw[fill] (-0.5,-1.5) circle [radius=0.15];
\end{tikzpicture}
\end{matrix}&&&&\begin{matrix}
\begin{tikzpicture}[scale=0.35]

\draw [dashed] (-2,3) rectangle (6.5,-1.5);

\node (v1) at (-1,3) {};
\node (v2) at (-1,-1.5) {};
\node (v3) at (0.5,3) {};
\node (v4) at (0.5,-1.5) {};
\node (v5) at (5,3) {};
\node (v6) at (5,-1.5) {};
\node at (3,0.8) {$\cdots$};

\draw  (-1,3) -- (-1,-1.5)[postaction={decorate, decoration={markings,mark=at position .5 with {\arrow[black]{stealth}}}}];
\draw  (0.5,3) -- (0.5,-1.5)[postaction={decorate, decoration={markings,mark=at position .5 with {\arrow[black]{stealth}}}}];
\draw  (5,3) -- (5,-1.5)[postaction={decorate, decoration={markings,mark=at position .5 with {\arrow[black]{stealth}}}}];
\draw[fill] (-1,3) circle [radius=0.15];
\draw[fill] (-1,-1.5) circle [radius=0.15];
\draw[fill] (0.5,-1.5) circle [radius=0.15];
\draw[fill] (0.5,3) circle [radius=0.15];
\draw[fill] (5,-1.5) circle [radius=0.15];
\draw[fill] (5,3) circle [radius=0.15];
\end{tikzpicture}
\end{matrix}
\end{matrix}
$$
\caption{}
\label{8}
\end{figure}

\begin{prop}\label{Y}
Any elementary PPG is equivalent to a tensor product of prime and unitary ones, which is unique up to equivalence.
\end{prop}

A PPG is called \textbf{invertible} if  each of its connected components has exactly one edge, see the right of Fig \ref{8}. Any invertible PPG is equivalent to a tensor product of unitary ones.

The meaning of a prime (unitary, invertible, elementary) PPG-class is clear.

As a summarization, we provides the following result, which is easy to check.
\begin{thm}\label{ppg}
There is a monoidal category $\mathbf{PPG}$ with non-negative integers as objects and with PPG-classes and the empty graph $\bigcirc$ as morphisms. For $m,n\geq 1$, a morphism from $m$ to $n$ is a PPG-class with $m$ sources and $n$ sinks. On objects the tensor product is given by the addition of integers, and on morphisms the tensor product is the tensor product of PPG-classes and $\bigcirc$. The composition is the composition of PPG-classes and $\bigcirc$.  The unit object is $0$, whose identity morphism is $\bigcirc$. The identity morphism of a non-unit object $n$ is the invertible PPG-class with $n$ connected components.
\end{thm}

There is a tensor scheme $\mathbf{PRM}$ with morphisms being prime PPG-classes, with only one object and with source and target maps given by the numbers of sources and sinks, respectively. According to our definition of a PPG (Definition \ref{progre}), $\mathbf{PRM}$ is actually a semi-tensor scheme (see Definition \ref{T}). Applying Joyal and Street's construction of free monoidal category on a tensor scheme (Theorem $1.2$ in \cite{[JS91]}), we get the following result.

\begin{thm}\label{free1}
The monoidal category $\mathbf{PPG}$ is free on the tensor scheme $\mathbf{PRM}$.
\end{thm}

\begin{rem}\label{iso}
Note that in Theorem \ref{ppg}, we trivially assume $\bigcirc\otimes \bigcirc=\bigcirc$, $\bigcirc\circ \bigcirc=\bigcirc$ and for any PPG-class $G$, $G\otimes \bigcirc=\bigcirc\otimes G=G$. Also note that,  the unit object $0$ is \textbf{isolated} in $\mathbf{PPG}$, that is, for any $n\geq1$ there is no morphism from $n$ to $0$ or from $0$ to $n$, and $\bigcirc$ is the unique morphism from $0$ to $0$.
 \end{rem}

\section{A combinatorial framework}
In this section, we show a combinatorial  framework for the graphical calculus. We begin by fixing some terminologies.
\begin{defn}\label{pro}
A \textbf{processive graph} is a non-empty acyclic directed graph with all sources and sinks being of degree one.
\end{defn}
The underlying graph in Fig \ref{1} is a processive graph. Clearly, a processive graph has no isolated vertex, at least one source and at least one sink. A processive graph is a special \textbf{progressive graph} introduced by Joyal and Street \cite{[JS91]}, which is exactly an acyclic directed graph possibly with isolated vertices, see the underlying graph in Fig \ref{25} for an example.

For a processive graph, a vertex of degree one is called a \textbf{boundary vertex}, otherwise it is called an \textbf{internal vertex}. An isolated vertex is an internal vertex. A vertex of a processive graph is a boundary vertex if and only if it is a source or a sink. A vertex is called \textbf{processive} if it is neither a source nor a sink.
A vertex of a processive graph is processive if and only if it is an internal vertex.

As previous, a processive graph is called \textbf{elementary} if each of its connected components has at most one processive vertex, and is called \textbf{prime} if it is connected and has exactly one processive vertex, and is called \textbf{unitary} if it has exactly one edge, and is called \textbf{invertible} if it has no processive vertex.

The following is the key notion in this framework, which serves as a combinatorial counterpart of a PPG-class.

\begin{defn}\label{pop-graph}
A \textbf{planarly ordered processive graph} or \textbf{POP-graph}, is a processive graph $G$ equipped with a linear order $\prec$ on its edge set $E(G)$ such that

$(P1)$ $e_1\rightarrow e_2$ implies $e_1\prec e_2$;

$(P2)$  if $e_1\prec e_2 \prec e_3$ and $e_1\rightarrow e_3$, then either $e_1\rightarrow e_2$ or $e_2\rightarrow e_3$,

where $e_1\rightarrow e_2$ denotes that there is a directed path starting from $e_1$, ending with $e_2$.
\end{defn}
We write the POP-graph as $(G,\prec)$ and call the order $\prec$ a \textbf{planar order} of $G$.

\begin{rem}
The notion of a planar order can also be defined for a general poset $(X,\leq)$,  and in this case, this notion coincides with the notion of a \textbf{nonseparating linear extension} of $(X,\leq)$ introduced in \cite{[BFR71]}.
\end{rem}

A POP-graph is called \textbf{elementary} (\textbf{prime}, \textbf{unitary}, \textbf{invertible}) if the underlying processive graph is elementary (prime, unitary, invertible).

\begin{ex}\label{ex1} Fig \ref{10} shows three examples of elementary POP-graphs, where $(P1)$ and $(P2)$ are easy to check.  However, there is a natural convention for drawing POP-graphs, especially those elementary ones, which is from left to right and from up to down according to the planar order. Fig \ref{10} shows the convention.  Actually, our definition of a planar order is motivated by these examples.

\begin{figure}[htbp]
\centering
$$
\begin{matrix}
\begin{matrix}
\begin{tikzpicture}[scale=0.5]
\node (v2) at (-2,0.5) {};
\node (v5) at (0.5,0.5) {};
\draw[fill] (0.5,0.5) circle [radius=0.11];
\node (v12) at (3,0.5) {};
\draw[fill] (3,0.5) circle [radius=0.11];
\node [scale=0.7](v1) at (-2,2) {$1$};
\node (v3) at (-2,-1) {};
\node [scale=0.7] (v4) at (-0.5,2) {$2$};
\node [scale=0.7](v6) at (0.5,2) {$3$};
\node [scale=0.7](v7) at (1.5,2) {$4$};
\node [scale=0.7](v8) at (-0.5,-1) {$5$};
\node [scale=0.7](v9) at (0.5,-1) {$6$};
\node [scale=0.7] (v10) at (1.5,-1) {$7$};
\node [scale=0.7](v11) at (2.5,2) {$8$};
\node [scale=0.7](v13) at (3.5,2) {$9$};
\node [scale=0.7](v14) at (2.5,-1) {$10$};
\node [scale=0.7] (v15) at (3.5,-1) {$11$};
\draw  (v1) -- (v3)[postaction={decorate, decoration={markings,mark=at position .45 with {\arrow[black]{stealth}}}}];
\draw  (v4)  --  (0.5,0.5)[postaction={decorate, decoration={markings,mark=at position .45 with {\arrow[black]{stealth}}}}];
\draw  (v6)  --  (0.5,0.5)[postaction={decorate, decoration={markings,mark=at position .45 with {\arrow[black]{stealth}}}}];
\draw  (v7)  --  (0.5,0.5)[postaction={decorate, decoration={markings,mark=at position .45 with {\arrow[black]{stealth}}}}];
\draw  (0.5,0.5)  --  (v8)[postaction={decorate, decoration={markings,mark=at position .75 with {\arrow[black]{stealth}}}}];
\draw  (0.5,0.5)  --  (v9)[postaction={decorate, decoration={markings,mark=at position .75 with {\arrow[black]{stealth}}}}];
\draw  (0.5,0.5)  --  (v10)[postaction={decorate, decoration={markings,mark=at position .75 with {\arrow[black]{stealth}}}}];
\draw  (v11)  --  (3,0.5)[postaction={decorate, decoration={markings,mark=at position .45 with {\arrow[black]{stealth}}}}];
\draw  (v13)  --  (3,0.5)[postaction={decorate, decoration={markings,mark=at position .45 with {\arrow[black]{stealth}}}}];
\draw  (3,0.5)  --  (v14)[postaction={decorate, decoration={markings,mark=at position .75 with {\arrow[black]{stealth}}}}];
\draw  (3,0.5)  --  (v15)[postaction={decorate, decoration={markings,mark=at position .75 with {\arrow[black]{stealth}}}}];

\node [scale=0.7](v16) at (4.5,2) {$12$};
\node [scale=0.7](v17) at (5.5,2) {$13$};
\node [scale=0.7](v18) at (4.5,-1) {};
\node [scale=0.7] (v19) at (5.5,-1) {};
\draw  (v16)  --  (v18)[postaction={decorate, decoration={markings,mark=at position .45 with {\arrow[black]{stealth}}}}];
\draw  (v17)  --  (v19)[postaction={decorate, decoration={markings,mark=at position .45 with {\arrow[black]{stealth}}}}];
\end{tikzpicture}

\end{matrix}\\
\begin{matrix}
\begin{tikzpicture}[scale=0.5]
\node (v17) at (-0.5,-3.5) {};
\draw[fill] (-0.5,-3.5) circle [radius=0.11];
\node [scale=0.7](v16) at (-2,-2) {$1$};
\node [scale=0.7](v18) at (-0.5,-2) {$2$};
\node [scale=0.7](v19) at (0.5,-2) {$3$};
\node (v23) at (1.5,-3.5) {};
\node (v26) at (2.5,-3.5) {};
\node (v29) at (3.5,-3.5) {};
\node [scale=0.7](v20) at (-1.5,-5) {$4$};
\node [scale=0.7](v21) at (0.5,-5) {$5$};
\node [scale=0.7](v22) at (1.5,-2) {$6$};
\node (v24) at (1.5,-5) {};
\node [scale=0.7](v25) at (2.5,-2) {$7$};
\node (v27) at (2.5,-5) {};
\node [scale=0.7](v28) at (3.5,-2) {$8$};
\node (v30) at (3.5,-5) {};

\draw  (v16) -- (-0.5,-3.5)[postaction={decorate, decoration={markings,mark=at position .45 with {\arrow[black]{stealth}}}}];
\draw  (v18) -- (-0.5,-3.5)[postaction={decorate, decoration={markings,mark=at position .45 with {\arrow[black]{stealth}}}}];
\draw  (v19) -- (-0.5,-3.5)[postaction={decorate, decoration={markings,mark=at position .45 with {\arrow[black]{stealth}}}}];
\draw  (-0.5,-3.5) -- (v20)[postaction={decorate, decoration={markings,mark=at position .75 with {\arrow[black]{stealth}}}}];
\draw  (-0.5,-3.5) -- (v21)[postaction={decorate, decoration={markings,mark=at position .75 with {\arrow[black]{stealth}}}}];
\draw  (v22) -- (v24)[postaction={decorate, decoration={markings,mark=at position .45 with {\arrow[black]{stealth}}}}];
\draw  (v25) -- (v27)[postaction={decorate, decoration={markings,mark=at position .45 with {\arrow[black]{stealth}}}}];
\draw  (v28) -- (v30)[postaction={decorate, decoration={markings,mark=at position .45 with {\arrow[black]{stealth}}}}];

\node [scale=0.7](v66) at (4.5,-2) {$9$};
\node [scale=0.7](v67) at (5.5,-2) {$11$};
\node [scale=0.7](v68) at (4.5,-5) {$10$};
\node [scale=0.7] (v69) at (5.5,-5) {};
\draw  (v66)  --  (v68)[postaction={decorate, decoration={markings,mark=at position .2 with {\arrow[black]{stealth}}}}][postaction={decorate, decoration={markings,mark=at position .9 with {\arrow[black]{stealth}}}}];
\draw  (v67)  --  (v69)[postaction={decorate, decoration={markings,mark=at position .45 with {\arrow[black]{stealth}}}}];
\draw[fill] (4.5,-3.5) circle [radius=0.11];
\end{tikzpicture}
\end{matrix}\\
\begin{matrix}
\begin{tikzpicture}[scale=0.5]

\node (v32) at (-1.5,-7.5) {};

\node [scale=0.7](v31) at (-1.5,-6) {$1$};
\node (v33) at (-1.5,-9) {};
\node (v35) at (1.5,-7.5) {};
\draw [fill] (1.5,-7.5) circle [radius=0.11];
\node [scale=0.7] (v34) at (0.5,-6) {$2$};
\node [scale=0.7](v36) at (1.5,-6) {$3$};
\node [scale=0.7](v37) at (2.5,-6) {$4$};
\node [scale=0.7](v38) at (1,-9) {$5$};
\node [scale=0.7](v39) at (2,-9) {$6$};
\node (v41) at (3.5,-7.5) {};
\draw [fill](3.5,-7.5) circle [radius=0.11];
\node [scale=0.7](v40) at (3.5,-6) {$7$};
\node [scale=0.7](v42) at (3.5,-9) {$8$};

\draw  (v31) -- (v33)[postaction={decorate, decoration={markings,mark=at position .45 with {\arrow[black]{stealth}}}}];
\draw  (v34) -- (1.5,-7.5)[postaction={decorate, decoration={markings,mark=at position .45 with {\arrow[black]{stealth}}}}];
\draw  (v36) -- (1.5,-7.5)[postaction={decorate, decoration={markings,mark=at position .45 with {\arrow[black]{stealth}}}}];
\draw  (v37) -- (1.5,-7.5)[postaction={decorate, decoration={markings,mark=at position .45 with {\arrow[black]{stealth}}}}];
\draw  (1.5,-7.5) -- (v38)[postaction={decorate, decoration={markings,mark=at position .75 with {\arrow[black]{stealth}}}}];
\draw  (1.5,-7.5) -- (v39)[postaction={decorate, decoration={markings,mark=at position .75 with {\arrow[black]{stealth}}}}];
\draw  (v40) -- (3.5,-7.5)[postaction={decorate, decoration={markings,mark=at position .45 with {\arrow[black]{stealth}}}}];
\draw  (3.5,-7.5) -- (v42)[postaction={decorate, decoration={markings,mark=at position .75 with {\arrow[black]{stealth}}}}];

\node [scale=0.7](v16) at (4.5,-6) {$9$};
\node [scale=0.7](v17) at (5.5,-6) {$10$};
\node [scale=0.7](v18) at (4.5,-9) {};
\node [scale=0.7] (v19) at (5.5,-9) {};
\draw  (v16)  --  (v18)[postaction={decorate, decoration={markings,mark=at position .45 with {\arrow[black]{stealth}}}}];
\draw  (v17)  --  (v19)[postaction={decorate, decoration={markings,mark=at position .45 with {\arrow[black]{stealth}}}}];
\end{tikzpicture}
\end{matrix}
\end{matrix}
$$

\caption{}
\label{10}
\end{figure}

\end{ex}

In this paper, when drawing POP-graphs, we will always use the convention shown in Fig \ref{10}. A graphical explanation of $(P1)$ and $(P2)$ under this convention is shown in Fig \ref{11}.

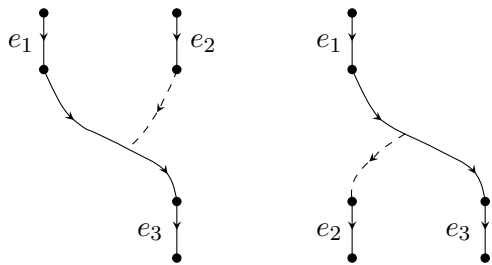
\begin{figure}[H]
\centering
$$
\begin{matrix}
\begin{matrix}
\begin{tikzpicture}[scale=0.5]
\node (v1) at (-1,3) {};
\draw[fill] (-1,3)  circle [radius=0.11];
\node (v2) at (-1,1.5) {};
\draw[fill] (-1,1.5) circle [radius=0.11];
\node (v3) at (2.5,-2) {};
\draw[fill] (2.5,-2) circle [radius=0.11];
\node (v4) at (2.5,-3.5) {};
\draw[fill] (2.5,-3.5) circle [radius=0.11];
\draw (-1,3) --(-1,1.5)[postaction={decorate, decoration={markings,mark=at position 0.5 with {\arrow[black]{stealth}}}}];
\draw  (2.5,-2) -- (2.5,-3.5)[postaction={decorate, decoration={markings,mark=at position 0.5 with {\arrow[black]{stealth}}}}];
\draw  plot[smooth, tension=.7] coordinates {(v2) (-0.5,0.5) (0,0)(0.4,-0.2) (1.2,-0.6) (2.2,-1.2) (v3)}[postaction={decorate, decoration={markings,mark=at position 0.3 with {\arrow[black]{stealth}}}}][postaction={decorate, decoration={markings,mark=at position 0.85 with {\arrow[black]{stealth}}}}];
\node at (-1.6,2.2) {};
\node at (1.8,-2.8) {};
\node (v5) at (2.5,3) {};
\draw[fill] (2.5,3) circle [radius=0.11];
\node (v6) at (2.5,1.5) {};
\draw[fill] (2.5,1.5) circle [radius=0.11];

\draw  (2.5,3) --(2.5,1.5)[postaction={decorate, decoration={markings,mark=at position 0.5 with {\arrow[black]{stealth}}}}];
\draw [dashed]plot[smooth, tension=.7] coordinates {(v6) (2.4,1.2) (2.2,0.8) (2,0.4) (1.6,-0.2) (1.2,-0.6)}[postaction={decorate, decoration={markings,mark=at position 0.5 with {\arrow[black]{stealth}}}}];
\node at (3.2,2.2) {$e_2$};
\node at (-1.6,2.2) {$e_1$};
\node at (1.8,-2.8) {$e_3$};
\end{tikzpicture}
\end{matrix}&&&
\begin{matrix}
\begin{tikzpicture}[scale=0.5]
\node (v1) at (-1,3) {};
\draw[fill] (-1,3)  circle [radius=0.11];
\node (v2) at (-1,1.5) {};
\draw[fill] (-1,1.5) circle [radius=0.11];
\node (v3) at (2.5,-2) {};
\draw[fill] (2.5,-2) circle [radius=0.11];
\node (v4) at (2.5,-3.5) {};
\draw[fill] (2.5,-3.5) circle [radius=0.11];
\draw (-1,3) --(-1,1.5)[postaction={decorate, decoration={markings,mark=at position 0.5 with {\arrow[black]{stealth}}}}];
\draw  (2.5,-2) -- (2.5,-3.5)[postaction={decorate, decoration={markings,mark=at position 0.5 with {\arrow[black]{stealth}}}}];
\draw  plot[smooth, tension=.7] coordinates {(v2) (-0.5,0.5) (0,0) (1.2,-0.6) (2.2,-1.2) (v3)}[postaction={decorate, decoration={markings,mark=at position 0.3 with {\arrow[black]{stealth}}}}][postaction={decorate, decoration={markings,mark=at position 0.85 with {\arrow[black]{stealth}}}}];
\node at (-1.6,2.2) {};
\node at (1.8,-2.8) {};
\node at (-1.6,2.2) {$e_1$};
\node at (1.8,-2.8) {$e_3$};
\node (v7) at (-1,-2) {};
\draw[fill] (-1,-2) circle [radius=0.11];
\node (v8) at (-1,-3.5) {};
\draw[fill] (-1,-3.5) circle [radius=0.11];
\draw  (-1,-2) -- (-1,-3.5)[postaction={decorate, decoration={markings,mark=at position 0.5 with {\arrow[black]{stealth}}}}];
\draw  [dashed]plot[smooth, tension=.7] coordinates {(0.4,-0.2) (-0.2,-0.6) (-0.4,-0.8) (-0.8,-1.2) (-1,-1.6) (v7)}[postaction={decorate, decoration={markings,mark=at position 0.5 with {\arrow[black]{stealth}}}}];
\node at (-1.6,-2.8) {$e_2$};
\end{tikzpicture}
\end{matrix}
\end{matrix}
$$
\caption{ If $e_1\prec e_2 \prec e_3$ and $e_1\rightarrow e_3$, then $e_2$ has only two possible drawings.}
\label{11}
\end{figure}

\begin{ex}\label{ex2}
Fig \ref{12} shows an example of non-elementary POP-graph, which is actually a composition of the elementary POP-graphs in Fig \ref{10}.

\begin{figure}[htbp]
\centering
\begin{tikzpicture}[scale=0.4]

\node (v2) at (-4,3) {};
\node (v1) at (-1.5,5.5) {};
\node (v7) at (-1.5,1) {};
\node (v9) at (1.5,5.5) {};
\node (v14) at (2,1.5) {};
\node [scale=0.7](v3) at (-3,7.5) {};
\node [scale=0.7](v4) at (-2,7.5) {};
\node [scale=0.7](v5) at (-0.5,7.5) {};
\node[scale=0.7] (v6) at (-4.8,7.4) {};
\node [scale=0.7](v11) at (-4.5,-1) {};
\node [scale=0.7](v12) at (-2,-1) {};
\node [scale=0.7](v13) at (0,-1) {};
\node [scale=0.7](v15) at (2,-1) {};
\node [scale=0.7](v8) at (1,7.5) {};
\node [scale=0.7](v10) at (2.5,7.5) {};
\node [scale=0.7] at (-2.5,3.5) {$6$};
\node[scale=0.7]  at (-3,5.2) {$5$};
\node [scale=0.7] at (-1.2,3.3) {$9$};
\node [scale=0.7] at (0.5,3.25) {$12$};
\node [scale=0.7] at (2.2,3.7) {$15$};
\node [scale=0.7] at (-3,1.7) {$8$};
\draw[fill] (-4,3) circle [radius=0.11];
\draw[fill] (v1) circle [radius=0.11];
\draw[fill] (v7) circle [radius=0.11];
\draw[fill] (v9) circle [radius=0.11];
\draw[fill] (v14) circle [radius=0.11];
\draw  plot[smooth, tension=1] coordinates {(v1) (-2.5,5)  (-3.5,4) (v2)}[postaction={decorate, decoration={markings,mark=at position .5 with {\arrow[black]{stealth}}}}];
\draw  plot[smooth, tension=1] coordinates {(v1) (-2,4.5)  (-3,3.5) (v2)}[postaction={decorate, decoration={markings,mark=at position .5 with {\arrow[black]{stealth}}}}];

\draw  (v3) -- (-1.5,5.5)[postaction={decorate, decoration={markings,mark=at position .5 with {\arrow[black]{stealth}}}}];
\draw  (v4) -- (-1.5,5.5)[postaction={decorate, decoration={markings,mark=at position .5 with {\arrow[black]{stealth}}}}];
\draw  (v5) -- (-1.5,5.5)[postaction={decorate, decoration={markings,mark=at position .5 with {\arrow[black]{stealth}}}}];

\draw  (v6) -- (-4,3)[postaction={decorate, decoration={markings,mark=at position .5 with {\arrow[black]{stealth}}}}];
\draw  (-1.5,5.5) -- (-1.5,1)[postaction={decorate, decoration={markings,mark=at position .5 with {\arrow[black]{stealth}}}}];
\draw  (-4,3) -- (-1.5,1)[postaction={decorate, decoration={markings,mark=at position .5 with {\arrow[black]{stealth}}}}];

\draw  (v8)--(1.5,5.5)[postaction={decorate, decoration={markings,mark=at position .5 with {\arrow[black]{stealth}}}}];
\draw  (v10) -- (1.5,5.5)[postaction={decorate, decoration={markings,mark=at position .5 with {\arrow[black]{stealth}}}}];
\draw  (1.5,5.5) -- (-1.5,1)[postaction={decorate, decoration={markings,mark=at position .5 with {\arrow[black]{stealth}}}}];
\draw  (-4,3) -- (v11)[postaction={decorate, decoration={markings,mark=at position .5 with {\arrow[black]{stealth}}}}];
\draw  (-1.5,1) -- (v12)[postaction={decorate, decoration={markings,mark=at position .65 with {\arrow[black]{stealth}}}}];
\draw  (v13) -- (-1.5,1)[postaction={decorate, decoration={markings,mark=at position .5 with {\arrowreversed[black]{stealth}}}}];
\draw  (1.5,5.5) -- (2,1.5)[postaction={decorate, decoration={markings,mark=at position .5 with {\arrow[black]{stealth}}}}];
\draw  (2,1.5) -- (v15)[postaction={decorate, decoration={markings,mark=at position .5 with {\arrow[black]{stealth}}}}];

\node (v20) at (4.5,3.25) {};

\node[scale=0.7] (v21) at (4.5,7.5) {};

\node [scale=0.7](v22) at (4.5,-1) {};
\draw[fill] (v20) circle [radius=0.11];
\draw  (v22) -- (4.5,3.25)[postaction={decorate, decoration={markings,mark=at position .5 with {\arrowreversed[black]{stealth}}}}];
\draw  (4.5,3.24) -- (v21)[postaction={decorate, decoration={markings,mark=at position .5 with {\arrowreversed[black]{stealth}}}}];

\node [scale=0.7](v23) at (6.5,7.5) {};
\node (v24) at (6.5,-1) {};
\draw  (v23) -- (v24)[postaction={decorate, decoration={markings,mark=at position .5 with {\arrow[black]{stealth}}}}];

\node [scale=0.7]at (-5,6.5) {$1$};
\node [scale=0.7]at (-3,6.5) {$2$};
\node [scale=0.7]at (-1.5,7) {$3$};
\node [scale=0.7]at (-0.5,6.5) {$4$};
\node [scale=0.7]at (0.6,6.5) {$10$};
\node [scale=0.7]at (2.6,6.5) {$11$};
\node [scale=0.7]at (5,6) {$17$};
\node [scale=0.7]at (-4.7,0.5) {$7$};
\node [scale=0.7]at (-2.5,0) {$13$};
\node [scale=0.7]at (0,0) {$14$};
\node [scale=0.7]at (2.5,0) {$16$};
\node [scale=0.7]at (5,0.5) {$18$};
\node [right][scale=0.7] at (6.5,3.5) {$19$};
\end{tikzpicture}
\caption{}
\label{12}
\end{figure}

\end{ex}
The notion of an isomorphism of two POP-graphs is clear. Since planar orders are linear orders, two POP-graphs have at most one isomorphism, therefore we do not bother to say an isomorphic class of POP-graphs. When there is an isomorphism between $(G_1,\prec_1)$ and $(G_2,\prec_2)$, we write $(G_1,\prec_1)=(G_2,\prec_2)$.

To introduce tensor product and composition for POP-graphs, we need some notations. For a finite set $S$ with a linear order $\prec$, if $S=\sqcup_{i=1}^n E_i$ and $\max(E_i)\prec \min(E_j)$ for any $i<j$, then we write  $ \prec= E_1\triangleleft E_2\triangleleft\cdots\triangleleft E_n.$  In this case, each $E_i$ $(1\leq i\leq n)$ is an \textbf{interval} of $(S,\prec)$, which is of the form $[a, b]=\{s\in S\mid a\preceq s\preceq b\}$ for some $a, b\in S$. Similarly, we use the notations $(a,b)$, $(a,b]$, and $[a,b)$ as usual.
We also write $1=\min(S)$ and $+\infty=\max(S)$.

Given two POP-graphs $(G_1,\prec_1)$ and $(G_2,\prec_2)$, their \textbf{tensor product} is defined as the POP-graph $(G_1\sqcup G_2, \prec_1 \triangleleft\prec_2)$, where all edges of $G_1$ are smaller than edges of $G_2$. It is easy to see that the tensor product is associative.
\begin{ex} Fig \ref{13}  shows a tensor product of two prime POP-graphs.

\begin{figure}[htbp]
\centering
\begin{tikzpicture}[scale=0.9]
\node[scale=0.8] (v1) at (-4,1.5) {$1$};
\node[scale=0.8] (v3) at (-3,1.5) {$2$};
\node (v2) at (-3.5,0.5) {};
\draw[fill] (v2) circle [radius=0.07];
\node[scale=0.8] (v4) at (-4,-0.5) {$3$};
\node [scale=0.8](v5) at (-3,-0.5) {$4$};
\node at (-2,0.5) {$\otimes$};
\node [scale=0.8](v6) at (-1,1.5) {$1$};
\node [scale=0.8](v8) at (0,1.5) {$2$};
\node (v7) at (-0.5,0.5) {};
\draw[fill] (v7) circle [radius=0.07];
\node [scale=0.8](v9) at (-0.5,-0.5) {$3$};
\node at (1,0.5) {$=$};
\node[scale=0.8] (v10) at (2,1.5) {$1$};
\node [scale=0.8](v12) at (3,1.5) {$2$};
\node (v11) at (2.5,0.5) {};
\draw[fill] (v11) circle [radius=0.07];
\node [scale=0.8](v13) at (2,-0.5) {$3$};
\node [scale=0.8](v14) at (3,-0.5) {$4$};
\node [scale=0.8](v15) at (4,1.5) {$5$};
\node [scale=0.8](v17) at (5,1.5) {$6$};
\node (v16) at (4.5,0.5) {};
\draw[fill] (v16) circle [radius=0.07];
\node [scale=0.8](v18) at (4.5,-0.5) {$7$};
\draw  (v1) -- (-3.5,0.5)[postaction={decorate, decoration={markings,mark=at position .5 with {\arrow[black]{stealth}}}}];
\draw  (v3) -- (-3.5,0.5)[postaction={decorate, decoration={markings,mark=at position .5 with {\arrow[black]{stealth}}}}];
\draw  (-3.5,0.5) -- (v4)[postaction={decorate, decoration={markings,mark=at position .65 with {\arrow[black]{stealth}}}}];
\draw  (-3.5,0.5)-- (v5)[postaction={decorate, decoration={markings,mark=at position .65 with {\arrow[black]{stealth}}}}];
\draw  (v6) -- (-0.5,0.5)[postaction={decorate, decoration={markings,mark=at position .5 with {\arrow[black]{stealth}}}}];
\draw  (v8) -- (-0.5,0.5)[postaction={decorate, decoration={markings,mark=at position .5 with {\arrow[black]{stealth}}}}];
\draw  (-0.5,0.5) -- (v9)[postaction={decorate, decoration={markings,mark=at position .65 with {\arrow[black]{stealth}}}}];
\draw  (v10) -- (2.5,0.5)[postaction={decorate, decoration={markings,mark=at position .5 with {\arrow[black]{stealth}}}}];
\draw  (v12)-- (2.5,0.5)[postaction={decorate, decoration={markings,mark=at position .5 with {\arrow[black]{stealth}}}}];
\draw  (2.5,0.5) -- (v13)[postaction={decorate, decoration={markings,mark=at position .65 with {\arrow[black]{stealth}}}}];
\draw  (2.5,0.5)-- (v14)[postaction={decorate, decoration={markings,mark=at position .65 with {\arrow[black]{stealth}}}}];
\draw  (v15) -- (4.5,0.5)[postaction={decorate, decoration={markings,mark=at position .5 with {\arrow[black]{stealth}}}}];
\draw  (v17) -- (4.5,0.5)[postaction={decorate, decoration={markings,mark=at position .5 with {\arrow[black]{stealth}}}}];
\draw  (4.5,0.5) -- (v18)[postaction={decorate, decoration={markings,mark=at position .65 with {\arrow[black]{stealth}}}}];

\end{tikzpicture}
\caption{}
\label{13}
\end{figure}
\end{ex}

An edge of a processive graph is called an \textbf{input edge} if it starts from a boundary vertex (or a source), and an \textbf{output edge} if it ends with a boundary vertex (or a sink).

Given a POP-graph $(G_1,\prec_1)$ with output edges $o_1\prec_1\cdots\prec_1 o_n$, then $E(G_1)$ can be represented as $$Q_1\triangleleft\{o_1\}\triangleleft...\triangleleft Q_k\triangleleft\{o_k\} \triangleleft...\triangleleft Q_n\triangleleft\{o_n\},$$  where $Q_1 = [1, o_1)$ and $Q_k=(o_{k-1}, o_{k})$ for $2\le k\le n$ are called \textbf{basic intervals with respect to output edges}.

Given a POP-graph $(G_2,\prec_2)$ with input edges $i_1\prec_2\cdots\prec_2 i_n$, then  $E(G_2)$ can be represented as  $$\{i_1\}\triangleleft P_1\triangleleft...\triangleleft \{i_k\}\triangleleft P_k \triangleleft...\triangleleft\{i_n\}\triangleleft P_n,$$ where $P_k=(i_k, i_{k+1})$ for $1\le k\le n-1$, and $P_n=(i_n, +\infty]$ are called \textbf{basic intervals with respect to input edges}.

To define the composition $(G_2,\prec_2)\circ (G_1,\prec_1)$, we first compose $G_1$ and $G_2$ into a processive graph $G_2\circ G_1$
whose edge set is the disjoint union of  $E(G_1)-\{o_1,\cdots,o_n\}$, $E(G_2)-\{i_1,\cdots,i_n\}$ and $\{\overline{e_1},\cdots,\overline{e_n}\}$, where $\overline{e_1},\cdots,\overline{e_n}$ are newly added edges, as Fig \ref{14} shows. We mention that in  $G_2\circ G_1$ we remove all sinks of $G_1$ and all source of $G_2$.

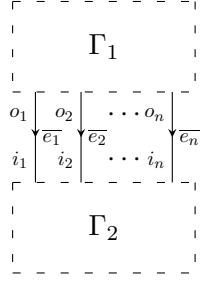
\begin{figure}[htbp]
\centering
\begin{tikzpicture}[scale=0.6]

\draw [loosely dashed] (-1,2.5) -- (-1,0.5);
\draw [loosely dashed] (-1,2.5) --(3,2.5);
\draw [loosely dashed] (3,2.5) -- (3,0.5);
\draw [loosely dashed] (-1,0.5) -- (3,0.5);

\node (v9) at (-0.5,0.5) {};
\node (v13) at (2.5,0.5) {};
\node (v11) at (0.5,0.5) {};
\node (v16) at (-0.5,-1.5) {};
\node (v18) at (0.5,-1.5) {};
\node (v20) at (2.5,-1.5) {};
\node (v15) at (-0.5,-1) {};

\node (v10) at (-0.5,0) {};
\node [scale=0.7][left] at (-0.5,0) {$o_1$};
\node (v12) at (0.5,0) {};
\node [scale=0.7][left] at (0.5,0) {$o_2$};
\node (v14) at (2.5,0) {};
\node [scale=0.7][left] at (2.5,0) {$o_n$};

\node [scale=0.7][right] at (-0.5,-0.5) {$\overline{e_1}$};
\node [scale=0.7][right] at (0.5,-0.5) {$\overline{e_2}$};
\node [scale=0.7][right] at (2.5,-0.5) {$\overline{e_n}$};

\node at (1,1.5) {$\Gamma_1$};

\node at (1.5,0) {$\cdots$};


\node (v5) at (-1,-1.5) {};
\node (v6) at (3,-1.5) {};
\node (v7) at (-1,-3.5) {};
\node (v8) at (3,-3.5) {};

\draw [loosely dashed] (-1,-1.5) -- (3,-1.5);
\draw [loosely dashed] (-1,-1.5) -- (-1,-3.5);
\draw [loosely dashed] (3,-1.5) -- (3,-3.5);
\draw [loosely dashed] (-1,-3.5) -- (3,-3.5);

\node [scale=0.7][left] at (-0.5,-1) {$i_1$};
\node (v17) at (0.5,-1) {};
\node[scale=0.7] [left] at (0.5,-1) {$i_2$};
\node (v19) at (2.5,-1) {};
\node [scale=0.7][left] at (2.5,-1) {$i_n$};
\node (v10) at (-0.5,0) {};

\draw  (-0.5,0.5) -- (-0.5,-1.5)[postaction={decorate, decoration={markings,mark=at position .5 with {\arrow[black]{stealth}}}}];
\draw  (0.5,0.5) -- (0.5,-1.5)[postaction={decorate, decoration={markings,mark=at position .5 with {\arrow[black]{stealth}}}}];
\draw  (2.5,0.5) -- (2.5,-1.5)[postaction={decorate, decoration={markings,mark=at position .5 with {\arrow[black]{stealth}}}}];

\node at (1,-2.5) {$\Gamma_2$};

\node at (1.5,-1) {$\cdots$};
\end{tikzpicture}
\caption{The composition $G_2\circ G_1$, where $o_k$ and $i_k$ are jointed into $\overline{e_k}$.}
\label{14}
\end{figure}

\begin{defn}\label{com}
Assumed as above, the \textbf{composition} of $(G_1,\prec_1)$ and $(G_2,\prec_2)$ is the processive graph $G_2\circ G_1$ together with the composition $\prec_2\circ\prec_1$ of $\prec_1$ and $\prec_2$, which is the linear order on $E(G_2\circ G_1)$ in the shuffled form $$Q_1\triangleleft\{\overline{e_1}\}\triangleleft P_1\triangleleft...\triangleleft Q_k\triangleleft\{\overline{e_k}\}\triangleleft P_k \triangleleft...\triangleleft Q_n\triangleleft\{\overline{e_n}\}\triangleleft P_n.$$
\end{defn}
We write $(G_2,\prec_2)\circ (G_2,\prec_2)=(G_2\circ G_1, \prec_2\circ\prec_1)$.

\begin{ex}
Fig \ref{15} shows a composition of two POP-graphs.

\begin{figure}[htbp]
\centering
\begin{tikzpicture}[scale=0.7]

\node [scale=0.8](v1) at (-2,1) {$1$};
\node (v2) at (-2,-0.5) {};
\draw[fill] (v2) circle [radius=0.07];
\node [scale=0.8](v3) at (-2.5,-2) {$2$};
\node [scale=0.8](v4) at (-1.5,-2) {$3$};
\node [scale=0.8](v6) at (-0.5,-0.5) {};
\draw[fill] (v6) circle [radius=0.07];
\node [scale=0.8](v5) at (-1,1) {$4$};
\node [scale=0.8](v7) at (0,1) {$5$};
\node [scale=0.8](v8) at (-0.5,-2) {$6$};
\draw  (v1) -- (-2,-0.5)[postaction={decorate, decoration={markings,mark=at position .5 with {\arrow[black]{stealth}}}}];
\draw  (-2,-0.5) -- (v3)[postaction={decorate, decoration={markings,mark=at position .65 with {\arrow[black]{stealth}}}}];
\draw(-2,-0.5) -- (v4)[postaction={decorate, decoration={markings,mark=at position .65 with {\arrow[black]{stealth}}}}];
\draw  (v5) -- (-0.5,-0.5)[postaction={decorate, decoration={markings,mark=at position .5 with {\arrow[black]{stealth}}}}];
\draw  (v7) --(-0.5,-0.5)[postaction={decorate, decoration={markings,mark=at position .5 with {\arrow[black]{stealth}}}}];
\draw  (-0.5,-0.5) -- (v8)[postaction={decorate, decoration={markings,mark=at position .65 with {\arrow[black]{stealth}}}}];
\node at (2,-0.5) {$\circ$};

\node at (8,-0.5) {$=$};

\node (v19) at (10,0.5) {};
\draw[fill] (v19) circle [radius=0.07];
\node (v22) at (11.5,0.5) {};
\draw[fill] (v22) circle [radius=0.07];
\node [scale=0.8](v18) at (9.5,2.5) {$1$};
\node[scale=0.8] (v20) at (10.5,2.5) {$2$};
\node [scale=0.8](v21) at (11,2.5) {$7$};
\node [scale=0.8](v23) at (12,2.5) {$8$};
\draw  (v18) -- (10,0.5)[postaction={decorate, decoration={markings,mark=at position .5 with {\arrow[black]{stealth}}}}];
\draw  (v20) -- (10,0.5)[postaction={decorate, decoration={markings,mark=at position .5 with {\arrow[black]{stealth}}}}];
\draw  (v21) -- (11.5,0.5)[postaction={decorate, decoration={markings,mark=at position .5 with {\arrow[black]{stealth}}}}];
\draw  (v23) -- (11.5,0.5)[postaction={decorate, decoration={markings,mark=at position .5 with {\arrow[black]{stealth}}}}];
\node (v24) at (10,-1.5) {};
\draw[fill] (v24) circle [radius=0.07];
\node (v25) at (11.5,-1.5) {};
\draw[fill] (v25) circle [radius=0.07];
\draw  (10,0.5) -- (10,-1.5)[postaction={decorate, decoration={markings,mark=at position .5 with {\arrow[black]{stealth}}}}];
\node [scale=0.8](v26) at (9.5,-3.5) {$4$};
\node [scale=0.8](v27) at (10.5,-3.5) {$5$};
\node [scale=0.8](v28) at (11.5,-3.5) {$10$};
\draw  (10,0.5)-- (11.5,-1.5)[postaction={decorate, decoration={markings,mark=at position .5 with {\arrow[black]{stealth}}}}];
\draw  (11.5,0.5) -- (11.5,-1.5)[postaction={decorate, decoration={markings,mark=at position .5 with {\arrow[black]{stealth}}}}];
\draw (10,-1.5) -- (v26)[postaction={decorate, decoration={markings,mark=at position .65 with {\arrow[black]{stealth}}}}];
\draw  (10,-1.5)-- (v27)[postaction={decorate, decoration={markings,mark=at position .65 with {\arrow[black]{stealth}}}}];
\draw  (11.5,-1.5) --(v28)[postaction={decorate, decoration={markings,mark=at position .65 with {\arrow[black]{stealth}}}}];
\node[scale=0.8] at (9.5,-0.5) {$3$};
\node[scale=0.8] at (10.6,-0.8) {$6$};
\node [scale=0.8]at (12,-0.6) {$9$};

\node (v10) at (4.5,-0.5) {};
\draw[fill] (v10) circle [radius=0.07];
\node (v15) at (6,-0.5) {};
\draw[fill] (v15) circle [radius=0.07];
\node[scale=0.8] (v9) at (4,1) {$1$};
\node [scale=0.8](v11) at (5,1) {$2$};
\node [scale=0.8](v14) at (5.5,1) {$5$};
\node [scale=0.8](v16) at (6.5,1) {$6$};

\node [scale=0.8](v12) at (4,-2) {$3$};
\node [scale=0.8](v13) at (5,-2) {$4$};
\node [scale=0.8](v17) at (6,-2) {$7$};
\draw  (v9) -- (4.5,-0.5)[postaction={decorate, decoration={markings,mark=at position .5 with {\arrow[black]{stealth}}}}];
\draw  (v11) -- (4.5,-0.5)[postaction={decorate, decoration={markings,mark=at position .5 with {\arrow[black]{stealth}}}}];
\draw  (4.5,-0.5) -- (v12)[postaction={decorate, decoration={markings,mark=at position .65 with {\arrow[black]{stealth}}}}];
\draw  (4.5,-0.5) --(v13)[postaction={decorate, decoration={markings,mark=at position .65 with {\arrow[black]{stealth}}}}];
\draw  (v14) --(6,-0.5)[postaction={decorate, decoration={markings,mark=at position .5 with {\arrow[black]{stealth}}}}];
\draw  (v16) -- (6,-0.5)[postaction={decorate, decoration={markings,mark=at position .5 with {\arrow[black]{stealth}}}}];
\draw  (6,-0.5) --(v17)[postaction={decorate, decoration={markings,mark=at position .65 with {\arrow[black]{stealth}}}}];

\node [scale=0.8]at (-2.6,-1.4) {$x$};
\node [scale=0.8]at (-1.4,-1.4) {$y$};
\node[scale=0.8] at (-0.2,-1.4) {$z$};

\node [scale=0.8]at (9.2,-2.8) {$x$};
\node [scale=0.8]at (10.6,-2.8) {$y$};
\node [scale=0.8]at (11.8,-2.8) {$z$};

\node [scale=0.8]at (3.8,0.4) {$a$};
\node[scale=0.8] at (5,0.4) {$b$};
\node[scale=0.8] at (5.4,0.4) {$c$};
\node[scale=0.8] at (6.6,0.4) {$d$};

\node [scale=0.8]at (9.4,1.6) {$a$};
\node [scale=0.8]at (10.5,1.6) {$b$};
\node [scale=0.8]at (10.9,1.6) {$c$};
\node [scale=0.8]at (12,1.6) {$d$};

\end{tikzpicture}

\caption{}
\label{15}
\end{figure}
\end{ex}
To show that $(G_2\circ G_1,\prec_2\circ\prec_1)$ is well-defined,  we need the following key result in this paper, whose proof will be given in Section $5$.

\begin{thm}\label{crux}
The composition of two POP-graphs is again a POP-graph.
\end{thm}
The associativity of the composition directly follows from the definition of $\prec_2\circ\prec_1$.
\begin{ex}
The composition of POP-graphs in Example \ref{ex1} yields the POP-graph in Example \ref{ex2}.
\end{ex}

The following lemma can be directly checked from the definitions.
\begin{lem}
The tensor product and composition of POP-graphs satisfy the middle-four-interchange law.
\end{lem}

Parallel to Theorem \ref{ppg}, we have the following result, which can be easily checked.

\begin{thm}\label{pop}
There is a monoidal category $\mathcal{POP}$ with non-negative integers as objects and with POP-graphs and the empty graph $\bigcirc$ as morphisms. For $m,n\geq 1$, a morphism from $m$ to $n$ is a POP-graph with $m$ sources and $n$ sinks. On objects the tensor product is given by the addition of integers, and on morphisms the tensor product is the tensor product of POP-graphs and $\bigcirc$. The composition is the composition of POP-graphs. The unit object is $0$, whose identity morphism is $\bigcirc$. The identity morphism of a non-unit object $n$ is the invertible POP-graph with $n$ edges.
\end{thm}

Similarly, there is a tensor scheme $\mathcal{PRM}$ with morphisms being  prime POP-graphs, with only one object and with source and target maps given by the numbers of sources and sinks, respectively.
Parallel to Theorem \ref{free1}, we have the following result, which will be proved in Section $7$.
\begin{thm}\label{free2}
The monoidal category $\mathcal{POP}$ is free on the tensor scheme $\mathcal{PRM}$.
\end{thm}

\section{Equivalence of two frameworks}
In this section, we show that the topological and combinatorial frameworks are equivalent. For this, we first give a reformulation of the notion of a processive plane graph.
\begin{defn}
A \textbf{BPP-graph} is a boxed planar drawing of a processive graph $G$, that is, a planar drawing of $G$ such that $G$ is drawn in a plane box with all sources on one horizontal boundary of the plane box and all sinks on the other horizontal boundary of the plane box.
\end{defn}
When the processive graph is clear or irrelevant, we only say a \textbf{BP-drawing} for short.  As shown in Fig \ref{1}, a PPG is exactly an upward BPP-graph.

An \textbf{anchor} \cite{[JS91]} of processive graph $G$ consists of two linear orders, one on the set $I(G)$ of input edges of $G$ and the other on the set $O(G)$ of output edges of $G$. Any BP-drawing of $G$ defines an anchor: $i_1<i_2$ in $I(G)$ if $s(i_1)$ (starting vertex of $i_1$) is on the left of $s(i_2)$ as points of one horizontal boundary and $o_1< o_2$ in $O(G)$ if $t(o_1)$ (ending vertex of $o_1$) is on the left of $t(o_2)$ as points of the other horizontal boundary.
\begin{defn}\label{eq2}
We say two BP-drawings are \textbf{equivalent} if they are connected by a planar isotopy.
\end{defn}
Then Definition \ref{eq1} means that two upward BP-drawings (=PPGs) are equivalent if they are equivalent \textbf{as BP-drawings}. Clearly, equivalent BP-drawings of a processive graph $G$ define the same anchor of $G$.

\begin{rem}\label{three}
The notion of a BPP-graph is essentially equivalent to that of a \textbf{plane $st$ graph} \cite{[BT88]}, see Fig \ref{plan-st}, which is a planar drawing of an acyclic directed graph with exactly one source $s$ and exactly one sink $t$ such that both $s$ and $t$ are drawn on the boundary of the external face (or equivalently, there is a distinguished edge $e$ connecting $s$ and $t$). Planar $st$ graphs have not only important applications in graph theory but also  a direct meaning in category theory. Actually, they are essentially \textbf{pasting schemes for $2$-categories} introduced by Power \cite{[P90]}, a special property of which is that any
of them can be deformed through a planar isotopy into an upward one (\cite{[LEC66]}, or see Theorem $14$ in \cite{[GT95]}).
\begin{figure}[h]
\centering
\begin{tikzpicture}[scale=0.3]
\node (v2) at (-4,3) {};
\draw[fill] (-1.5,5.5) circle [radius=0.15];
\node (v1) at (-1.5,5.5) {};
\node (v7) at (-1.5,1) {};
\node (v9) at (1.5,5.5) {};
\node (v14) at (2,1.5) {};
\node (v3) at (-3,7.5) {};
\node (v4) at (-2,7.5) {};
\node (v5) at (-0.5,7.5) {};
\node (v6) at (-4.8,7.5) {};
\node (v11) at (-4.5,-1) {};
\node (v12) at (-2,-1) {};
\node (v13) at (0,-1) {};
\node (v15) at (2,-1) {};
\node (v8) at (1,7.5) {};
\node (v10) at (2.5,7.5) {};
\node  at (-2.5,3.5) {};
\node  at (-3,5.2) {};
\node  at (-1.2,3.3) {};
\node  at (0.5,3.25) {};
\node  at (2.2,3.7) {};
\node  at (-3,1.7) {};
\draw[fill] (-4,3) circle [radius=0.15];
\draw[fill] (v1) circle [radius=0.15];
\draw[fill] (v7) circle [radius=0.15];
\draw[fill] (v9) circle [radius=0.15];
\draw[fill] (v14) circle [radius=0.15];
\draw[fill] (v1) circle [radius=0.15];
\draw[fill] (v2) circle [radius=0.15];

\draw  plot[smooth, tension=1] coordinates {(v1) (-2.5,5)  (-3.5,4) (v2)}[postaction={decorate, decoration={markings,mark=at position .5 with {\arrow[black]{stealth}}}}];
\draw  plot[smooth, tension=1] coordinates {(v1) (-2,4.5)  (-3,3.5) (v2)}[postaction={decorate, decoration={markings,mark=at position .5 with {\arrow[black]{stealth}}}}];

\draw  (-3,7.5) -- (-1.5,5.5)[postaction={decorate, decoration={markings,mark=at position .5 with {\arrow[black]{stealth}}}}];
\draw  (-2,7.5) -- (-1.5,5.5)[postaction={decorate, decoration={markings,mark=at position .5 with {\arrow[black]{stealth}}}}];
\draw  (-0.5,7.5) -- (-1.5,5.5)[postaction={decorate, decoration={markings,mark=at position .5 with {\arrow[black]{stealth}}}}];

\draw  (-4.8,7.5)-- (-4,3)[postaction={decorate, decoration={markings,mark=at position .5 with {\arrow[black]{stealth}}}}];
\draw  (-1.5,5.5)  -- (-1.5,1)[postaction={decorate, decoration={markings,mark=at position .5 with {\arrow[black]{stealth}}}}];
\draw  (-4,3) -- (-1.5,1)[postaction={decorate, decoration={markings,mark=at position .5 with {\arrow[black]{stealth}}}}];

\draw (1,7.5)--(1.5,5.5)[postaction={decorate, decoration={markings,mark=at position .5 with {\arrow[black]{stealth}}}}];
\draw  (2.5,7.5) -- (1.5,5.5)[postaction={decorate, decoration={markings,mark=at position .5 with {\arrow[black]{stealth}}}}];
\draw  (1.5,5.5) -- (-1.5,1)[postaction={decorate, decoration={markings,mark=at position .5 with {\arrow[black]{stealth}}}}];
\draw  (-4,3) -- (-4.5,-1)[postaction={decorate, decoration={markings,mark=at position .5 with {\arrow[black]{stealth}}}}];
\draw  (-1.5,1) -- (-2,-1)[postaction={decorate, decoration={markings,mark=at position .65 with {\arrow[black]{stealth}}}}];
\draw  (0,-1) -- (-1.5,1)[postaction={decorate, decoration={markings,mark=at position .5 with {\arrowreversed[black]{stealth}}}}];
\draw  (1.5,5.5) -- (2,1.5)[postaction={decorate, decoration={markings,mark=at position .5 with {\arrow[black]{stealth}}}}];
\draw  (2,1.5) -- (2,-1)[postaction={decorate, decoration={markings,mark=at position .5 with {\arrow[black]{stealth}}}}];

\node (v17) at (6.5,7.5) {};
\node (v16) at (6.5,-1) {};

\draw  (6.5,-1) -- (6.5,7.5)[postaction={decorate, decoration={markings,mark=at position .5 with {\arrowreversed[black]{stealth}}}}];
\node (v18) at (4.5,7.5) {};
\node (v19) at (4.5,-1) {};
\node (v20) at (4.5,3.25) {};

\draw[fill] (v20) circle [radius=0.15];
\draw  (4.5,-1) -- (4.5,3.25)[postaction={decorate, decoration={markings,mark=at position .5 with {\arrowreversed[black]{stealth}}}}];
\draw  (4.5,3.24) -- (4.5,7.5)[postaction={decorate, decoration={markings,mark=at position .5 with {\arrowreversed[black]{stealth}}}}];

\node at (-6.5,7.5) {};
\node at (-6.5,-1) {};
\node at (8,-1) {};
\node at (8,7.5) {};

\node [above] (v23) at (0.5,10.5) {$s$};
\node [below] (v33) at (0.5,-3.5){$t$};
\draw  (0.5,10.5) -- (-3,7.5);
\draw  (0.5,10.5) --(-2,7.5) ;
\draw  (0.5,10.5) --(-0.5,7.5);
\draw  (0.5,10.5) -- (-4.8,7.5);
\draw  (0.5,10.5) -- (1,7.5);
\draw (0.5,10.5) -- (2.5,7.5);
\draw  (0.5,10.5) --(6.5,7.5) ;
\draw  (0.5,10.5) -- (4.5,7.5);

\draw  (-4.5,-1) -- (0.5,-3.5);
\draw (-2,-1) -- (0.5,-3.5);
\draw (0,-1) -- (0.5,-3.5);
\draw  (2,-1) -- (0.5,-3.5);
\draw (6.5,-1) -- (0.5,-3.5);
\draw  (4.5,-1) -- (0.5,-3.5);

\draw[fill] (0.5,-3.5) circle [radius=0.15];
\draw[fill] (0.5,10.5) circle [radius=0.15];
\draw [loosely dashed] (-6.5,7.5)--(-6.5,-1);
\draw [loosely dashed] (8,7.5)--(8,-1);
\draw [loosely dashed] (-6.5,7.5)--(8,7.5);
\draw [loosely dashed] (-6.5,-1)--(8,-1);

\draw [dashed] plot[smooth, tension=.7] coordinates { (0.5,10.5) (-6.5,8.5) (-9,5.5) (-9,1) (-6.5,-2) (0.5,-3.5)}[postaction={decorate, decoration={markings,mark=at position .5 with {\arrow[black]{stealth}}}}];
\node at (-11,3) {$e$};
\end{tikzpicture}
\caption{}
\label{plan-st}
\end{figure}
 Moreover, a dual graph of a plane $st$ graph is a plane $st$ graph, in particular, a PPG is Poincar$\acute{e}$ dual to an upward pasting scheme for $2$-categories, see Fig \ref{9}.

\begin{figure}[htbp]
\centering
\begin{tikzpicture}[scale=0.5]

\draw [loosely dashed] (-3.5,2.5) rectangle (5.5,-2.5);
\node (v1) at (-3,0) {};
\node (v2) at (-1.5,0) {};
\node (v3) at (1,1.5) {};
\node (v4) at (3,0) {};
\node (v5) at (5,0) {};
\node (v6) at (0,-1.8) {};
\node (v7) at (1.8,-1.8) {};
\draw[fill] (v1) circle [radius=0.11];\draw[fill] (v2) circle [radius=0.11];\draw[fill] (v3) circle [radius=0.11];\draw[fill] (v4) circle [radius=0.11];\draw[fill] (v5) circle [radius=0.11];\draw[fill] (v6) circle [radius=0.11];\draw[fill] (v7) circle [radius=0.11];
\draw  (-3,0) -- (-1.5,0)[postaction={decorate, decoration={markings,mark=at position .65 with {\arrow[black]{stealth}}}}];
\draw (-1.5,0)-- (1,1.5)[postaction={decorate, decoration={markings,mark=at position .5 with {\arrow[black]{stealth}}}}];
\draw  (1,1.5) --  (3,0)[postaction={decorate, decoration={markings,mark=at position .5 with {\arrow[black]{stealth}}}}];
\draw   (3,0) -- (5,0)[postaction={decorate, decoration={markings,mark=at position .5 with {\arrow[black]{stealth}}}}];
\draw  (-1.5,0) -- (0,-1.8)[postaction={decorate, decoration={markings,mark=at position .5 with {\arrow[black]{stealth}}}}];
\draw  (0,-1.8) -- (1.8,-1.8)[postaction={decorate, decoration={markings,mark=at position .5 with {\arrow[black]{stealth}}}}];
\draw  (1.8,-1.8)  --  (3,0)[postaction={decorate, decoration={markings,mark=at position .5 with {\arrow[black]{stealth}}}}];

\node (v8) at (-2.5,2.5) {};
\node (v9) at (-2.5,-2.5) {};
\draw [densely dashed,thick](-2.5,2.5)-- (-2.5,-2.5)[postaction={decorate, decoration={markings,mark=at position .45 with {\arrow[black]{stealth}}}}];
\node (v11) at (1,0) {};

\node (v10) at (-1,2.5) {};
\node (v12) at (2.5,2.5) {};
\node (v13) at (-1.2,-2.5) {};
\node (v14) at (1,-2.5) {};
\node (v15) at (3.4,-2.5) {};
\node (v17) at (4.5,2.5) {};
\node (v16) at (4.5,-2.5) {};
\draw [densely dashed, thick] (-1,2.5) -- (1,0)[postaction={decorate, decoration={markings,mark=at position .5 with {\arrow[black]{stealth}}}}];
\draw  [densely dashed,thick]((2.5,2.5) -- (1,0)[postaction={decorate, decoration={markings,mark=at position .5 with {\arrow[black]{stealth}}}}];
\draw  [densely dashed,thick] ((1,0) -- (-1.2,-2.5)[postaction={decorate, decoration={markings,mark=at position .5 with {\arrow[black]{stealth}}}}];
\draw  [densely dashed,thick](1,0) --  (1,-2.5)[postaction={decorate, decoration={markings,mark=at position .5 with {\arrow[black]{stealth}}}}];
\draw  [densely dashed,thick] (1,0) -- (3.4,-2.5)[postaction={decorate, decoration={markings,mark=at position .45 with {\arrow[black]{stealth}}}}];
\draw   [densely dashed,thick](4.5,2.5)  -- (4.5,-2.5)[postaction={decorate, decoration={markings,mark=at position .45 with {\arrow[black]{stealth}}}}];

\node (v18) at (1,3.6) {};
\node (v19) at (1,-4.2) {};
\draw[ball color=white] (v18) circle [radius=0.11];
\draw[ball color=white] (v19) circle [radius=0.11];
\draw [dotted] (1,3.6) --(-2.5,2.5)[postaction={decorate, decoration={markings,mark=at position .5 with {\arrow[black]{stealth}}}}];
\draw  [dotted](1,3.6)-- (-1,2.5)[postaction={decorate, decoration={markings,mark=at position .5 with {\arrow[black]{stealth}}}}];
\draw  [dotted](1,3.6) -- (2.5,2.5)[postaction={decorate, decoration={markings,mark=at position .5 with {\arrow[black]{stealth}}}}];
\draw  [dotted](1,3.6) -- (4.5,2.5) [postaction={decorate, decoration={markings,mark=at position .5 with {\arrow[black]{stealth}}}}];
\draw  [dotted](-2.5,-2.5) -- (1,-4.2)[postaction={decorate, decoration={markings,mark=at position .5 with {\arrow[black]{stealth}}}}];
\draw  [dotted](-1.2,-2.5) -- (1,-4.2)[postaction={decorate, decoration={markings,mark=at position .5 with {\arrow[black]{stealth}}}}];
\draw  [dotted](1,-2.5)  -- (1,-4.2)[postaction={decorate, decoration={markings,mark=at position .5 with {\arrow[black]{stealth}}}}];
\draw  [dotted](3.4,-2.5) -- (1,-4.2)[postaction={decorate, decoration={markings,mark=at position .5 with {\arrow[black]{stealth}}}}];
\draw [dotted](4.5,-2.5) -- (1,-4.2)[postaction={decorate, decoration={markings,mark=at position .5 with {\arrow[black]{stealth}}}}];
\draw [dotted] plot[smooth, tension=.7] coordinates {(v18) (4.6,3.4) (6.4,2.4) (6.8,0) (6.4,-2.6) (4.6,-3.4) (v19)}[postaction={decorate, decoration={markings,mark=at position .5 with {\arrow[black]{stealth}}}}];
\draw [dotted] plot[smooth, tension=.7] coordinates {(v1) (-3.8,-0.4) (-4.6,-1.8) (-4,-3.6) (-0.8,-5) (3.4,-5) (6,-4.4) (7.6,-3.2) (7.6,-1.2) (6.4,-0.2) (v5)}[postaction={decorate, decoration={markings,mark=at position .5 with {\arrow[black]{stealth}}}}];
\draw[ball color=white] (v18) circle [radius=0.11];
\draw[ball color=white] (v19) circle [radius=0.11];
\draw[ball color=white] (v11) circle [radius=0.11];
\node at (-3,0.6) {$s$};
\node at (5,0.6) {$t$};
\node at (1,4.2) {$s^\ast$};
\node at (0.8,-4.8) {$t^\ast$};
\node at (0.2,-0.2) {$\Downarrow$};
\node at (7.4,1.6) {$e^\ast$};
\node at (3.2,-5.5) {$e$};
\end{tikzpicture}
\caption{}
\label{9}
\end{figure}
\end{rem}

For prime processive graphs, it is easy to see that its planar orders are in bijective with its anchors, which determines  equivalence classes of its upward BP-drawings. As shown in Example \ref{ex1}, we can easily see the following result, which acts as the cornerstone of the relationship between the two frameworks.
\begin{lem}\label{basic}
For any prime processive graph, its planar orders are in bijective with equivalence classes of its upward BP-drawings.
\end{lem}

Here are some notations. The set of internal vertices of an acyclic directed graph $G$ is denoted as $V_{int}(G)$, which, with the order that $v_1<v_2$ if there is a directed path from $v_1$ to $v_2$ (denoted as $v_1\rightarrow v_2$), is a poset. The set of incoming edges and the set of outgoing edges of $v$ are denoted as $I(v)$ and $O(v)$, respectively.

The following proposition  rationalizes our definition of equivalence relation of PPGs (Definition \ref{eq1}), which also shows that the upward property is not an essential requirement in the definition of a PPG.
\begin{prop}\label{Thm6}
Any BP-drawing of a processive graph is equivalent to an upward one.
\end{prop}
\begin{proof}
Let $G$ be a processive graph with a BP-drawing $\phi$. We use induction on $| V_{int}(G)|$. If $|V_{int}(G)|=1$, $G$ is elementary and the result is obvious.
Assume the theorem holds for $|V_{int}(G)|<n$, we will show that the theorem also holds for $|V_{int}(G)|= n$.

Let $v\in V_{int}(G)$ be a maximal vertex of $G$ and $\curvearrowright$ be the cyclic order on the set $E(v)$ of incident edges of $v$ induced by $\phi$. We will prove by contradiction two claims: $(1)$ $O(v)$ is an interval of $(E(v),\curvearrowright)$; $(2)$ $O(v)$ is an interval of $O(G)$ with respect to the anchor.

$(1)$ Suppose there exist $o_1,o_2\in O(v)$ and $h\in I(v)$ such that  $o_2\curvearrowright h\curvearrowright o_1$. Since $G$ is processive, there is a directed path $P$ from a source $s$ to $v$ and ending with $h$. Then  $\phi(P)\cap(\phi(o_1)\cup\phi(o_2))\neq \emptyset$, see the left of Fig \ref{a1}, which contradicts the planarity of $\phi$.

$(2)$ By maximality of $v$,  $O(v)\subseteq O(G)$. Suppose there exist $o_1,o_2\in O(v)$ and $o_3\in O(G)-O(v)$ such that $o_1<o_3<o_2$ with respect to the anchor. Since $G$ is processive, there is a directed path $\widetilde{P}$ from a source $\widetilde{s}$ to the ending vertex of $o_3$. The planarity of $\phi$ and  maximality of $v$ imply that $\phi(\widetilde{P})\cap(\phi(o_1)\cup\phi(o_2))=\{\phi(v)\}$, see the right of Fig \ref{a1}, which contradicts  $o_3\not\in O(v)$.

\begin{figure}[htbp]
\centering
$$
\begin{matrix}
\begin{matrix}
\begin{tikzpicture}[scale=0.6]

\node (v1) at (0,0) {};
\node (v3) at (0,-6) {};
\node (v2) at (10.5,0) {};
\node (v4) at (10.5,-6) {};
\draw [loosely dashed] (0,0) -- (10.5,0);
\draw  [loosely dashed](0,0)-- (0,-6);
\draw [loosely dashed] (10.5,0) -- (10.5,-6);
\draw  [loosely dashed](0,-6) -- (10.5,-6);

\node [left](v5) at (6.1975,-3.9875) {};
\node[scale=0.7]  at (5.5,-3.8) {$\phi(v)$};

\draw[fill] (6.1975,-3.9875) circle [radius=0.11];
\draw [densely dashed][domain=125:-150] plot ({0.5*cos(\x)+6.1975}, {0.5*sin(\x)-3.9875})[postaction={decorate, decoration={markings,mark=at position .45 with {\arrow[black]{stealth}}}}];

\draw  plot[smooth, tension=.7] coordinates {(6.1975,-3.9875) (5.8187,-4.3498) (5.6,-4.8) (5.4049,-5.2073) (5.2,-5.5) (5.2,-6)}[postaction={decorate, decoration={markings,mark=at position .42 with {\arrow[black]{stealth}}}}];
\node  at (5.2,-4.35) {};
\draw  plot[smooth, tension=.7] coordinates {(6.1975,-3.9875)(6.7881,-4.1871) (6.9639,-4.5737) (7.2099,-4.9955) (7.3856,-5.3821) (7.6317,-5.6985) (7.8777,-6)}[postaction={decorate, decoration={markings,mark=at position .51 with {\arrow[black]{stealth}}}}];
\node [scale=0.7][right] at (7.3,-4.9955) {$\phi(o_2)$};
\draw[fill] (5.2,-6) circle [radius=0.11];
\node [below] at (5.2,-6.2) {};
\draw[fill] (7.8777,-6) circle [radius=0.11];
\node [below] at (7.8777,-6.2) {};

\draw  plot[smooth, tension=.7] coordinates {(7.6317,0) (7.702,-0.3558) (7.702,-0.7776) (7.6317,-1.1291) (7.5965,-1.4454) (7.4208,-1.8672) (7.3505,-2.2187) (7.1747,-2.6053) (7.1396,-2.9217) (6.999,-3.238) (6.8936,-3.3435) (6.4881,-3.6247)(6.3,-3.8) (6.1975,-3.9875)}[postaction={decorate, decoration={markings,mark=at position .51 with {\arrow[black]{stealth}}}}];
\draw  plot[smooth, tension=.7] coordinates {(3.7301,0) (3.7301,-0.2503) (3.7301,-0.7073) (3.8004,-1.34) (3.8004,-1.9375) (3.9059,-2.4999) (3.9762,-2.9217) (4.1168,-3.6247) (4.4331,-4.1168) (4.89,-4.6791) (5.4049,-5.2073) (5.7336,-5.3821) (6.0851,-5.3119) (6.296,-4.9955) (6.296,-4.3979) (6.1975,-3.9875)}[postaction={decorate, decoration={markings,mark=at position .5 with {\arrow[black]{stealth}}}}][postaction={decorate, decoration={markings,mark=at position .9 with {\arrow[black]{stealth}}}}];
\node[scale=0.7] [right] at (6.,-5.5) {$\phi(h)$};

\node [scale=0.7][right] at (3.8004,-1.9375) {$\phi(P)$};
\draw[fill] (7.6317,0) circle [radius=0.11];

\draw[fill] (3.7301,0) circle [radius=0.11];
\node [scale=0.7][above] at (3.7301,0.2) {$\phi(s)$};
\draw (5.4049,-5.2073) circle [radius=0.11];
\node [scale=0.7][left] at (5.3,-5.3) {$\phi(o_1)$};

\end{tikzpicture}
\end{matrix}
&
\begin{matrix}
\begin{tikzpicture}[scale=0.6]

\node (v1) at (0,0) {};
\node (v3) at (0,-6) {};
\node (v2) at (10.5,0) {};
\node (v4) at (10.5,-6) {};
\draw [loosely dashed] (0,0) -- (10.5,0);
\draw  [loosely dashed](0,0)-- (0,-6);
\draw [loosely dashed] (10.5,0) -- (10.5,-6);
\draw  [loosely dashed](0,-6) -- (10.5,-6);

\node [left](v5) at (6.1975,-3.9875) {};
\node [scale=0.7] at (5.5,-3.8) {$\phi(v)$};

\draw[fill] (6.1975,-3.9875) circle [radius=0.11];
\draw [densely dashed][domain=125:-150] plot ({0.5*cos(\x)+6.1975}, {0.5*sin(\x)-3.9875})[postaction={decorate, decoration={markings,mark=at position .45 with {\arrow[black]{stealth}}}}];

\draw  plot[smooth, tension=.7] coordinates {(6.1975,-3.9875) (5.8187,-4.3498) (5.6,-4.8) (5.3945,-5.0649) (5.2,-5.5) (5.2,-6)}[postaction={decorate, decoration={markings,mark=at position .42 with {\arrow[black]{stealth}}}}];
\node  at (5.3,-4.35) {};
\draw  plot[smooth, tension=.7] coordinates {(6.1975,-3.9875)(6.7881,-4.1871) (6.9639,-4.5737) (7.2099,-4.9955) (7.3856,-5.3821) (7.6317,-5.6985) (7.8777,-6)}[postaction={decorate, decoration={markings,mark=at position .51 with {\arrow[black]{stealth}}}}];
\node [scale=0.7][right] at (7.3,-4.9955) {$\phi(o_2)$};
\draw[fill] (5.2,-6) circle [radius=0.11];
\node [below] at (5.2,-6.2) {};
\draw[fill] (7.8777,-6) circle [radius=0.11];
\node [below] at (7.8777,-6.2) {};

\draw  plot[smooth, tension=.7] coordinates {(7.6317,-0.0043) (7.702,-0.3558) (7.702,-0.7776) (7.6317,-1.1291) (7.5965,-1.4454) (7.4208,-1.8672) (7.3505,-2.2187) (7.1747,-2.6053) (7.1396,-2.9217) (6.999,-3.238) (6.8936,-3.3435) (6.4881,-3.6247)(6.3,-3.8) (6.1975,-3.9875)}[postaction={decorate, decoration={markings,mark=at position .51 with {\arrow[black]{stealth}}}}];

\draw[fill] (7.6317,0) circle [radius=0.11];

\draw[fill] (3,0) circle [radius=0.11];
\node [scale=0.7][above] at (3,0.2) {$\phi(\widetilde{s})$};

\node (v6) at (3.4489,-3.5192) {};

\node[left] at (3.4489,-3.5192) {};
\node (v7) at  (5.9,-5.6) {};
\node [scale=0.7][below] at (4.4,-1.6) {$\phi(\widetilde{P})$};
\draw[fill] (6,-6) circle [radius=0.11];

\draw  plot[smooth, tension=.7] coordinates {(3,0) (3.5,-0.5) (4,-1.1) (4.7,-1.6) (5.2,-2.2)  (5.9,-3.3)(6.2,-4) (6.2,-4.5) (6,-5.1) (v7) (6,-6)}[postaction={decorate, decoration={markings,mark=at position .51 with {\arrow[black]{stealth}}}}][postaction={decorate, decoration={markings,mark=at position .91 with {\arrow[black]{stealth}}}}];

\node[scale=0.7] [left] at (5.1,-5.1649) {$\phi(o_1)$};

\node [scale=0.7][right] at(6.,-5.4){$\phi(o_3)$};
\end{tikzpicture}
\end{matrix}
\end{matrix}
$$
\caption{}
\label{a1}
\end{figure}

The two claims enable us to cut $\phi$, along some dotted line in the plane box, into two BP-drawings  with one of them containing $v$ as the unique internal vertex,  see Fig  \ref{a2} for an example. By the induction hypothesis, both of the two BP-drawings are equivalent to upward ones. So as their composition, $\phi$ is equivalent to an upward one.

\begin{figure}[H]
\centering
\begin{tikzpicture}[scale=0.6]
\draw [loosely dashed] (0,0) -- (10.5,0);
\draw  [loosely dashed](0,0)-- (0,-6);
\draw [loosely dashed] (10.5,0) -- (10.5,-6);
\draw  [loosely dashed](0,-6) -- (10.5,-6);

\node (v7) at (6.371,-4.2269) {};

\node (v2) at (2.1882,-3.9105) {};
\node (v3) at (3.6996,-2.0125) {};
\node (v14) at (8.5854,-3.1021) {};
\node (v1) at (1.0986,0) {};
\node (v6) at (3.5239,0) {};
\node (v11) at (6.2304,0) {};
\node (v12) at (7.2849,0) {};
\node (v13) at (8.6908,0) {};
\node (v15) at (9.6047,0) {};
\node (v4) at (1.4501,-6) {};
\node (v5) at (2.7857,-6) {};
\node (v8) at (4.6838,-5.9971) {};
\node (v9) at (5.9843,-6.0323) {};
\node (v10) at (7.0388,-5.962) {};
\node (v16) at (8.5502,-5.9971) {};
\node (v17) at (9.3938,-5.9971) {};

\draw[fill,scale=0.8] (v1) circle [radius=0.11];
\draw[fill,scale=0.8] (v2) circle [radius=0.11];
\draw[fill,scale=0.8] (v3) circle [radius=0.11];
\draw[fill,scale=0.8] (v4) circle [radius=0.11];
\draw[fill,scale=0.8] (v5) circle [radius=0.11];
\draw[fill,scale=0.8] (v6) circle [radius=0.11];
\draw[fill,scale=0.8] (v7) circle [radius=0.11];
\draw[fill,scale=0.8] (v8) circle [radius=0.11];
\draw[fill,scale=0.8] (v9) circle [radius=0.11];
\draw[fill,scale=0.8] (v10) circle [radius=0.11];
\draw[fill,scale=0.8] (v11) circle [radius=0.11];
\draw[fill,scale=0.8] (v12) circle [radius=0.11];
\draw[fill,scale=0.8] (v13) circle [radius=0.11];
\draw[fill,scale=0.8] (v14) circle [radius=0.11];
\draw[fill,scale=0.8] (v15) circle [radius=0.11];
\draw[fill,scale=0.8] (v16) circle [radius=0.11];
\draw[fill,scale=0.8] (v17) circle [radius=0.11];

\draw  plot[smooth, tension=.7] coordinates {(v1) (1.0634,-0.4084) (1.0634,-0.8653) (1.0986,-1.2519) (1.2743,-1.744) (1.3798,-2.2361) (1.4852,-2.8688) (1.8719,-3.3961) (v2)}[postaction={decorate, decoration={markings,mark=at position .6 with {\arrow[black]{stealth}}}}];

\draw  plot[smooth, tension=.7] coordinates {(v3) (2.8912,-2.2713) (2.8912,-2.8337) (2.6803,-3.6421) (v2)}[postaction={decorate, decoration={markings,mark=at position .6 with {\arrow[black]{stealth}}}}];
\draw  plot[smooth, tension=.7] coordinates {(v3) (3.7348,-2.5876) (3.5942,-3.5015) (3.4887,-4.0639) (2.856,-4.2748) (v2)}[postaction={decorate, decoration={markings,mark=at position .6 with {\arrow[black]{stealth}}}}];

\draw  plot[smooth, tension=.7] coordinates {(v2) (1.9773,-4.1693) (1.8016,-4.4154) (1.661,-4.8723) (1.5907,-5.3293) (1.4852,-5.7159) (v4)}[postaction={decorate, decoration={markings,mark=at position .6 with {\arrow[black]{stealth}}}}];
\draw  plot[smooth, tension=.7] coordinates {(v2) (2.2937,-4.3451) (2.4343,-4.7669) (2.61,-5.0832) (2.6803,-5.4699) (v5)}[postaction={decorate, decoration={markings,mark=at position .6 with {\arrow[black]{stealth}}}}];

\draw  plot[smooth, tension=.7] coordinates {(v6) (3.5942,-0.3029) (3.7348,-0.7247) (3.7348,-1.1114) (3.8051,-1.6034) (v3)}[postaction={decorate, decoration={markings,mark=at position .6 with {\arrow[black]{stealth}}}}];
\draw  plot[smooth, tension=.7] coordinates {(v7) (6.4764,-3.8882) (6.4764,-3.3961) (6.1952,-3.2555) (5.7734,-3.8179) (5.5626,-4.3802) (5.3868,-4.8723) (5.2111,-5.3644) (5.0353,-5.6808) (v8)}[postaction={decorate, decoration={markings,mark=at position .6 with {\arrow[black]{stealth}}}}];
\draw  plot[smooth, tension=.7] coordinates {(v7) (6.1952,-4.556) (6.1249,-4.9778) (6.0546,-5.3293) (6.0195,-5.7862) (v9)}[postaction={decorate, decoration={markings,mark=at position .6 with {\arrow[black]{stealth}}}}];
\draw  plot[smooth, tension=.7] coordinates {(v7) (6.617,-4.5208) (6.7576,-4.8723) (6.9334,-5.0129) (7.0388,-5.259) (7.074,-5.5402) (v10)}[postaction={decorate, decoration={markings,mark=at position .6 with {\arrow[black]{stealth}}}}];
\draw  plot[smooth, tension=.7] coordinates {(v11) (6.4061,-0.6544) (6.5467,-1.4628) (6.5467,-2.0604) (6.617,-2.4119) (6.6522,-2.6228) (7.0037,-3.2203) (7.1443,-3.8179) (6.9334,-4.099) (v7)}[postaction={decorate, decoration={markings,mark=at position .6 with {\arrow[black]{stealth}}}}];

\draw  plot[smooth, tension=.7] coordinates {(v12) (7.5661,-0.3029) (7.6012,-0.8301) (7.8824,-1.4628) (7.8472,-2.3064) (8.023,-2.8337) (7.8824,-3.7827) (7.9176,-4.2748) (7.7418,-4.7669) (7.2849,-4.7669) (6.9685,-4.4505) (v7)}[postaction={decorate, decoration={markings,mark=at position .53 with {\arrow[black]{stealth}}}}];
\draw  plot[smooth, tension=.7] coordinates {(v13) (8.7611,-0.549) (8.8314,-1.1465) (8.7611,-2.1658) (v14) }[postaction={decorate, decoration={markings,mark=at position .53 with {\arrow[black]{stealth}}}}];
\draw  plot[smooth, tension=.7] coordinates {(v15) (9.6399,-0.4084) (9.7102,-0.9356) (9.6399,-1.5331) (9.5696,-2.0955) (9.3235,-2.5876) (v14)}[postaction={decorate, decoration={markings,mark=at position .5 with {\arrow[black]{stealth}}}}];
\draw  plot[smooth, tension=.7] coordinates {(v14) (8.6557,-3.607) (8.5502,-4.2748) (8.5151,-4.7669) (8.5151,-5.3644) (v16)}[postaction={decorate, decoration={markings,mark=at position .5 with {\arrow[black]{stealth}}}}];
\draw  plot[smooth, tension=.7] coordinates {(v14) (8.9369,-3.5718) (9.1829,-3.9936) (9.2884,-4.5911) (9.2884,-5.1184) (9.3235,-5.5753)(v17)}[postaction={decorate, decoration={markings,mark=at position .5 with {\arrow[black]{stealth}}}}];
\draw [dotted]  plot[smooth, tension=.7] coordinates {(0.0089,-5.505) (0.8877,-5.505) (1.907,-5.505) (2.9615,-5.4699) (3.9105,-5.4699) (4.6838,-5.1887) (5.1759,-4.2748) (5.5977,-3.4664) (6.1249,-2.9391) (6.6522,-3.15) (6.7576,-3.6773) (7.3552,-5.0832) (7.5661,-5.5402) (8.023,-5.5753) (9.7805,-5.4347) (10.5186,-5.4347)};
\end{tikzpicture}

\caption{}
\label{a2}
\end{figure}
\end{proof}

An equivalence class of BPP-graphs is called an \textbf{BPP-class}. Proposition \ref{Thm6} implies the following result.
\begin{coro}\label{bpp}
 PPG-classes are in bijection with BPP-classes.
\end{coro}

\begin{rem}\label{seven}
Proposition \ref{Thm6} is essentially equivalent to a classical result in graph theory, as pointed out in Remark \ref{three}, that any plane $st$ graph can be deformed through a planar isotopy into an upward one.
\end{rem}

By Lemma \ref{basic}, it is not difficult to see that  prime PPG-classes are in bijective with prime POP-graphs, and therefore that $\mathcal{PRM}$ and $\mathbf{PRM}$ are equivalent (for the definition of a morphism of tensor schemes, see Definition \ref{mor1}).

Together with Theorem \ref{free1}, Theorem \ref{free2}, the equivalence of $\mathcal{PRM}$ and $\mathbf{PRM}$ implies the following result, which indicates the equivalence of the two frameworks.

\begin{thm}\label{main}
$\mathbf{PPG}$ and $\mathcal{POP}$ are equivalent as monoidal categories.
\end{thm}
Theorem \ref{main} also shows that POP-graphs combinatorially characterize PPG-classes and BPP-classes (by Corollary \ref{bpp}), and therefore justify our definitions of equivalence relations of PPGs (Definition \ref{eq1}) and BPP-graphs (Definition \ref{eq2}). For example, the POP-graph in Fig \ref{12} characterizes the equivalence class of the PPG in Fig \ref{1}.

\begin{rem}\label{four}
The combinatorial characterization of a PPG-class in terms of a processive graph and a planar order is essentially equivalent to the characterization of a planar embedding of an $st$-graph in terms of the conjugate order of edge poset (Theorem $14$ in \cite{[FM96]}).
\end{rem}

\begin{rem}\label{five}
Similar to the combinatorial characterization of a PPG-class,
there is a totally combinatorial characterization of a progressive plane graph (without isolated vertices) by the notion of a \textbf{UPO-graph} (abbreviation of  \textbf{upward planarly ordered graph}) \cite{[LY16]}.
\end{rem}

\section{Properties of POP-graphs}

In this section, we show some basic properties of POP-graphs. We fix a POP-graph $(G,\prec)$.

\begin{lem}\label{lem 1}
Let $e_1,e_2,e,e'\in E(G)$.

$(1)$ If $e_1\rightarrow e\leftarrow e_2$ and $e_1\prec e'\prec e_2$, then $e_1\nrightarrow e'$ implies $e'\rightarrow e$.

$(2)$ If $e_1\leftarrow e\rightarrow e_2$ and $e_1\prec e'\prec e_2$, then $e'\nrightarrow e_2$ implies $e\rightarrow e'$.
\end{lem}
The first result can be represented graphically as Fig \ref{16}.
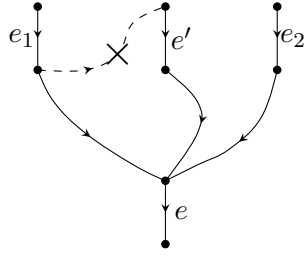
\begin{figure}[htbp]
\centering
\begin{tikzpicture}[scale=0.7]
\node (v1) at (-2.7,2.4) {};
\node (v2) at (-2.7,1.2) {};
\node (v3) at (1.8,2.4) {};
\node (v4) at (1.8,1.2) {};
\node (v5) at (-0.3,-0.9) {};
\node (v6) at (-0.3,-2.1) {};
\draw  (-2.7,2.4) -- (-2.7,1.2)[postaction={decorate, decoration={markings,mark=at position 0.5 with {\arrow[black]{stealth}}}}];
\draw[fill] (-2.7,2.4) circle [radius=0.07];
\draw[fill] (-2.7,1.2) circle [radius=0.07];
\draw  (1.8,2.4) -- (1.8,1.2)[postaction={decorate, decoration={markings,mark=at position 0.5 with {\arrow[black]{stealth}}}}];
\draw[fill] (1.8,2.4) circle [radius=0.07];
\draw[fill] (1.8,1.2) circle [radius=0.07];
\draw  (-0.3,-0.9)-- (-0.3,-2.1)[postaction={decorate, decoration={markings,mark=at position 0.5 with {\arrow[black]{stealth}}}}];
\draw[fill] (-0.3,-0.9) circle [radius=0.07];
\draw[fill] (-0.3,-2.1) circle [radius=0.07];
\draw  plot[smooth, tension=.7] coordinates {(v2) (-2.4,0.6) (-1.8,0) (-0.9,-0.6) (v5)}[postaction={decorate, decoration={markings,mark=at position 0.5 with {\arrow[black]{stealth}}}}];
\draw  plot[smooth, tension=.7] coordinates {(v4) (1.5,0.3)  (0.9,-0.3) (0.3,-0.6) (-0.3,-0.9)}[postaction={decorate, decoration={markings,mark=at position 0.5 with {\arrow[black]{stealth}}}}];
\node (v7) at (-0.3,2.4) {};
\node (v8) at (-0.3,1.2) {};
\draw  (-0.3,2.4) -- (-0.3,1.2)[postaction={decorate, decoration={markings,mark=at position 0.5 with {\arrow[black]{stealth}}}}];
\draw[fill] (-0.3,2.4) circle [radius=0.07];
\draw[fill] (-0.3,1.2) circle [radius=0.07];
\draw  plot[smooth, tension=.7] coordinates {(v8) (0.3,0.6) (0.3,0) (-0.3,-0.9)}[postaction={decorate, decoration={markings,mark=at position 0.5 with {\arrow[black]{stealth}}}}];
\node at (-3,1.8) {$e_1$};
\node at (0,1.8) {$e'$};
\node at (2.1,1.8) {$e_2$};
\node at (0,-1.5) {$e$};
\draw [dashed] plot[smooth, tension=.7] coordinates {(v2) (-1.8,1.2) (-1.2,1.5) (-0.9,2.1) (v7)}[postaction={decorate, decoration={markings,mark=at position 0.35 with {\arrow[black]{stealth}}}}];
\node [scale=1.5]at (-1.2,1.5) {$\times$};
\end{tikzpicture}
\caption{Under the conditions in $(1)$, $e_1\nrightarrow e'$ implies $e'\rightarrow e$.}
\label{16}
\end{figure}

\begin{proof}
We only prove $(1)$, and the proof for $(2)$ is similar. By $(P1)$, $e_2\rightarrow e$ implies that $e_2\prec e$. So $e_1\prec e'\prec e_2\prec e$, then by $(P2)$, $e_1\rightarrow e$ implies that either $e_1\to e'$ or $e'\rightarrow e$.
\end{proof}

Recall that the sets of input edges and output edges of $G$ are denoted as $I(G)$ and $O(G)$, respectively.
For any $e\in E(G)$, we introduce four notations:
$$i^-(e)=min\{i_k\in I(G)| i_k\rightarrow e\},$$
$$i^+(e)=max\{i_k\in I(G)| i_k\rightarrow e\},$$
$$o^-(e)=min\{o_k\in O(G)| e\rightarrow o_k\},$$
$$o^+(e)=max\{o_k\in O(G)| e\rightarrow o_k\}.$$

\begin{prop}\label{lem 2}

$(1)$ For any  $i\in I(G)$ and  $e\in E(G)-I(G)$, we have
 $i^-(e)\preceq i\preceq i^+(e)\Longleftrightarrow i\rightarrow e.$

 $(2)$ For any $o\in O(G)$ and $e\in E(G)-O(G)$, we have
 $o^-(e)\preceq o\preceq o^+(e)\Longleftrightarrow e\rightarrow o.$
\end{prop}
We can get a graphical representation of $(1)$ by replacing the labels $e_1$, $e_2$ and $e'$ in Fig \ref{16} with $i^-(e)$, $i^+(e)$ and $i$, respectively.

\begin{proof}We only prove $(1)$. The proof of $(2)$ is similar and we omit it here.
The direction ($\Longleftarrow$) is obvious. Now we show the direction
($\Longrightarrow$). First, $i\in I(G)$ implies that $i^-(e)\nrightarrow i$. If $i=i^-(e)$ or $i^+(e)$, then $i\rightarrow e$. Otherwise, $i^-(e)\prec i\prec i^+(e)$, then $i\rightarrow e$ follows from Lemma \ref{lem 1} $(1)$.
\end{proof}

The following result shows a characterization of basic intervals.
\begin{thm}\label{lem 3}

$(1)$ Let $e\in E(G)-I(G)$. Then
$e\in P_k\Longleftrightarrow i^+(e)=i_k,\ \ \  (1\leq k\leq m).$

$(2)$ Let $e\in E(G)-O(G)$.  Then
$e\in Q_k\Longleftrightarrow o^-(e)=o_k,\ \ \  (1\leq k\leq n).$
\end{thm}

\begin{proof}
$(1)$ ($\Longleftarrow$). Assume $i^+(e)=i_k$, then by $(P1)$, $i_k=i^+(e)\prec e$. We have two cases. If $k=m$, then $e\in (i_m, +\infty]$ and the proof is completed.

Now we assume that $1\leq k\leq n-1$. It suffices to show $e\prec i_{k+1}$. Otherwise, $i_{k+1}\prec e$, and hence $i_k\prec i_{k+1} \prec e$. Then by ($P2$),  $i_k\rightarrow e$ implies that either $i_k\rightarrow i_{k+1}$ or $i_{k+1}\rightarrow e$ (see Fig \ref{17}), both will lead to a contradiction. Thus we must have $e\prec i_{k+1}$, and hence $e\in (i_k, i_{k+1})$.

\begin{figure}[H]
\centering
$$
\begin{matrix}
\begin{matrix}
\begin{tikzpicture}

\node (v1) at (-1.2,2.2) {};
\draw[fill] (-1.2,2.2) circle [radius=0.07];
\node (v2) at (-1.2,1.4) {};
\draw[fill]  (-1.2,1.4) circle [radius=0.07];
\node (v3) at (-1.8,0) {};
\draw[fill] (-1.8,0) circle [radius=0.07];
\node (v4) at (-1.8,-0.8) {};
\draw[fill] (-1.8,-0.8) circle [radius=0.07];
\draw  plot[smooth, tension=.7] coordinates {(-1.2,2.2)(-1.2,1.4) (-1.3,1.1) (-1.5,0.7) (-1.7,0.4) (-1.8,0) (-1.8,-0.8)}[postaction={decorate, decoration={markings,mark=at position 0.5 with {\arrow[black]{stealth}}}}];
\draw  (-1.2,2.2)  -- (-1.2,1.4)[postaction={decorate, decoration={markings,mark=at position 0.5 with {\arrow[black]{stealth}}}}];
\draw  (-1.8,0)  -- (-1.8,-0.8)[postaction={decorate, decoration={markings,mark=at position 0.5 with {\arrow[black]{stealth}}}}];
\node (v5) at (0.1,2.2) {};
\draw[fill] (0.1,2.2) circle [radius=0.07];
\node (v6) at (0.1,1.4) {};
\draw[fill] (0.1,1.4) circle [radius=0.07];
\draw  (0.1,2.2)  -- (0.1,1.4)[postaction={decorate, decoration={markings,mark=at position 0.5 with {\arrow[black]{stealth}}}}];
\draw [dashed] plot[smooth, tension=.7] coordinates {(v2) (-1,1.2) (-0.6,1.2) (-0.4,1.6) (-0.3,2) (v5)}[postaction={decorate, decoration={markings,mark=at position 0.45 with {\arrow[black]{stealth}}}}];
\node at (-1.5,1.8) {$i_k$};
\node at (-2.1,-0.5) {$e$};
\node at (0.6,1.8) {$i_{k+1}$};
\node [scale=1.5]at (-0.4,1.6) {$\times$};
\node (v7) at (-2.8,2.2) {};
\node (v8) at (1.8,2.2) {};
\draw  [loosely dashed](v7) -- (v8);
\end{tikzpicture}
\end{matrix}&&&
\begin{matrix}
\begin{tikzpicture}

\node (v1) at (-1.2,2.2) {};
\draw[fill] (-1.2,2.2) circle [radius=0.07];
\node (v2) at (-1.2,1.4) {};
\draw[fill]  (-1.2,1.4) circle [radius=0.07];
\node (v3) at (-1.8,0) {};
\draw[fill] (-1.8,0) circle [radius=0.07];
\node (v4) at (-1.8,-0.8) {};
\draw[fill] (-1.8,-0.8) circle [radius=0.07];
\draw  plot[smooth, tension=.7] coordinates {(-1.2,2.2)(-1.2,1.4) (-1.3,1.1) (-1.5,0.7) (-1.7,0.4) (-1.8,0) (-1.8,-0.8)}[postaction={decorate, decoration={markings,mark=at position 0.5 with {\arrow[black]{stealth}}}}];
\draw  (-1.2,2.2)  -- (-1.2,1.4)[postaction={decorate, decoration={markings,mark=at position 0.5 with {\arrow[black]{stealth}}}}];
\draw  (-1.8,0)  -- (-1.8,-0.8)[postaction={decorate, decoration={markings,mark=at position 0.5 with {\arrow[black]{stealth}}}}];
\node (v5) at (0.1,2.2) {};
\draw[fill] (0.1,2.2) circle [radius=0.07];
\node (v6) at (0.1,1.4) {};
\draw[fill] (0.1,1.4) circle [radius=0.07];
\draw  (0.1,2.2)  -- (0.1,1.4)[postaction={decorate, decoration={markings,mark=at position 0.5 with {\arrow[black]{stealth}}}}];
\node at (-1.5,1.8) {$i_k$};
\node at (-2.1,-0.5) {$e$};
\node at (0.6,1.8) {$i_{k+1}$};
\draw [dashed] plot[smooth, tension=.7] coordinates {(v6) (-0.2,1.1) (-0.6,0.9) (-0.9,0.8) (-1.2,0.6) (-1.5,0.3) (v3)}[postaction={decorate, decoration={markings,mark=at position 0.5 with {\arrow[black]{stealth}}}}];
\node [scale=1.5]at (-0.6,0.9) {$\times$};
\node (v7) at (-2.8,2.2) {};
\node (v8) at (1.8,2.2) {};
\draw  [loosely dashed](v7) -- (v8);
\end{tikzpicture}
\end{matrix}
\end{matrix}
$$
\caption{}
\label{17}
\end{figure}
($\Longrightarrow$). We just use the fact that $P_i\cap P_j=\emptyset$ for any $i\neq j$. Assume $e\in P_k$ and $i^+(e)= i_l$ for some $1\le k,l\le m$. By the proof of ($\Longleftarrow$) we know that $e\in P_l$. It forces that $k=l$, which completes the proof.

$(2)$ The proof is similar and we omit it here.
\end{proof}

The following result shows that a maximal internal vertex can be cut down from a POP-graph, just as that in Proposition \ref{Thm6}.
\begin{prop}\label{lem 4}
Let $v\in V_{int}(G)$ be a maximal vertex, that is, there is no vertex $v'\in V_{int}(G)$ such that $v\rightarrow v'$. Then

$(1)$ $O(v)$ is a subset of $O(G)$ and is an interval of $(E(G), \prec)$. In particular, $O(v)$ is an interval of $(O(G),\prec)$.

$(2)$ for any $h\in I(v)$ and $o\in O(G)-O(v)$, we have
$o\prec h \Longleftrightarrow o\prec min\ O(v)$, and $h\prec o \Longleftrightarrow max\ O(v)\prec o.$
\end{prop}

Fig \ref{18} shows an example of this proposition.
\begin{figure}[htbp]
\centering
\begin{tikzpicture}[scale=0.4]

\node (v2) at (-4,3) {};
\node (v1) at (-1.5,5.5) {};
\node (v7) at (-1.5,1) {};
\node [right]at (-1.5,1) {$v$};
\node (v9) at (1.5,5.5) {};
\node (v14) at (2,1.5) {};
\node [scale=0.7](v3) at (-3,7.5) {$2$};
\node [scale=0.7](v4) at (-2,7.5) {$3$};
\node [scale=0.7](v5) at (-0.5,7.5) {$4$};
\node[scale=0.7] (v6) at (-4.8,7.4) {$1$};
\node [scale=0.7](v11) at (-4.5,-1) {$7$};
\node [scale=0.7](v12) at (-2,-1) {$13$};
\node [scale=0.7](v13) at (0,-1) {$14$};
\node [scale=0.7](v15) at (2,-1) {$16$};
\node [scale=0.7](v8) at (1,7.5) {$10$};
\node [scale=0.7](v10) at (2.5,7.5) {$11$};
\node [scale=0.7] at (-2.5,3.5) {$6$};
\node[scale=0.7]  at (-3,5.2) {$5$};
\node [scale=0.7] at (-1.2,3.3) {$9$};
\node [scale=0.7] at (0.5,3.25) {$12$};
\node [scale=0.7] at (2.2,3.7) {$15$};
\node [scale=0.7] at (-3,1.7) {$8$};
\draw[fill] (-4,3) circle [radius=0.11];
\draw[fill] (v1) circle [radius=0.11];
\draw[fill] (v7) circle [radius=0.11];
\draw[fill] (v9) circle [radius=0.11];
\draw[fill] (v14) circle [radius=0.11];
\draw  plot[smooth, tension=1] coordinates {(v1) (-2.5,5)  (-3.5,4) (v2)}[postaction={decorate, decoration={markings,mark=at position .5 with {\arrow[black]{stealth}}}}];
\draw  plot[smooth, tension=1] coordinates {(v1) (-2,4.5)  (-3,3.5) (v2)}[postaction={decorate, decoration={markings,mark=at position .5 with {\arrow[black]{stealth}}}}];

\draw  (v3) -- (-1.5,5.5)[postaction={decorate, decoration={markings,mark=at position .5 with {\arrow[black]{stealth}}}}];
\draw  (v4) -- (-1.5,5.5)[postaction={decorate, decoration={markings,mark=at position .5 with {\arrow[black]{stealth}}}}];
\draw  (v5) -- (-1.5,5.5)[postaction={decorate, decoration={markings,mark=at position .5 with {\arrow[black]{stealth}}}}];

\draw  (v6) -- (-4,3)[postaction={decorate, decoration={markings,mark=at position .5 with {\arrow[black]{stealth}}}}];
\draw  (-1.5,5.5) -- (-1.5,1)[postaction={decorate, decoration={markings,mark=at position .5 with {\arrow[black]{stealth}}}}];
\draw  (-4,3) -- (-1.5,1)[postaction={decorate, decoration={markings,mark=at position .5 with {\arrow[black]{stealth}}}}];

\draw  (v8)--(1.5,5.5)[postaction={decorate, decoration={markings,mark=at position .5 with {\arrow[black]{stealth}}}}];
\draw  (v10) -- (1.5,5.5)[postaction={decorate, decoration={markings,mark=at position .5 with {\arrow[black]{stealth}}}}];
\draw  (1.5,5.5) -- (-1.5,1)[postaction={decorate, decoration={markings,mark=at position .5 with {\arrow[black]{stealth}}}}];
\draw  (-4,3) -- (v11)[postaction={decorate, decoration={markings,mark=at position .5 with {\arrow[black]{stealth}}}}];
\draw  (-1.5,1) -- (v12)[postaction={decorate, decoration={markings,mark=at position .65 with {\arrow[black]{stealth}}}}];
\draw  (v13) -- (-1.5,1)[postaction={decorate, decoration={markings,mark=at position .5 with {\arrowreversed[black]{stealth}}}}];
\draw  (1.5,5.5) -- (2,1.5)[postaction={decorate, decoration={markings,mark=at position .5 with {\arrow[black]{stealth}}}}];
\draw  (2,1.5) -- (v15)[postaction={decorate, decoration={markings,mark=at position .5 with {\arrow[black]{stealth}}}}];

\node (v20) at (4.5,3.25) {};

\node[scale=0.7] (v21) at (4.5,7.5) {$17$};

\node [scale=0.7](v22) at (4.5,-1) {$18$};
\draw[fill] (v20) circle [radius=0.11];
\draw  (v22) -- (4.5,3.25)[postaction={decorate, decoration={markings,mark=at position .5 with {\arrowreversed[black]{stealth}}}}];
\draw  (4.5,3.24) -- (v21)[postaction={decorate, decoration={markings,mark=at position .5 with {\arrowreversed[black]{stealth}}}}];

\node [scale=0.7](v23) at (6.5,7.5) {$19$};
\node (v24) at (6.5,-1) {};
\draw  (v23) -- (v24)[postaction={decorate, decoration={markings,mark=at position .5 with {\arrow[black]{stealth}}}}];

\draw [dotted]   (-6.2,0)-- (-2.8,0);
\draw [dotted]  (-2.8,0) edge (-2.8,1.6);
\draw [dotted]  (-2.8,1.6) edge (0.2,1.6);
\draw [dotted]  (0.2,1.6) edge (0.2,0.2);
\draw [dotted]  (0.2,0.2) edge (8,0.2);
\end{tikzpicture}

\caption{}
\label{18}
\end{figure}

\begin{proof}
$(1)$ Since $v$ is maximal, then $O(v)\subseteq O(G)$. We prove by contradiction that $O(v)$ is an interval of $(E(G),\prec)$.
Suppose there exist $e_1,e_2\in O(v)$ and an edge $e\in E(G)-O(v)$ such that $e_1\prec e \prec e_2$. Since $G$ is processive, $I(v)$ is nonempty.  Take $\widetilde{e}\in I(v)$, then  $\widetilde{e}\prec e_1\prec e\prec e_2$ and $\widetilde{e}\rightarrow e_2$. By $(P2)$, $\widetilde{e}\rightarrow e$ (the left of Fig \ref{19}) or $e\rightarrow e_2$ (the right of Fig \ref{19}). If $\widetilde{e}\rightarrow e$, then the maximality of $v$ implies that $e\in O(v)$, which contradicts $e\in E(G)-O(v)$. If $e\rightarrow e_2$, then $e\rightarrow e_1$. Then by $(P1)$,  $e\prec e_1$, which contradicts $e_1\prec e$.

$(2)$ Notice that $o^-(h)=\min  O(v)$ and $o^+(h)=\max  O(v)$. Moreover, the maximality of $v$ implies that $h\nrightarrow o$. Then the lemma is a direct consequence of Proposition \ref{lem 2} $(2)$.

\begin{figure}[htbp]
\centering
$$
\begin{matrix}
\begin{matrix}
\begin{tikzpicture}[scale=1.3]
\node (v1) at (-1.7,-2.7) {};
\node (v2) at (1.7,-2.7) {};
\draw  [loosely dashed](v1)--(v2);
\node (v4) at (-0.6,-1.3) {};
\draw[fill] (v4) circle [radius=0.04];
\node (v5) at (-1.1,-2.7) {};

\node (v6) at (0.7,-2.7) {};

\node (v3) at (-1.1,-0.4) {};
\draw[fill] (v3) circle [radius=0.04];
\draw  (-1.1,-0.4) -- (-0.6,-1.3)[postaction={decorate, decoration={markings,mark=at position .5 with {\arrow[black]{stealth}}}}];
\draw  (-0.6,-1.3) -- (-1.1,-2.7)[postaction={decorate, decoration={markings,mark=at position .5 with {\arrow[black]{stealth}}}}];
\draw  (-0.6,-1.3) node (v9) {} -- (0.7,-2.7)[postaction={decorate, decoration={markings,mark=at position .5 with {\arrow[black]{stealth}}}}];
\node (v7) at (-0.5,-2.1) {};
\draw[fill] (v7) circle [radius=0.04];
\node (v8) at (-0.6,-2.5) {};
\draw[fill] (v8) circle [radius=0.04];
\draw  (-0.5,-2.1) node (v10) {} -- (-0.6,-2.5)[postaction={decorate, decoration={markings,mark=at position .5 with {\arrow[black]{stealth}}}}];

\draw [dashed] plot[smooth, tension=.7] coordinates {(v9) (-0.55,-1.55) (-0.35,-1.8)(-0.35,-2) (v10)}[postaction={decorate, decoration={markings,mark=at position .35 with {\arrow[black]{stealth}}}}];
\node [scale=1.5]at (-0.33,-1.9) {$\times$};
\node at (-0.7,-0.7) {$\widetilde{e}$};
\node at (-1.1,-2.1) {$e_1$};
\node at (-0.7,-2.3) {$e$};
\node at (0.4,-2.2) {$e_2$};
\node at (-0.3,-1.1) {$v$};

\draw[fill] (-1.1,-2.7)circle [radius=0.04];
\draw[fill] (0.7,-2.7)circle [radius=0.04];
\end{tikzpicture}
\end{matrix}&&
\begin{matrix}
\begin{tikzpicture}[scale=1.3]
\node (v1) at (-1.7,-2.7) {};
\node (v2) at (1.7,-2.7) {};
\draw  [loosely dashed](v1)--(v2);
\node (v4) at (-0.6,-1.3) {};
\draw[fill] (v4) circle [radius=0.04];
\node (v5) at (-1.1,-2.7) {};

\node (v6) at (0.7,-2.7) {};

\node (v3) at (-1.1,-0.4) {};
\draw[fill] (v3) circle [radius=0.04];
\draw  (-1.1,-0.4) -- (-0.6,-1.3)[postaction={decorate, decoration={markings,mark=at position .5 with {\arrow[black]{stealth}}}}];
\draw  (-0.6,-1.3) -- (-1.1,-2.7)[postaction={decorate, decoration={markings,mark=at position .5 with {\arrow[black]{stealth}}}}];
\draw  (-0.6,-1.3) node (v9) {} -- (0.7,-2.7)[postaction={decorate, decoration={markings,mark=at position .5 with {\arrow[black]{stealth}}}}];
\node (v7) at (-0.5,-2.1) {};
\draw[fill] (v7) circle [radius=0.04];
\node (v8) at (-0.6,-2.5) {};
\draw[fill] (v8) circle [radius=0.04];
\draw  (-0.5,-2.1) node (v10) {} -- (-0.6,-2.5)[postaction={decorate, decoration={markings,mark=at position .5 with {\arrow[black]{stealth}}}}];
\draw [dashed] plot[smooth, tension=.7] coordinates { (-0.6,-2.5) (-0.3,-2.4) (-0.2,-2.1) (-0.3,-1.9) (-0.4,-1.7) (-0.6,-1.3)}[postaction={decorate, decoration={markings,mark=at position .65 with {\arrow[black]{stealth}}}}];

\node at (-0.7,-0.7) {$\widetilde{e}$};
\node at (-1.1,-2.1) {$e_1$};
\node at (-0.7,-2.3) {$e$};
\node at (0.4,-2.2) {$e_2$};
\node at (-0.3,-1.1) {$v$};
\node[scale=1.5] at (-0.2,-2.1) {$\times$};
\draw[fill] (-1.1,-2.7)circle [radius=0.04];
\draw[fill] (0.7,-2.7)circle [radius=0.04];
\end{tikzpicture}
\end{matrix}
\end{matrix}
$$
\caption{}
\label{19}
\end{figure}

\end{proof}

Similarly, we have the following result.

\begin{prop}\label{lem 6}
Let $v\in V_{int}(G)$ be a minimal vertex, that is, there is no vertex $v'\in V_{int}(G)$ such that $v'\rightarrow v$. Then

$(1)$ $I(v)$ is a subset of $I(G)$ and is an interval of $(E(G), \prec)$. In particular, $I(v)$ is an interval of $(I(G),\prec)$.

$(2)$ for any  $h\in O(v)$ and $i\in I(G)-I(v)$, we have
$i\prec h \Longleftrightarrow i\prec min\ I(v)$, and $h\prec o \Longleftrightarrow max\ I(v)\prec i.$
\end{prop}

\section{The proof of Theorem \ref{crux}}
In this section, we will give a proof of a key result, Theorem \ref{crux}, which justifies the definition of composition of POP-graphs (Definition \ref{com}).

Given two POP-graphs $(G_1,\prec_1)$ and $(G_2,\prec_2)$, as before we set $G=G_2\circ G_1$, which is obviously processive, and assume $\prec=\prec_2\circ\prec_1=Q_1\triangleleft\{\overline{e_1}\}\triangleleft P_1\triangleleft...\triangleleft Q_k\triangleleft\{\overline{e_k}\}\triangleleft P_k\triangleleft...\triangleleft Q_n\triangleleft\{\overline{e_n}\}\triangleleft P_n.$ To prove Theorem \ref{crux}, we only need to show that $\prec_2\circ\prec_1$ is a planar order of $G=G_2\circ G_1$.

From definition, $(P1)$ is clear for $\prec$. Since $\prec$ is a linear order, it is easy to see that $(P2)$ is equivalent to $(\widetilde{P2})$ that for any $e_1,e_2,e_3\in E(G)$, if $t(e_1)=s(e_3)$ and $e_1\prec e_2\prec e_3$, then $e_1\rightarrow e_2$ or $e_2\rightarrow e_3$, where $t(e_1)=s(e_3)$ means that $e_1e_3$ is a directed path of length two.

Now we show that $\prec$ satisfies $(\widetilde{P2})$. Assume $e_1,e_2,e_3\in E(G)$ with $t(e_1)=s(e_3)$ and  $e_1\prec e_2 \prec e_3$, we want to show case by case that either $e_1\rightarrow e_2$ or $e_2\rightarrow e_3$. Since $e_1e_3$ is a length two directed path, then by the construction of $G_2\circ G_1$ we have either $e_1,e_3\in \Big(E(G_1)-\{o_1,\cdots,o_n\}\Big)\sqcup \{\overline{e_1},\cdots,\overline{e_n}\}\backsimeq E(G_1)$ or $e_1,e_3\in \Big(E(G_2)-\{i_1,\cdots,i_n\}\Big)\sqcup \{\overline{e_1},\cdots,\overline{e_n}\}\backsimeq  E(G_2)$, where, for simplicity, we freely identify $\overline{e_k}$ with $o_k$ or(and) $i_k$ for each $k$.

\noindent\textbf{Case 1}: $e_1,e_3\in E(G_1)$. There are two subcases.

   \textbf{Subcase 1.1}: $e_2\in E(G_1)$. By $(P2)$ of $\prec_1$, we have either $e_1\rightarrow e_2$ or $e_2\rightarrow e_3$ in $G_1$ and hence either $e_1\rightarrow e_2$ or $e_2\rightarrow e_3$ in $G$.

   \textbf{Subcase 1.2}: $e_2\not\in E(G_1)$, that is, $e_2\in E(G_2)-\{i_1,\cdots,i_n\}$. Assume that $o^-(e_1)=o_\mu$, $o^+(e_1)=o_\nu$ in $G_1$ and $i^+(e_2)=i_\lambda$ in $G_2$ for some $\mu,\nu,\lambda\in\{1,\cdots, n\}$ (see Fig \ref{20}), we want to show that $\mu\leq \lambda< \nu$, that is, $o_\mu\preceq_1 o_\lambda\prec_1 o_\nu$ in $G_1$.

\begin{figure}[H]
\centering
\begin{tikzpicture}[scale=0.7]
\node (v1) at (-0.6,2.4) {};
\draw[fill] (-0.6,2.4) circle [radius=0.07];
\node (v2) at (-0.6,1.2) {};
\draw[fill] (-0.6,1.2) circle [radius=0.07];
\node (v3) at (-0.6,-0.2) {};
\draw[fill] (-0.6,-0.2) circle [radius=0.07];
\draw (-0.6,2.4) -- (-0.6,1.2)[postaction={decorate, decoration={markings,mark=at position .5 with {\arrow[black]{stealth}}}}];
\draw (-0.6,1.2) -- (-0.6,-0.2)[postaction={decorate, decoration={markings,mark=at position .5 with {\arrow[black]{stealth}}}}];
\node (v4) at (-5,-2.6) {};
\node (v5) at (4.5,-2.6) {};
\node (v6) at (-5,-4.2) {};
\node (v7) at (4.5,-4.2) {};
\draw  [dotted](v4) edge (v5);
\draw [dotted] (v6) edge (v7);
\draw  plot[smooth, tension=.7] coordinates {(v2) (-1.8,0.8) (-2.8,0.2) (-3.2,-1) (-3.2,-2) (-3.2,-2.6) (-3.2,-4.2)}[postaction={decorate, decoration={markings,mark=at position .4 with {\arrow[black]{stealth}}}}][postaction={decorate, decoration={markings,mark=at position 0.87 with {\arrow[black]{stealth}}}}];
\draw  plot[smooth, tension=.7] coordinates {(-0.6,1.2) (1.4,0.4) (2.4,-0.6) (2.8,-1.6) (2.8,-2.6)(2.8,-4.2)}[postaction={decorate, decoration={markings,mark=at position .3 with {\arrow[black]{stealth}}}}][postaction={decorate, decoration={markings,mark=at position 0.88 with {\arrow[black]{stealth}}}}];

\node (v9) at (2.8,-4.2) {};
\node (v8) at (2.8,-2.6) {};

\node at (-3.2,-2) {};
\draw[fill] (-3.2,-2) circle [radius=0.07];
\node at (2.8,-1.6) {};
\draw[fill] (2.8,-1.6) circle [radius=0.07];
\node (v10) at (-2.2,-6.2) {};
\draw[fill] (-2.2,-6.2) circle [radius=0.07];
\node (v11) at (-2.2,-7.6) {};
\draw[fill] (-2.2,-7.6) circle [radius=0.07];
\draw  (-2.2,-6.2)-- (-2.2,-7.6)[postaction={decorate, decoration={markings,mark=at position 0.5 with {\arrow[black]{stealth}}}}];
\draw  plot[smooth, tension=.7] coordinates {(v3) (0.4,-0.4) (0.8,-1.2) (1,-2) (1,-2.6)(1,-4.2)}[postaction={decorate, decoration={markings,mark=at position .3 with {\arrow[black]{stealth}}}}][postaction={decorate, decoration={markings,mark=at position 0.828 with {\arrow[black]{stealth}}}}];
\node at (1,-2) {};
\draw[fill] (1,-2) circle [radius=0.07];
\node (v12) at (1,-2.6) {};
\node (v13) at (1,-4.2) {};

\draw  plot[smooth, tension=.7] coordinates {(v10) (-1.6,-6) (-1,-5.6) (-0.6,-5.2) (-0.4,-4.6) (-0.4,-4.2)(-0.4,-2.6)(-0.4,-1.8)}[postaction={decorate, decoration={markings,mark=at position .3 with {\arrowreversed[black]{stealth}}}}][postaction={decorate, decoration={markings,mark=at position 0.75 with {\arrowreversed[black]{stealth}}}}];
\draw[fill] (-0.4,-1.8) circle [radius=0.07];
\node at (-0.4,-4.6) {};
\draw[fill] (-0.4,-4.6) circle [radius=0.07];
\node (v14) at (-0.4,-2.6) {};
\node (v15) at (-0.4,-4.2) {};

\node at (-0.2,1.8) {$e_1$};
\node at (-0.2,0.4) {$e_3$};
\node at (-2.8,-3.4) {$\overline{e_\mu}$};
\node at (-0.8,-3.4) {$\overline{e_\lambda}$};
\node at (1.8,-3.0) {$o^+(e_3)$};
\node at (3.2,-3.4) {$\overline{e_\nu}$};
\node at (-2.6,-6.8) {$e_2$};
\node at (-3.6,-2.4) {$o_\mu$};
\node at (0,-4) {$i_\lambda$};
\node at (3.2,-2.4) {$o_\nu$};
\node at (3,1) {$G_1$};
\node at (3,-6.2) {$G_2$};
\node at (-0,-2.4) {$o_\lambda$};
\draw [dashed] plot[smooth, tension=.7] coordinates {(-0.6,1.2) (-1.2,0.6) (-1.6,0) (-1.8,-0.8) (-1.4,-1.6) (-0.4,-1.8)}[postaction={decorate, decoration={markings,mark=at position 0.5 with {\arrow[black]{stealth}}}}];
\end{tikzpicture}
\caption{}
\label{20}
\end{figure}

On one hand, notice that $e_1\in E(G_1)-\{o_1,\cdots,o_n\}$ (by $e_3\in E(G_1)$ and $e_1\rightarrow e_3$) and $e_2\in E(G_2)-\{i_1,\cdots,i_n\}$, so by Theorem \ref{lem 3}, $e_1\in Q_\mu$ and $e_2\in P_\lambda$. By the shuffle construction of $\prec$, $e_1\prec e_2$ implies that $\mu\leq\lambda$.

   On the other hand,  $e_2\prec e_3$ implies that $\lambda<\nu$. In fact, $\overline{e_\lambda}\prec e_2$ (by $i_\lambda\rightarrow e_2$ in $G_2$ or equivalently $\overline{e_\lambda}\rightarrow e_2$ in $G$ and $(P1)$) and $e_2\prec e_3$ imply that $\overline{e_\lambda}\prec e_3$ in $G$. Since $(G_1,\prec_1)$ is a POP-graph, by Proposition \ref{lem 2} $(2)$, $e_1\rightarrow o^+(e_3)$ (by $e_1\rightarrow e_3$ and $e_3\rightarrow o^+(e_3)$) implies that $o^+(e_3)\preceq_1 o_\nu$ in $G_1$. Note that $e_3\preceq_1 o^+(e_3)$ (by $e_3\rightarrow o^+(e_3)$ and $(P1)$), so $e_3\preceq_1 o_\nu$ in $G_1$ and equivalently $e_3\prec \overline{e_\nu}$ in $G$. Finally we get that  $\overline{e_\lambda}\prec e_3\preceq  \overline{e_\nu}$ in $G$, from which $\lambda<\nu$ follows.

   By the facts that $(G_1,\prec_1)$ is a POP-graph and Proposition \ref{lem 2} (2), $o_\mu\preceq_1 o_\lambda\prec_1 o_\nu$ implies that $e_1\rightarrow o_\lambda$ in $G_1$ (such a path is represented by the dashed curve in Fig \ref{20}). Combining with $i_\lambda\rightarrow e_2$ in $G_2$, we get that $e_1\rightarrow e_2$ in $G$.

\noindent\textbf{Case 2}: $e_1,e_3\in E(G_2)$. There are two subcases $e_2\in E(G_2)$ and  $e_2\not\in E(G_2)$, which are symmetric with \textbf{Subcase 1.2} and \textbf{Subcase 1.1}, respectively.

\section{Decomposition and cancellation}
In this section, we study some algebraic properties of tensor product and composition of POP-graphs.

The following result is a combinatorial counterpart of Proposition \ref{X}.
\begin{thm}\label{decomposition}
Any POP-graph $(G,\prec)$ has an elementary decomposition,  that is,   $(G,\prec)=(G_n,\prec_n)\circ\cdots \circ(G_1,\prec_1)$, with each $(G_k,\prec_k)$ $(1\leq k\leq n)$ being elementary.
\end{thm}

\begin{proof}

Let $(G, \prec)$ be a POP-graph and $v\in V_{int}(G)$ be a maximal vertex under $\rightarrow$. we will show that $(G, \prec)$ can be presented as a composition $(G_2, \prec_2)\circ(G_1, \prec_1)$ such that $G_2$ is elementary and  $V_{int}(G_2)=\{v\}$. Graphically, taking the POP-graph in Fig \ref{18} as an example, the idea is to cut it along the dotted line into two POP-graphs.

Definition of $(G_1, \prec_1)$:
$(1)$ $E(G_1)=E(G)-O(v)$;
$(2)$ $V(G_1)=\Big(V(G)-\{v\}-\{t(o)|o\in O(v)\}\Big)\sqcup \{t_{h}|h\in I(v)\}$;
$(3)$ for each $e\in E(G_1)-I(v)$, keep $s(e)$ and $t(e)$ unchanged; and for each $h\in I(v)$, keep $s(h)$ unchanged and set $t(h)=t_{h}$;
$(4)$ $\prec_1$ is the restriction of $\prec$.

Definition of  $(G_2, \prec_2)$:
$(1)$ $E(G_2)=O(G)\sqcup I(v)$;
$(2)$ $V(G_2)= \{v\}\sqcup\{t(o)\mid o\in O(G)\}\sqcup \{s_h\mid h\in \Big(O(G)-O(v)\Big)\cup I(v) \} $;
$(3)$ $t(h)$ is unchanged for any $h\in E(G_2)$; $s(h)= s_h$ for $h\in \Big(O(G)-O(v)\Big)\cup I(v)$, and $s(o)= v$ for any $o\in O(v)$;
$(4)$ $\prec_2$ is the restriction of $\prec$.

The fact that $(G, \prec)=(G_2, \prec_2)\circ(G_1, \prec_1)$ can be directly checked.
\end{proof}

The following result is a combinatorial counterpart of Proposition \ref{Y}.
\begin{thm}\label{decomposition 2}
Any elementary POP-graph $(G,\prec)$ has a unique primary decomposition, that is, $(G,\prec)=(G_n,\prec_n)\otimes\cdots \otimes(G_1,\prec_1)$, with each $(G_k,\prec_k)$ $(1\leq k\leq n)$ being prime or unitary.
\end{thm}
\begin{proof}
This follows from the fact that for any processive vertex $v$ of $G$, the set $E(v)$ of incident edges of $v$ is an interval of $(E(G),\prec)$.  In fact, if we assume $\overline{E(v)}=[h_1, h_2]$, then by $(P1)$, we must have $t(h_1)=v=s(h_2)$. For any $e\in [h_1,h_2]$, by $(P2)$, we must have either $t(e)=v$ or $s(e)=v$, hence $e\in E(v)$.
\end{proof}

The following result aims to prove that the composition satisfies cancellation law.
\begin{prop}\label{lem 8}
Let $(H,\prec)$ be an elementary POP-graph with exact one internal vertex. Then

$(1)$ $(H,\prec)\circ(G_1,\prec_1)=(H,\prec)\circ(G_2,\prec_2)$ implies that $(G_1,\prec_1)=(G_2,\prec_2)$.

$(2)$ $(G_1,\prec_1)\circ(H,\prec)=(G_2,\prec_2)\circ(H,\prec)$ implies that $(G_1,\prec_1)=(G_2,\prec_2)$.
\end{prop}
\begin{proof}
We only prove $(1)$, $(2)$ is similar. Let $v$ be the unique internal vertex of $H$, which is of course minimal. By Proposition \ref{lem 6} $(1)$, $I(v)$ is an interval of $(I(H),<)$.
Since $H$ is elementary, by Proposition \ref{lem 6} $(2)$, we can assume $\prec=[i_1,i_{K-1}]\triangleleft I(v)\triangleleft O(v)\triangleleft[i_{L+1},i_n]$, $I(v)=[i_K,\cdots,i_L]$, as shown in Fig \ref{21}.

\begin{figure}[htbp]
\centering
\begin{tikzpicture}[scale=2]
\draw  [loosely dashed](0.5,-0.5) rectangle (4.5,-1.5);

\node at (1.3,-1) {$[i_1,i_{K-1}]$};
\node at (2.7,-0.8) {$I(v)=[i_K,i_L]$};
\node at (2.7,-1.2) {$O(v)$};
\node at (4,-1) {$[i_{L+1},i_n]$};

\node (v1) at (2,-0.5) {};
\node (v2) at (2,-1.5) {};
\node (v3) at (3.5,-0.5) {};
\node (v4) at (3.5,-1.5) {};
\node (v5) at (2,-1) {};
\node (v6) at (3.5,-1) {};
\draw [loosely dashed] (2,-0.5) -- (2,-1.5);
\draw [loosely dashed](3.5,-0.5) --(3.5,-1.5);
\draw[loosely dashed]  (2,-1)-- (3.5,-1);
\end{tikzpicture}
\caption{}
\label{21}
\end{figure}
Assume
$\prec_{1}=Q_1'\triangleleft\{o_1'\}\triangleleft\cdots\triangleleft Q'_n\triangleleft\{o'_{n}\}$ and $\prec_{2}=Q_1''\triangleleft\{o_1''\}\triangleleft\cdots\triangleleft Q''_n\triangleleft\{o''_{n}\}.$
Then
$$\prec\circ\prec_{1}=Q_1'\triangleleft\{\overline{e_1}'\}\triangleleft \cdots\triangleleft Q'_L\triangleleft \{\overline{e_{L}}'\} \triangleleft O(v)\triangleleft  Q'_{L+1}\triangleleft\{\overline{e_{L+1}}'\}\triangleleft\cdots\triangleleft Q'_n\triangleleft\{\overline{e_n}'\},$$  $$\prec\circ\prec_{2}=Q_1''\triangleleft\{\overline{e_1}''\}\triangleleft \cdots\triangleleft Q''_L\triangleleft \{\overline{e_{L}}''\} \triangleleft O(v)\triangleleft  Q''_{L+1}\triangleleft\{\overline{e_{L+1}}''\}\triangleleft\cdots\triangleleft Q''_n\triangleleft\{\overline{e_n}''\}.$$

Recall that $(H,\prec)\circ(G_1,\prec_1)=(H,\prec)\circ(G_2,\prec_2)$ means that there exist bijections $\phi:E(H\circ G_1)\rightarrow E(H\circ G_2)$ and $\psi:V(H\circ G_1)\rightarrow V(H\circ G_2)$, which preserve the adjacency relations and the planar orders. Then to show that the restriction of $\phi,\psi$ induce an isomorphism of $(G_1,\prec_1)$ and $(G_2,\prec_2)$, we only need to show that $\phi$ preserves the shuffle structures of $\prec\circ\prec_{1}$ and $\prec\circ\prec_{2}$. To show this, we will prove that $\phi([\overline{e_1}',\overline{e_{K-1}}'])=[\overline{e_1}'',\overline{e_{K-1}}'']$, $\phi(I(v))=I(v)$, $\phi(O(v))=O(v)$, $\phi([\overline{e_{L+1}}',\overline{e_{n}}'])=[\overline{e_{L+1}}'',\overline{e_n}'']$.

In fact, consider the sets of output edges, we have $$O(H\circ G_1)=[\overline{e_1}',\overline{e_{K-1}}']\triangleleft O(v)\triangleleft [\overline{e_{L+1}}',\overline{e_{n}}']
,$$ $$O(H\circ G_2)=[\overline{e_1}'',\overline{e_{K-1}}'']\triangleleft O(v)\triangleleft[\overline{e_{L+1}}'',\overline{e_{n}}''].$$ Since $\phi$ induces a bijection between $O(H\circ G_1)$ and $O(H\circ G_1)$, by counting the number of elements, we must have $\phi([\overline{e_1}',\overline{e_{K-1}}'])=[\overline{e_1}'',\overline{e_{K-1}}'']$, $\phi(O(v))=O(v)$, $\phi([\overline{e_{L+1}}',\overline{e_{n}}'])=[\overline{e_{L+1}}'',\overline{e_n}'']$, where the second fact implies that $\phi(I(v))=I(v)$.

\end{proof}

The following result is a direct consequence of Theorem \ref{decomposition} and Proposition \ref{lem 8}.
\begin{thm}\label{cancel}
The composition  satisfies  cancellation law, that is, if $(G_2,\prec_2)\circ(G_1,\prec_1)=(G_2',\prec_2') \circ(G_1',\prec_1')$, then $(G_1,\prec_1)= (G_1',\prec_1')$ implies that $(G_2,\prec_2)=(G_2',\prec_2')$ and $(G_2,\prec_2)=(G_2',\prec_2')$ implies that $(G_1,\prec_1)=(G_1',\prec_1')$.
\end{thm}

The cancellation law for tensor product is obvious.

\section{Freeness of $\mathcal{POP}$}
In this section, we want to prove Theorem \ref{free2}, that is, to show the freeness of $\mathcal{POP}$. For this, we only need to show that for any \textbf{POP-diagram} on a POP-graph $(G,\prec)$ in a monoidal category, its \textbf{value} can be defined and is independent of the decompositions of $(G,\prec)$ (see Definition \ref{pop-diagram} for the definition of a POP-diagram in a monoidal category, which  is similar as that in Joyal and Street's framework \cite{[JS91]}). By Theorem \ref{decomposition} and \ref{decomposition 2}, we can always define a value for a POP-diagram in a monoidal category. Then the only thing we are left to show is that the value of a POP-diagram in a monoidal category is independent of its decompositions.

We use induction on $|V_{int}(G)|$. If $|V_{int}(G)|=1$,  $(G,\prec)$ is elementary and by Theorem \ref{decomposition 2}, has a unique primary decomposition, and therefore the value of any diagram on $(G,\prec)$  is unique.
Assume that the uniqueness of the value is true for $|V_{int}(G)|< n$, we want to show that the uniqueness of the value is also true for $|V_{int}(G)|= n$.
Assume $(G_n,\prec_n)\circ\cdots\circ (G_1,\prec_1)$ and  $(G'_n,\prec'_n)\circ\cdots\circ (G'_1,\prec'_1)$ be two elementary decompositions of $(G,\prec)$ with $V_{int}(G_i)=\{v_i\}$ and $V_{int}(G'_i)=\{v'_i\}$ ($1\leq i\leq n$). Clear, we must have $v_n=v_l'$ for some $l\in[1,\cdots,n]$.

If $n=l$,  we must have $(G_n,\prec_n)=(G'_n,\prec_n')$ (consider the proof of Proposition \ref{lem 8}),  and by Theorem \ref{cancel}, $(G_{n-1},\prec_{n-1})\circ\cdots\circ (G_1,\prec_1)=(G'_{n-1},\prec'_{n-1})\circ\cdots\circ (G'_1,\prec'_1)$.
Then by the induction hypothesis,  the values of diagrams on $(G_n,\prec_n)$ and $(G_{n-1},\prec_{n-1})\circ\cdots\circ (G_1,\prec_1)$ are equal to the values of diagrams $(G'_n,\prec_n')$ and $(G'_{n-1},\prec'_{n-1})\circ\cdots\circ (G'_1,\prec'_1)$, respectively. Compose the values in $\mathcal{V}$, we obtain the unique value.

Otherwise, the proof is reduced to the simple claim that: if $(G,\prec)$ is elementary with $V_{int}(G)=\{v_1,v_2\}$, then it has exactly two elementary decompositions $(G_2,\prec_2)\circ (G_1,\prec_1)$, $(G'_2,\prec'_2)\circ (G'_1,\prec'_1)$ such that $V_{int}(G_1)=V_{int}(G'_2)=\{v_1\}$ and $V_{int}(G'_1)=V_{int}(G_2)=\{v_2\}$, see Fig \ref{22} for an example. Clearly, the values of any diagram with respect to the two decompositions are equal.
\begin{figure}[h]
\centering
$$
\begin{matrix}
\begin{matrix}
\begin{tikzpicture}[scale=0.9]
\draw [loosely dashed] (-1.5,0.5) rectangle (5,-2.5);
\node (v1) at (-1.5,-1) {};
\node (v2) at (5,-1) {};
\draw  [loosely dashed](-1.5,-1)--(5,-1);

\node (v3) at (-1,0.5) {};
\node (v4) at (-1,-2.5) {};
\node (v5) at (-0.6,0.5) {};
\node (v6) at (-0.6,-2.5) {};
\draw   (-1,0.5) -- (-1,-2.5)[postaction={decorate, decoration={markings,mark=at position 0.25 with {\arrow[black]{stealth}}}}][postaction={decorate, decoration={markings,mark=at position 0.75 with {\arrow[black]{stealth}}}}];
\draw  (-0.6,0.5) -- (-0.6,-2.5)[postaction={decorate, decoration={markings,mark=at position 0.25 with {\arrow[black]{stealth}}}}][postaction={decorate, decoration={markings,mark=at position 0.75 with {\arrow[black]{stealth}}}}];

\node (v8) at (0.4,-0.3) {};
\draw[fill] (v8) circle [radius=0.07];
\node [left] at (0.4,-0.3) {$v_1$};
\node (v7) at (0,0.5) {};
\node (v9) at (0.8,0.5) {};
\node (v10) at (0,-1) {};
\node (v11) at (0.5,-1) {};
\node (v12) at (0.9,-1) {};
\node (v13) at (0,-2.5) {};
\node (v14) at (0.5,-2.5) {};
\node (v15) at (0.9,-2.5) {};
\draw (0,0.5)  --(0.4,-0.3)[postaction={decorate, decoration={markings,mark=at position 0.5 with {\arrow[black]{stealth}}}}];
\draw  (0.8,0.5) -- (0.4,-0.3)[postaction={decorate, decoration={markings,mark=at position 0.5 with {\arrow[black]{stealth}}}}];
\draw  (0.4,-0.3) --  (0,-1)[postaction={decorate, decoration={markings,mark=at position 0.65 with {\arrow[black]{stealth}}}}];
\draw  (0.4,-0.3) -- (0.5,-1)[postaction={decorate, decoration={markings,mark=at position 0.65 with {\arrow[black]{stealth}}}}];
\draw  (0.4,-0.3) -- (0.9,-1)[postaction={decorate, decoration={markings,mark=at position 0.65 with {\arrow[black]{stealth}}}}];
\draw  (0,-1) -- (0,-2.5)[postaction={decorate, decoration={markings,mark=at position 0.5 with {\arrow[black]{stealth}}}}];
\draw  ((0.5,-1) -- (0.5,-2.5)[postaction={decorate, decoration={markings,mark=at position 0.5 with {\arrow[black]{stealth}}}}];
\draw  (0.9,-1) -- (0.9,-2.5)[postaction={decorate, decoration={markings,mark=at position 0.5 with {\arrow[black]{stealth}}}}];

\node (v16) at (1.5,0.5) {};
\node (v17) at (1.5,-2.5) {};
\draw  (1.5,0.5) -- (1.5,-2.5)[postaction={decorate, decoration={markings,mark=at position 0.25 with {\arrow[black]{stealth}}}}][postaction={decorate, decoration={markings,mark=at position 0.75 with {\arrow[black]{stealth}}}}];

\node (v22) at (2.5,-1.8) {};
\draw[fill] (v22) circle [radius=0.07];
\node [left] at (2.5,-1.8) {$v_2$};
\node (v23) at (2.2,-2.5) {};
\node (v24) at (2.8,-2.5) {};
\node (v19) at (2.1,-1) {};
\node (v21) at (2.9,-1) {};
\node (v18) at (2.1,0.5) {};
\node (v20) at (2.9,0.5) {};
\draw  (2.1,0.5) -- (2.1,-1)[postaction={decorate, decoration={markings,mark=at position 0.5 with {\arrow[black]{stealth}}}}];
\draw  (2.9,0.5) --  (2.9,-1)[postaction={decorate, decoration={markings,mark=at position 0.5 with {\arrow[black]{stealth}}}}];
\draw (2.1,-1) -- (2.5,-1.8)[postaction={decorate, decoration={markings,mark=at position 0.5 with {\arrow[black]{stealth}}}}];
\draw   (2.9,-1) -- (2.5,-1.8)[postaction={decorate, decoration={markings,mark=at position 0.5 with {\arrow[black]{stealth}}}}];
\draw  (2.5,-1.8) -- (2.2,-2.5)[postaction={decorate, decoration={markings,mark=at position 0.65 with {\arrow[black]{stealth}}}}];
\draw  (2.5,-1.8) -- (2.8,-2.5)[postaction={decorate, decoration={markings,mark=at position 0.65 with {\arrow[black]{stealth}}}}];

\node (v25) at (3.5,0.5) {};
\node (v26) at (3.5,-2.5) {};
\node (v27) at (4.1,0.5) {};
\node (v28) at (4.1,-2.5) {};
\node (v29) at (4.6,0.5) {};
\node (v30) at (4.6,-2.5) {};
\draw  (3.5,0.5) -- (3.5,-2.5)[postaction={decorate, decoration={markings,mark=at position 0.25 with {\arrow[black]{stealth}}}}][postaction={decorate, decoration={markings,mark=at position 0.75 with {\arrow[black]{stealth}}}}];
\draw  (4.1,0.5) --  (4.1,-2.5)[postaction={decorate, decoration={markings,mark=at position 0.25 with {\arrow[black]{stealth}}}}][postaction={decorate, decoration={markings,mark=at position 0.75 with {\arrow[black]{stealth}}}}];
\draw  (4.6,0.5) -- (4.6,-2.5)[postaction={decorate, decoration={markings,mark=at position 0.25 with {\arrow[black]{stealth}}}}][postaction={decorate, decoration={markings,mark=at position 0.75 with {\arrow[black]{stealth}}}}];

\end{tikzpicture}
\end{matrix}
&\begin{matrix}
=
\end{matrix}
&
\begin{matrix}
\begin{tikzpicture}[scale=0.9]
\draw [loosely dashed] (-1.5,0.5) rectangle (5,-2.5);
\node (v1) at (-1.5,-1) {};
\node (v2) at (5,-1) {};
\draw  [loosely dashed](-1.5,-1)--(5,-1);

\node (v3) at (-1,0.5) {};
\node (v4) at (-1,-2.5) {};
\node (v5) at (-0.6,0.5) {};
\node (v6) at (-0.6,-2.5) {};
\draw   (-1,0.5) -- (-1,-2.5)[postaction={decorate, decoration={markings,mark=at position 0.25 with {\arrow[black]{stealth}}}}][postaction={decorate, decoration={markings,mark=at position 0.75 with {\arrow[black]{stealth}}}}];
\draw  (-0.6,0.5) -- (-0.6,-2.5)[postaction={decorate, decoration={markings,mark=at position 0.25 with {\arrow[black]{stealth}}}}][postaction={decorate, decoration={markings,mark=at position 0.75 with {\arrow[black]{stealth}}}}];

\node (v16) at (1.5,0.5) {};
\node (v17) at (1.5,-2.5) {};
\draw  (1.5,0.5) -- (1.5,-2.5)[postaction={decorate, decoration={markings,mark=at position 0.25 with {\arrow[black]{stealth}}}}][postaction={decorate, decoration={markings,mark=at position 0.75 with {\arrow[black]{stealth}}}}];

\node (v25) at (3.5,0.5) {};
\node (v26) at (3.5,-2.5) {};
\node (v27) at (4.1,0.5) {};
\node (v28) at (4.1,-2.5) {};
\node (v29) at (4.6,0.5) {};
\node (v30) at (4.6,-2.5) {};
\draw  (3.5,0.5) -- (3.5,-2.5)[postaction={decorate, decoration={markings,mark=at position 0.25 with {\arrow[black]{stealth}}}}][postaction={decorate, decoration={markings,mark=at position 0.75 with {\arrow[black]{stealth}}}}];
\draw  (4.1,0.5) --  (4.1,-2.5)[postaction={decorate, decoration={markings,mark=at position 0.25 with {\arrow[black]{stealth}}}}][postaction={decorate, decoration={markings,mark=at position 0.75 with {\arrow[black]{stealth}}}}];
\draw  (4.6,0.5) -- (4.6,-2.5)[postaction={decorate, decoration={markings,mark=at position 0.25 with {\arrow[black]{stealth}}}}][postaction={decorate, decoration={markings,mark=at position 0.75 with {\arrow[black]{stealth}}}}];

\node (v11) at (0.4,-1.8) {};
\draw[fill] (v11) circle [radius=0.07];
\node[left]at (0.4,-1.8) {$v_1$};
\node (v12) at (0,-2.5) {};
\node (v13) at (0.5,-2.5) {};
\node (v14) at (1.1,-2.5) {};
\node (v8) at (0,-1) {};
\node (v10) at (0.8,-1) {};
\node (v7) at (0,0.5) {};
\node (v9) at (0.8,0.5) {};
\draw  (0,0.5) -- (0,-1)[postaction={decorate, decoration={markings,mark=at position 0.5 with {\arrow[black]{stealth}}}}];
\draw  (0.8,0.5) -- (0.8,-1)[postaction={decorate, decoration={markings,mark=at position 0.5 with {\arrow[black]{stealth}}}}];
\draw  (0,-1)-- (0.4,-1.8)[postaction={decorate, decoration={markings,mark=at position 0.5 with {\arrow[black]{stealth}}}}];
\draw  (0.8,-1) -- (0.4,-1.8)[postaction={decorate, decoration={markings,mark=at position 0.5 with {\arrow[black]{stealth}}}}];
\draw  (0.4,-1.8) -- (0,-2.5)[postaction={decorate, decoration={markings,mark=at position 0.65 with {\arrow[black]{stealth}}}}];
\draw (0.4,-1.8) -- (0.5,-2.5)[postaction={decorate, decoration={markings,mark=at position 0.65 with {\arrow[black]{stealth}}}}];
\draw  (0.4,-1.8)-- (1.1,-2.5)[postaction={decorate, decoration={markings,mark=at position 0.65 with {\arrow[black]{stealth}}}}];

\node (v18) at (2.5,-0.3) {};
\draw[fill] (v18) circle [radius=0.07];
\node [left] at (2.5,-0.3) {$v_2$};
\node (v15) at (2.1,0.5) {};
\node (v19) at (2.8,0.5) {};
\node (v20) at (2.1,-1) {};
\node (v21) at (2.8,-1) {};
\node (v22) at (2.1,-2.5) {};
\node (v23) at (2.8,-2.5) {};
\draw   (2.1,0.5) -- (2.5,-0.3)[postaction={decorate, decoration={markings,mark=at position 0.5 with {\arrow[black]{stealth}}}}];
\draw  (2.8,0.5)  -- (2.5,-0.3)[postaction={decorate, decoration={markings,mark=at position 0.5 with {\arrow[black]{stealth}}}}];
\draw  (2.5,-0.3) -- (2.1,-1)[postaction={decorate, decoration={markings,mark=at position 0.65 with {\arrow[black]{stealth}}}}];
\draw (2.5,-0.3) -- (2.8,-1)[postaction={decorate, decoration={markings,mark=at position 0.65 with {\arrow[black]{stealth}}}}];
\draw  (2.1,-1) -- (2.1,-2.5)[postaction={decorate, decoration={markings,mark=at position 0.5 with {\arrow[black]{stealth}}}}];
\draw  (2.8,-1) -- (2.8,-2.5)[postaction={decorate, decoration={markings,mark=at position 0.5 with {\arrow[black]{stealth}}}}];
\end{tikzpicture}
\end{matrix}
\end{matrix}
$$
\caption{}
\label{22}
\end{figure}

In fact, if $l<n$, then for any $k\in [l+1,n]$, both $v'_l\not\rightarrow v'_k$ and $v'_k\not\rightarrow v'_l$  hold. Then by the above claim, we can construct step by step a series of elementary decompositions of $(G,\prec)$ to exchange $v'_l$ with $v'_{l+1}$, $v'_{l+1}$, ..., $v_n'$ such that all the values of the diagram with respect to these decompositions are equal, where in the last step we use the result of the above case of $n=l$.

\section{Free constructions by POP-graphs}
In this section, we introduce the category $\mathbf{Semi.Ten}$ of semi-tensor schemes and show a construction of a free monoidal category on a semi-tensor scheme using our combinatorial framework.

We begin with some notations. For a set $S$, we denote the set of words in $S$ by $W(S)$, which can be viewed as a free monoid on $S$, and denote the set of non-empty words by $W^+(S)$, which can be viewed as a free semi-group on $S$. When $S$ is empty, $W(S)=\{\emptyset\}$ and $W^+(S)$ is empty, where $\emptyset$ denotes the empty word. Clearly, $W(S)=W^+(S)\sqcup \{\emptyset\}$.

\begin{defn}\label{T}
A \textbf{semi-tensor scheme} $\mathcal{D}$ consists of two (\textbf{possibly empty}) sets $Ob(\mathcal{D})$, $Mor(\mathcal{D})$ and two functions from $Mor(\mathcal{D})$ to $W^+(Ob(\mathcal{D}))$  $$s,t:Mor(\mathcal{D})\rightarrow W^+(Ob(\mathcal{D})),$$ which are called source and target maps, respectively.
\end{defn}
Clearly, if $Ob(\mathcal{D})$ is empty, then $Mor(\mathcal{D})$ must be empty, and in this case we say that $\mathcal{D}$ is empty. $\mathbf{PRM}$ and $\mathcal{PRM}$ can also be viewed as examples of this definition.

\begin{rem}
The notion of a semi-tensor scheme is different from that of a tensor scheme (see Definition \ref{T1}), which was first introduced by Joyal and Street in \cite{[JS88]} and was also called a \textbf{monoidal signatures} in \cite{[S11]}. A tensor scheme is a special \textbf{computad} \cite{[S76]} (also called \textbf{polygraph} \cite{[B93]}).
The main difference between them is that $W(Ob(\mathcal{D}))$  is replaced by $W^+(Ob(\mathcal{D}))$. Due to this change, Definition \ref{T} is \textbf{not} a special case of that of a computad (polygraph).
\end{rem}

\begin{defn}\label{mor}
A \textbf{morphism} $\varphi:\mathcal{D}_1\rightarrow\mathcal{D}_2$ of semi-tensor schemes consists of two functions $\varphi_o:Ob(\mathcal{D}_1)\rightarrow Ob(\mathcal{D}_2)$ and $\varphi_m:Mor(\mathcal{D}_1)\rightarrow Mor(\mathcal{D}_2)$ such that the diagram
$$\xymatrix{W^+(Ob(\mathcal{D}_1))\ar[d]_{\widehat{\varphi_o}}&\ar[d]^{\varphi_m}Mor(\mathcal{D}_1)\ar[r]^{t_1}\ar[l]_{s_1}&\ar[d]^{\widehat{\varphi_o}}W^+(Ob(\mathcal{D}_1))
\\W^+(Ob(\mathcal{D}_2))&\ar[l]_{s_2}\ar[r]^{t_2}Mor(\mathcal{D}_2)&W^+(Ob(\mathcal{D}_2))}$$
commutes, where $\widehat{\varphi_o}:W^+(Ob(\mathcal{D}_1))\rightarrow W^+(Ob(\mathcal{D}_2))$ is the natural extension of $\varphi_o$ which sends $x_1\cdots x_n$ to $\varphi_o(x_1)\cdots \varphi_o(x_n)$.
\end{defn}

As in Joyal and Street's work, there are two types of diagrams, one for semi-tensor schemes and the other for semi-groupal/monoidal categories.
\begin{defn}
A \textbf{POP-diagram} $\Gamma$ in \textbf{semi-tensor scheme} $\mathcal{D}$ consists of a POP-graph $(G,\prec)$ and two label functions $$\gamma_o: E(G)\rightarrow Ob(\mathcal{D}),\ \ \ \ \gamma_m:V_{int}(G)\rightarrow Mor(\mathcal{D})$$
such that for every internal vertex $v\in V_{int}(G)$,
 $$s(\gamma_m(v))=\gamma_o(h_1)\cdots \gamma_o(h_m),\ \ \ \ t(\gamma_m(v))=\gamma_o(h'_1)\cdots \gamma_o(h'_n),$$
where $h_1\prec\cdots\prec h_m$ and $h'_1\prec\cdots\prec h'_n$  are the ordered lists of edges in $I(v)$ and $O(v)$, respectively.
The \textbf{domain} and \textbf{codomain} of this POP-diagram are the non-empty \textbf{words}  $\gamma_o(i_1)\cdots \gamma_o(i_k)$ and $\gamma_o(o_1) \cdots \gamma_o(o_l)$  in $Ob(\mathcal{D})$, respectively, where $i_1\prec \cdots\prec i_k$ and $o_1\prec \cdots\prec o_l$ are the ordered lists of edges in $I(G)$ and $O(G)$, respectively.
\end{defn}

\begin{defn}\label{pop-diagram}
A \textbf{POP-diagram} $\Gamma$ in \textbf{semi-groupal/monoidal category} $\mathcal{S}$ consists of a POP-graph $(G,\prec)$ and two label functions $$\gamma_o: E(G)\rightarrow Ob(\mathcal{S}),\ \ \ \ \gamma_m:V_{int}(G)\rightarrow Mor(\mathcal{S})$$
such that for every internal vertex $v\in V_{int}(G)$,
 $$s(\gamma_m(v))=\gamma_o(h_1)\otimes\cdots\otimes \gamma_o(h_m),\ \ \ \ t(\gamma_m(v))=\gamma_o(h'_1)\otimes\cdots \otimes\gamma_o(h'_n),$$
where $h_1\prec\cdots\prec h_m$ and $h'_1\prec\cdots\prec h'_n$  are the ordered lists of edges in $I(v)$ and $O(v)$, respectively.
The \textbf{domain} and \textbf{codomain} of this POP-diagram are the non-empty \textbf{words} $\gamma_o(i_1)\cdots \gamma_o(i_k)$ and $\gamma_o(o_1) \cdots \gamma_o(o_l)$  in $Ob(\mathcal{S})$, respectively,
where $i_1\prec \cdots\prec i_k$ and $o_1\prec \cdots\prec o_l$ are the ordered lists of edges in $I(G)$ and $O(G)$, respectively.
\end{defn}

We write $\Gamma=[G,\prec,\gamma_o, \gamma_m]$. In the case that $(G,\prec)$ is unitary or invertible ($V_{int}(G)$ is empty), then it has only edge label and $\gamma_m$ is the unique function from the empty set (as an initial object in the category of sets) to $Mor(\mathcal{D})$ or $Mor(\mathcal{S})$.

 The set of POP-diagrams in $\mathcal{D}$ (or $\mathcal{S}$) is denoted as $Diag(\mathcal{D})$ (or $Diag(\mathcal{S})$). For each type of POP-diagrams, their tensor products and compositions are clear. For any semi-groupal category $\mathcal{S}$, there is a function from $Diag(\mathcal{S})$ to $Mor(\mathcal{S})$ sending a POP-diagram in $\mathcal{S}$ to its \textbf{value}.

A POP-diagram is called prime (unitary, invertible, elementary) if the underlying processive graph is prime (unitary, invertible, elementary). The set of prime POP-diagrams in $\mathcal{D}$ (or $\mathcal{S}$) is denoted as $Prim(\mathcal{D})$ (or $Prim(\mathcal{S})$).

It is easy to see that a morphism $\varphi:\mathcal{D}_1\rightarrow\mathcal{D}_2$ of semi-tensor schemes can induce a \textbf{pushforward} $\varphi_\ast:Diag(\mathcal{D}_1)\rightarrow Diag(\mathcal{D}_2)$, which sends a POP-diagram $[G,\prec, \gamma_o,\gamma_m]$ in $\mathcal{D}_1$ to a POP-diagram  $\varphi_\ast([G,\prec, \gamma_o,\gamma_m])=[G,\prec, \varphi(\gamma_o),\varphi(\gamma_m)]$ in $\mathcal{D}_2$. Clearly, $\varphi_\ast$ preserves tensor product and composition of POP-diagrams. Similarly, a semi-groupal functor $\theta:\mathcal{S}_1\rightarrow\mathcal{S}_2$ produces a pushforward $\theta_\ast$, which sends, just as $\varphi_\ast$, a POP-diagram in $\mathcal{S}_1$ to a POP-diagram in $\mathcal{S}_2$.

A POP-diagram is called prime (unitary, invertible, elementary) if its underlying POP-graph is prime (unitary, invertible, elementary).

\begin{ex}\label{pri}
Let $\mathcal{M}$ be a monoidal category with unit object $I$. Fig \ref{26} shows a prime POP-diagram in $\mathcal{M}$ with domain $IIXY$, codomain $UIV$, and with the unique internal vertex labelled by a morphism $f:X\otimes Y\rightarrow U\otimes V$.
\begin{figure}[H]
\centering
\begin{tikzpicture}[scale=0.6]
\node (v2) at (0,0) {};
\draw[fill] (v2) circle [radius=0.09];
\node[scale=0.6] at (2.3,0) {$f:X\otimes Y\rightarrow U\otimes V$};
\node [scale=0.6](v1) at (-1.5,1.5) {$I$};
\node [scale=0.6](v3) at (-0.5,1.5) {$I$};
\node [scale=0.6](v4) at (0.5,1.5) {$X$};
\node [scale=0.6](v5) at (1.5,1.5) {$Y$};
\node [scale=0.6](v6) at (-1,-1.5) {$U$};
\node[scale=0.6] (v7) at (0,-1.5) {$I$};
\node[scale=0.6] (v8) at (1,-1.5) {$V$};
\draw  (v1) -- (0,0)[postaction={decorate, decoration={markings,mark=at position 0.47 with {\arrow[black]{stealth}}}}];
\draw  (v3) -- (0,0)[postaction={decorate, decoration={markings,mark=at position 0.47 with {\arrow[black]{stealth}}}}];
\draw  (v4) -- (0,0)[postaction={decorate, decoration={markings,mark=at position 0.47 with {\arrow[black]{stealth}}}}];
\draw  (v5) -- (0,0)[postaction={decorate, decoration={markings,mark=at position 0.47 with {\arrow[black]{stealth}}}}];
\draw  (0,0)-- (v6)[postaction={decorate, decoration={markings,mark=at position 0.65 with {\arrow[black]{stealth}}}}];
\draw (0,0) -- (v7)[postaction={decorate, decoration={markings,mark=at position 0.65 with {\arrow[black]{stealth}}}}];
\draw  (0,0) -- (v8)[postaction={decorate, decoration={markings,mark=at position 0.65 with {\arrow[black]{stealth}}}}];
\end{tikzpicture}
\caption{}
\label{26}
\end{figure}

\end{ex}

In what follows, we list the concrete definitions of the adjunctions shown in Fig \ref{23}.

$(1)$ Definition of $\mathcal{F}^+$. For any semi-tensor scheme $\mathcal{D}$, $\mathcal{F}^+(\mathcal{D})$ is a semi-groupal category such that (\romannumeral1) $Ob(\mathcal{F}^+(\mathcal{D}))=W^+(Ob(\mathcal{D}))$; (\romannumeral2) $Mor(\mathcal{F}^+(\mathcal{D}))=Diag(\mathcal{D})$; (\romannumeral3) the source and target of a morphism is given by its domain and codomain, respectively; (\romannumeral4) the tensor product of objects is given by concatenation of words; (\romannumeral5) the tensor product and composition of morphisms are given by tensor product and composition of POP-diagrams, respectively. Clearly, if $\mathcal{D}$ is empty, then $\mathcal{F}^+(\mathcal{D})$ is empty.

For a morphism $\varphi:\mathcal{D}_1\rightarrow\mathcal{D}_2$ of semi-tensor schemes, $\mathcal{F}^+(\varphi):\mathcal{F}^+(\mathcal{D}_1)\rightarrow\mathcal{F}^+(\mathcal{D}_2)$ is a semi-groupal functor, which acts on objects as $\widehat{\varphi_o}$ (see Definition \ref{mor}) and on morphisms as $\varphi_\ast$.

$(2)$ Definition of $\mathcal{U}^+$. For any semi-groupal category $\mathcal{S}$, $\mathcal{U}^+(\mathcal{S})$ is a semi-tensor scheme such that (\romannumeral1) $Ob(\mathcal{U}^+(\mathcal{S}))=Ob(\mathcal{S})$; (\romannumeral2)  $Mor(\mathcal{U}^+(\mathcal{S}))=Prim(\mathcal{S})$; (\romannumeral3) the source and target of a morphism is given by its domain and codomain, respectively. Clearly, if $\mathcal{S}$ is empty, then $\mathcal{U}^+(\mathcal{S})$ is empty.

For a semi-groupal functor $\theta:\mathcal{S}_1\rightarrow\mathcal{S}_2$, $\mathcal{U}^+(\theta):\mathcal{U}^+(\mathcal{S}_1)\rightarrow\mathcal{U}^+(\mathcal{S}_2)$ is a morphism of semi-tensor schemes which acts on objects as $\theta$ and on morphisms as $\theta_\ast$.

$(3)$ Definition of $\mathfrak{F}$. For any  semi-groupal category $\mathcal{S}$, $\mathfrak{F}(\mathcal{S})$ is the monoidal category $\mathcal{S}+\mathbf{1}$, which is constructed from $\mathcal{S}$ by adjoining an \textbf{isolated object} $\mathbf{1}$ (see Remark \ref{iso} for this definition)  as a unit object, that is, (\romannumeral1) $Ob(\mathcal{S}+\mathbf{1})=Ob(\mathcal{S})\sqcup \{\mathbf{1}\}$ and $Mor(\mathcal{S}+\mathbf{1})=Mor(\mathcal{S})\sqcup \{Id_\mathbf{1}\}$; (\romannumeral2)  $\mathbf{1}\otimes\mathbf{1}=\mathbf{1}$ and for any $X\in Ob(\mathcal{S})$, $\mathbf{1}\otimes X=X\otimes \mathbf{1}=X$; (\romannumeral3) $Id_\mathbf{1}\otimes Id_\mathbf{1}=Id_\mathbf{1}$ and for any $f\in Mor(\mathcal{S})$, $Id_\mathbf{1}\otimes f=f\otimes Id_\mathbf{1}=f$; (\romannumeral4) $Id_\mathbf{1}\circ Id_\mathbf{1}=Id_\mathbf{1}$; (\romannumeral5) for objects in $Ob(\mathcal{S})$, their tensor products are same as in $\mathcal{S}$ and for morphisms in $Ob(\mathcal{S})$, their tensor products and compositions are same as in $\mathcal{S}$.

For a semi-groupal functor $\theta:\mathcal{S}_1\rightarrow\mathcal{S}_2$, $\mathfrak{F}(\theta):\mathcal{S}+\mathbf{1}\rightarrow\mathcal{S}+\mathbf{1}'$ is a monoidal functor extending $F$ such that $\mathfrak{F}(\theta)(\mathbf{1})=\mathbf{1}'$, $\mathfrak{F}(\theta)(Id_\mathbf{1})=Id_{\mathbf{1}'}$.

$(4)$ Definition of $\mathfrak{U}$. It is the forgetful functor, which just treats a monoidal category as a semi-groupal category and a monoidal functor as a semi-groupal functor.

$(5)$ Definition of $\mathcal{F}$. The definition of $\mathcal{F}$ is similar as that of $\mathcal{F}^+$, except that (\romannumeral1) for any tensor scheme $\mathcal{D}$, $Ob(\mathcal{F}(\mathcal{D}))=W(Ob(\mathcal{D}))$; (\romannumeral2) $Mor(\mathcal{F}(\mathcal{D}))=Diag(\mathcal{D})\sqcup \{\bigcirc\}$, where $s(\bigcirc)=t(\bigcirc)=\emptyset$. Other conditions on $\emptyset$ and $\bigcirc$ are same as those in the definition of $\mathfrak{F}$ (by identifying $\emptyset$ with $\mathbf{1}$ and $\bigcirc$ with $Id_\mathbf{1}$). For a morphism $\varphi:\mathcal{D}_1\rightarrow\mathcal{D}_2$ of semi-tensor schemes, $\mathcal{F}(\varphi):\mathcal{F}(\mathcal{D}_1)\rightarrow\mathcal{F}(\mathcal{D}_2)$ is a monoidal functor, which extends $\mathcal{F}^+(\varphi):\mathcal{F}^+(\mathcal{D}_1)\rightarrow\mathcal{F}^+(\mathcal{D}_2)$ by the conditions that $\mathcal{F}(\varphi)(\emptyset)=\emptyset$, $\mathcal{F}(\varphi)(\bigcirc)=\bigcirc$.

$(6)$ Definition of $\mathcal{U}$. Its definition is same as that of $\mathcal{U}^+$.

\begin{thm}\label{adjo}
As defined above, we have three adjunctions $\mathcal{F}^+:\mathbf{Semi.Ten}\rightleftharpoons \mathbf{Semi.Gro}:\mathcal{U}^+$, $\mathfrak{F}:\mathbf{Semi.Gro}\rightleftharpoons \mathbf{Mon.Cat}:\mathfrak{U}$, $\mathcal{F}:\mathbf{Semi.Ten}\rightleftharpoons \mathbf{Mon.Cat}:\mathcal{U}$, especially the diagram in Fig \ref{23} is a commutative diagram of adjunctions, that is,  $\mathcal{F}\dashv \mathcal{U}=(\mathfrak{F}\dashv \mathfrak{U})\circ(\mathcal{F}^+\dashv \mathcal{U}^+)$.
\end{thm}

\begin{proof}
We show that $\mathcal{F}^+:\mathbf{Semi.Ten}\rightleftharpoons \mathbf{Semi.Gro}:\mathcal{U}^+$ is an adjunction.
Given a semi-tensor scheme $\mathcal{D}$ and a semi-groupal category $\mathcal{S}$, we want to show that there is a natural bijection between $Hom(\mathcal{F}^+(\mathcal{D}), \mathcal{S})$ and $Hom(\mathcal{D}, \mathcal{U}^+(\mathcal{S}))$. In fact, for any $\vartheta:\mathcal{F}^+(\mathcal{D})\rightarrow \mathcal{S}$, there is a $\widehat{\vartheta}:\mathcal{D}\rightarrow \mathcal{U}^+(\mathcal{S})$,  whose definition is as follows. On objects $\widehat{\vartheta}$ is the restriction of $\vartheta$; for any $f:X_1X_2\cdots X_m\rightarrow Y_1Y_2\cdots Y_n$ in $Mor(\mathcal{D})$, $\widehat{\vartheta}(f)$ is the following prime POP-diagram in $\mathcal{S}$
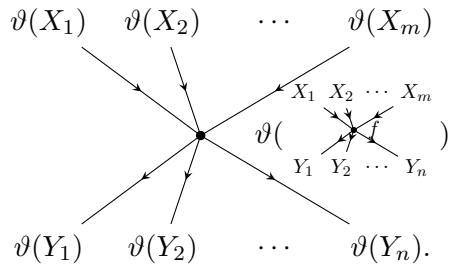
$$
\begin{matrix}
\begin{tikzpicture}[scale=1]
\node (v1) at (0,0) {};
\draw[fill] (0,0) circle [radius=0.055];
\node at (2,0){$\begin{matrix}\vartheta(\begin{matrix}\begin{tikzpicture}[scale=.33]
\node (v1) at (0,0) {};
\draw[fill] (0,0) circle [radius=0.1];
\node [scale=.7] at (0.8,0){$f$};
\node [scale=.7] (v2) at (-2,1.5) {$X_1$};
\node [scale=.7](v3) at (-0.5,1.5) {$X_2$};
\node [scale=.7](v4) at (1,1.5) {$\cdots$};
\node [scale=.7](v5) at (2.5,1.5) {$X_m$};
\node [scale=.7](v6) at (-2,-1.5) {$Y_1$};
\node [scale=.7](v7) at (-0.5,-1.5) {$Y_2$};
\node [scale=.7] at (1,-1.5) {$\cdots$};
\node [scale=.7](v8) at (2.5,-1.5) {$Y_n$};
\draw  (v2) -- (0,0)[postaction={decorate, decoration={markings,mark=at position .50 with {\arrow[black]{stealth}}}}];
\draw  (v3) -- (0,0)[postaction={decorate, decoration={markings,mark=at position .50 with {\arrow[black]{stealth}}}}];
\draw  (v5) -- (0,0)[postaction={decorate, decoration={markings,mark=at position .50 with {\arrow[black]{stealth}}}}];
\draw  (v6) -- (0,0)[postaction={decorate, decoration={markings,mark=at position .50 with {\arrowreversed[black]{stealth}}}}];
\draw  (v7) -- (0,0)[postaction={decorate, decoration={markings,mark=at position .50 with {\arrowreversed[black]{stealth}}}}];
\draw  (v8) -- (0,0)[postaction={decorate, decoration={markings,mark=at position .50 with {\arrowreversed[black]{stealth}}}}];
\end{tikzpicture}\end{matrix})\end{matrix}$};
\node (v2) at (-2,1.5) {$\vartheta(X_1)$};
\node (v3) at (-0.5,1.5) {$\vartheta(X_2)$};
\node (v4) at (1,1.5) {$\cdots$};
\node (v5) at (2.5,1.5) {$\vartheta(X_m)$};
\node (v6) at (-2,-1.5) {$\vartheta(Y_1)$};
\node (v7) at (-0.5,-1.5) {$\vartheta(Y_2)$};
\node at (1,-1.5) {$\cdots$};
\node (v8) at (2.5,-1.5) {$\vartheta(Y_n).$};
\draw  (v2) -- (0,0)[postaction={decorate, decoration={markings,mark=at position .50 with {\arrow[black]{stealth}}}}];
\draw  (v3) -- (0,0)[postaction={decorate, decoration={markings,mark=at position .50 with {\arrow[black]{stealth}}}}];
\draw  (v5) -- (0,0)[postaction={decorate, decoration={markings,mark=at position .50 with {\arrow[black]{stealth}}}}];
\draw  (v6) -- (0,0)[postaction={decorate, decoration={markings,mark=at position .50 with {\arrowreversed[black]{stealth}}}}];
\draw  (v7) -- (0,0)[postaction={decorate, decoration={markings,mark=at position .50 with {\arrowreversed[black]{stealth}}}}];
\draw  (v8) -- (0,0)[postaction={decorate, decoration={markings,mark=at position .50 with {\arrowreversed[black]{stealth}}}}];
\end{tikzpicture}
\end{matrix}
$$

Conversely, for any $\varphi:\mathcal{D}\rightarrow \mathcal{U}^+(\mathcal{S})$, there is a $\overline{\varphi}:\mathcal{F}^+(\mathcal{D})\rightarrow\mathcal{S}$, whose definition is as follows. For any $X_1X_2\cdots X_m\in Ob(\mathcal{F}^+(\mathcal{D}))$, $\overline{\varphi}(X_1X_2\cdots X_m)=\varphi(X_1)\otimes \varphi(X_2)\otimes \cdots\otimes \varphi(X_m)$; for any $\Gamma=[G,\prec,\gamma_o,\gamma_m]\in Diag(\mathcal{D})$, $\overline{\varphi}(\Gamma)$ is \textbf{the value of} the "pushforward" $[G,\prec,\gamma'_o,\gamma'_m]$, where for each edge $e$ of $G$, $\gamma'_o(e)=\varphi(\gamma(e))$, and for each internal vertex $v$ of $G$, $\gamma'_m(v)$ is \textbf{the value of} $\varphi(\gamma_m(v))$.
The naturality and bijectivity of this correspondence are easy to check. The other facts in this theorem are also easy to check.
\end{proof}

\begin{rem}\label{cg}
It can be directly checked that the associated monad of $\mathcal{F}\dashv \mathcal{U}$ (or $\mathcal{F}^+\dashv \mathcal{U}^+$) has a clear graphical description, which is given by \textbf{coarse-graining} of POP-graphs.
\end{rem}

\section{Unit convention and its generalizations}
In this section, we explain the unit convention as a kind of quotient construction and show an idea to generalize it.

For any monoidal category $\mathcal{M}$, the counit of $\mathcal{F}\dashv \mathcal{U}$ gives a monoidal functor $\varepsilon_\mathcal{M}:\mathcal{F}\mathcal{U}(\mathcal{M}) \rightarrow \mathcal{M}$, which sends a POP-diagram in $\mathcal{M}$ to its value.
In what follows, we will show that there is a quotient monoidal category $\Pi(\mathcal{M})$ of $\mathcal{F}\mathcal{U}(\mathcal{M})$ with the property that for any monoidal functor $\theta: \mathcal{F}\mathcal{U}(\mathcal{M})\rightarrow \mathcal{N}$ with $\theta(I_\mathcal{M})=I_\mathcal{N}$, there exists a unique monoidal functor $\widetilde{\theta}:\Pi(\mathcal{M})\rightarrow\mathcal{N}$ such that the following diagram commutes
$$\xymatrix{\mathcal{F}\mathcal{U}(\mathcal{M})\ar[rr]^{\pi_{\mathcal{M}}}\ar[rrd]^(0.6){\theta}&&\Pi(\mathcal{M})\ar[d]^{\exists !\widetilde{\theta}}\\ &&\mathcal{N},}$$
where $\pi_{\mathcal{M}}$ is the quotient monoidal functor, $I_\mathcal{M}$ and $I_\mathcal{N}$ are the unit objects of $\mathcal{M}$ and $\mathcal{N}$, respectively. (Under the adjunction $\mathcal{F}\dashv \mathcal{U}$, we can equivalently consider the class of morphisms of semi-tensor schemes $\varphi: \mathcal{U}(\mathcal{M})\rightarrow \mathcal{U}(\mathcal{N})$ with $\varphi(I_\mathcal{M})=I_{\mathcal{N}}$.)
\begin{rem}\label{rem}
Since $\theta: \mathcal{F}\mathcal{U}(\mathcal{M})\rightarrow \mathcal{N}$ is a monoidal functor, then $\theta(\emptyset)=I_\mathcal{N}$ and $\theta(\bigcirc)=Id_{I_\mathcal{N}}$. Note that in $\mathcal{F}\mathcal{U}(\mathcal{M})$, the identity morphism of $I_\mathcal{M}$ is $\begin{matrix}
\begin{tikzpicture}[scale=0.4]
\node (v2) at (-0.5,0.5) {};
\node [scale=0.5](v1) at (-0.5,2) {$I_\mathcal{M}$};
\node [scale=0.5](v3) at (-0.5,-1) {};
\draw  (v1) -- (v3)[postaction={decorate, decoration={markings,mark=at position 0.47 with {\arrow[black]{stealth}}}}];
\end{tikzpicture}
\end{matrix}$ (do not confuse with $\begin{matrix}
\begin{tikzpicture}[scale=0.4]
\node (v2) at (-0.5,0.5) {};
\draw[fill] (-0.5,0.5) circle [radius=0.09];
\node [scale=0.5](v1) at (-0.5,2) {$I_\mathcal{M}$};
\node [scale=0.5](v3) at (-0.5,-1) {$I_\mathcal{M}$};
\draw  (v1) -- (-0.5,0.5)[postaction={decorate, decoration={markings,mark=at position 0.47 with {\arrow[black]{stealth}}}}];
\draw  (-0.5,0.5)-- (v3)[postaction={decorate, decoration={markings,mark=at position 0.66 with {\arrow[black]{stealth}}}}];
\node [scale=0.7]at (0.5,0.5) {$Id_{I_\mathcal{M}}$};
\end{tikzpicture}
\end{matrix}$), so $\theta(I_\mathcal{M})=I_\mathcal{N}$ is equivalent to $\theta(
\begin{matrix}
\begin{tikzpicture}[scale=0.4]
\node (v2) at (-0.5,0.5) {};
\node [scale=0.5](v1) at (-0.5,2) {$I_\mathcal{M}$};
\node [scale=0.5](v3) at (-0.5,-1) {};
\draw  (v1) -- (v3)[postaction={decorate, decoration={markings,mark=at position 0.47 with {\arrow[black]{stealth}}}}];
\end{tikzpicture}
\end{matrix})=Id_{I_\mathcal{N}}$. Clearly, $\theta(I_\mathcal{M})=I_\mathcal{N}$ implies only that $\theta(\begin{matrix}
\begin{tikzpicture}[scale=0.4]
\node (v2) at (-0.5,0.5) {};
\draw[fill] (-0.5,0.5) circle [radius=0.09];
\node [scale=0.5](v1) at (-0.5,2) {$I_\mathcal{M}$};
\node [scale=0.5](v3) at (-0.5,-1) {$I_\mathcal{M}$};
\draw  (v1) -- (-0.5,0.5)[postaction={decorate, decoration={markings,mark=at position 0.47 with {\arrow[black]{stealth}}}}];
\draw  (-0.5,0.5)-- (v3)[postaction={decorate, decoration={markings,mark=at position 0.66 with {\arrow[black]{stealth}}}}];
\node [scale=0.7]at (0.5,0.5) {$Id_{I_\mathcal{M}}$};
\end{tikzpicture}
\end{matrix})\in Mor(I_\mathcal{N}, I_\mathcal{N})$, with no other constraints.
\end{rem}
$\Pi(\mathcal{M})$ is constructed as follows. $Ob(\Pi(\mathcal{M}))$ is the quotient set of $W(Ob(\mathcal{M}))$ by the equivalence relation $\sim_o$ that: $w_1\sim_o w_2$ if for any monoidal functor $\theta: \mathcal{F}\mathcal{U}(\mathcal{M})\rightarrow \mathcal{N}$ with  $\theta(I_\mathcal{M})=I_\mathcal{N}$, $\theta(w_1)=\theta(w_2).$
It is easy to see that $Ob(\Pi(\mathcal{M}))=W(Ob(\mathcal{M})-\{I_\mathcal{M}\})$; $Mor(\Pi(\mathcal{M}))$ is the quotient set of $Diag(\mathcal{M})\sqcup \{\bigcirc\}$ $(=Mor(\mathcal{F}\mathcal{U}(\mathcal{M})))$ by the equivalence relation $\sim_m$ that: $\Gamma_1\sim_m \Gamma_2$ if for any monoidal functor $\theta: \mathcal{F}\mathcal{U}(\mathcal{M})\rightarrow \mathcal{N}$ with  $\theta(I_\mathcal{M})=I_\mathcal{N}$, $\theta(\Gamma_1)=\theta(\Gamma_2).$ It turns out that an equivalence class of a POP-diagram $\Gamma$ can be uniquely represented by a "diagram" (called its \textbf{residue}) obtained from $\Gamma$ by \textbf{removing all edges labelled by $I_\mathcal{M}$} (see Fig \ref{isolated} and \ref{change1} for examples), or geometrically represented as a deformation class of \textbf{progressive plane diagram} in $\mathcal{M}$ (see Definition $1.3$ in \cite{[JS91]}) under the convention that removing all edges labelled by $I_\mathcal{M}$. Note that in this case isolated vertices labelled by $Id_{I_\mathcal{M}}$ appears as residues.  Clearly, $\bigcirc\sim_m \begin{matrix}
\begin{tikzpicture}[scale=0.4]
\node (v2) at (-0.5,0.5) {};
\node [scale=0.5](v1) at (-0.5,2) {$I_\mathcal{M}$};
\node [scale=0.5](v3) at (-0.5,-1) {};
\draw  (v1) -- (v3)[postaction={decorate, decoration={markings,mark=at position 0.47 with {\arrow[black]{stealth}}}}];
\end{tikzpicture}
\end{matrix}$, and their equivalence class is represented by the empty diagram $\bigcirc$.

\begin{rem}\label{six}
Similar as in the Definition of a PPG, the upward property is not an indispensable condition in the definition of a progressive plane graph. Following the idea of Di Battista, Tamassia \cite{[BT88]} and Kelly \cite{[K87]} that characterizing an upward plane graph as a subgraph of a plane $st$ graph, it is harmless to define a progressive plane graph as a subgraph of a BPP-graph. Just as pointed out in Caveat $3.2.$ of \cite{[S11]}, the allowed deformations of progressive plane graphs can be arbitrary planar isotopies, namely, we can say that two progressive plane graphs are \textbf{in the same deformation class}, or \textbf{equivalent}, if they are connected by a planar isotopy, where in each intermediate plane graph, the incidence relation of the underlying graph and the boundary of the plane box is unchanged.
\end{rem}

\begin{defn}
We call a progressive plane diagram in $\mathcal{M}$ \textbf{irreducible} if it contains no edges labelled by $I_{\mathcal{M}}$, otherwise we call it \textbf{reducible}.
\end{defn}

From now on, we will  say progressive plane diagrams, or just diagrams, for convenient to stand for their deformation classes. Then $Mor(\Pi(\mathcal{M}))$  is exactly the set of irreducible progressive plane diagram in $\mathcal{M}$.

\begin{figure}[h]
\centering
$$
\begin{matrix}
\begin{matrix}
\begin{tikzpicture}[scale=0.6]
\node (v2) at (0,0) {};
\draw[fill] (v2) circle [radius=0.09];
\node[scale=0.6] at (2.3,0) {$f:X\otimes Y\rightarrow U\otimes V$};
\node [scale=0.6](v1) at (-1.5,1.5) {$I$};
\node [scale=0.6](v3) at (-0.5,1.5) {$I$};
\node [scale=0.6](v4) at (0.5,1.5) {$X$};
\node [scale=0.6](v5) at (1.5,1.5) {$Y$};
\node [scale=0.6](v6) at (-1,-1.5) {$U$};
\node[scale=0.6] (v7) at (0,-1.5) {$I$};
\node[scale=0.6] (v8) at (1,-1.5) {$V$};
\draw  (v1) -- (0,0)[postaction={decorate, decoration={markings,mark=at position 0.47 with {\arrow[black]{stealth}}}}];
\draw  (v3) -- (0,0)[postaction={decorate, decoration={markings,mark=at position 0.47 with {\arrow[black]{stealth}}}}];
\draw  (v4) -- (0,0)[postaction={decorate, decoration={markings,mark=at position 0.47 with {\arrow[black]{stealth}}}}];
\draw  (v5) -- (0,0)[postaction={decorate, decoration={markings,mark=at position 0.47 with {\arrow[black]{stealth}}}}];
\draw  (0,0)-- (v6)[postaction={decorate, decoration={markings,mark=at position 0.65 with {\arrow[black]{stealth}}}}];
\draw (0,0) -- (v7)[postaction={decorate, decoration={markings,mark=at position 0.65 with {\arrow[black]{stealth}}}}];
\draw  (0,0) -- (v8)[postaction={decorate, decoration={markings,mark=at position 0.65 with {\arrow[black]{stealth}}}}];
\end{tikzpicture}
\end{matrix}&\longrightarrow&
\begin{matrix}
\begin{tikzpicture}[scale=0.6]
\node (v2) at (0,0) {};
\draw[fill] (v2) circle [radius=0.09];
\node[scale=0.6] at (2.3,0) {$f:X\otimes Y\rightarrow U\otimes V$};

\node [scale=0.6](v4) at (0.5,1.5) {$X$};
\node [scale=0.6](v5) at (1.5,1.5) {$Y$};
\node [scale=0.6](v6) at (-1,-1.5) {$U$};
\node[scale=0.6] (v8) at (1,-1.5) {$V$};
\draw  (v4) -- (0,0)[postaction={decorate, decoration={markings,mark=at position 0.47 with {\arrow[black]{stealth}}}}];
\draw  (v5) -- (0,0)[postaction={decorate, decoration={markings,mark=at position 0.47 with {\arrow[black]{stealth}}}}];
\draw  (0,0)-- (v6)[postaction={decorate, decoration={markings,mark=at position 0.65 with {\arrow[black]{stealth}}}}];

\draw  (0,0) -- (v8)[postaction={decorate, decoration={markings,mark=at position 0.65 with {\arrow[black]{stealth}}}}];
\end{tikzpicture}
\end{matrix}
\end{matrix}
$$
\caption{Removing all edges labelled by the unit object.}
\label{change1}
\end{figure}
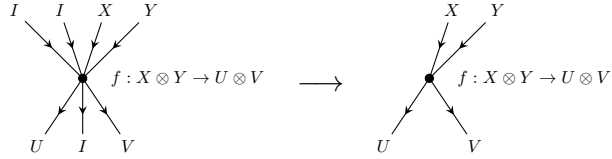
The source and target maps of $\Pi(\mathcal{M})$ are given by the domains and codomains of the residues (or irreducible progressive plane diagrams in $\mathcal{M}$), respectively.  In case that a residue has no input edge (or no output edge), the domain (or codomain) is defined to be $\emptyset$. Especially, both the domain and codomain of $\bigcirc$ are $\emptyset$. See Fig \ref{diagram1} for examples.

\begin{figure}[H]
\centering
$$
\begin{matrix}
\begin{matrix}
\begin{tikzpicture}[scale=0.5]
\node (v2) at (0,0) {};
\draw[fill] (v2) circle [radius=0.09];
\node[scale=0.7] at (2.7,0) {$f_1:X\otimes Y\rightarrow U\otimes V$};

\node [scale=0.7](v4) at (0.5,1.5) {$X$};
\node [scale=0.7](v5) at (1.5,1.5) {$Y$};
\node [scale=0.7](v6) at (-1,-1.5) {$U$};
\node[scale=0.7] (v8) at (1,-1.5) {$V$};
\draw  (v4) -- (0,0)[postaction={decorate, decoration={markings,mark=at position 0.47 with {\arrow[black]{stealth}}}}];
\draw  (v5) -- (0,0)[postaction={decorate, decoration={markings,mark=at position 0.47 with {\arrow[black]{stealth}}}}];
\draw  (0,0)-- (v6)[postaction={decorate, decoration={markings,mark=at position 0.65 with {\arrow[black]{stealth}}}}];

\draw  (0,0) -- (v8)[postaction={decorate, decoration={markings,mark=at position 0.65 with {\arrow[black]{stealth}}}}];
\node[scale=0.7] at (2,-2.2) {dom=$XY$, cod=$UV$};
\end{tikzpicture}
\end{matrix}&
\begin{matrix}
\begin{tikzpicture}[scale=0.5]
\node (v2) at (-3.5,-0.5) {};
\draw[fill] (-3.5,-0.5) circle [radius=0.09];
\node [scale=0.7](v1) at (-5,1.5) {$X_1$};
\node [scale=0.7](v3) at (-4,1.5) {$X_2$};
\node [scale=0.7](v4) at (-2,1.5) {$X_m$};
\draw  (v1) --(-3.5,-0.5)[postaction={decorate, decoration={markings,mark=at position 0.47 with {\arrow[black]{stealth}}}}];
\draw  (v3) -- (-3.5,-0.5)[postaction={decorate, decoration={markings,mark=at position 0.47 with {\arrow[black]{stealth}}}}];
\draw  (v4) -- (-3.5,-0.5)[postaction={decorate, decoration={markings,mark=at position 0.47 with {\arrow[black]{stealth}}}}];
\node at (-3,1.0) {$\cdots$};
\node  [scale=0.7]at (-3.5,-1) {$f_2:X_1\otimes X_2\otimes\cdots \otimes X_m\rightarrow I_\mathcal{M}$};
\node [scale=0.7] at (-3.5,-2) {dom=$X_1X_2\cdots X_m$, cod=$\emptyset$};
\end{tikzpicture}
\end{matrix}&
\begin{matrix}
\begin{tikzpicture}[scale=0.5]
\node (v5) at (-8.5,-0.5) {};
\draw[fill] (-8.5,-0.5) circle [radius=0.09];
\node [scale=0.7](v6) at (-9.5,-2.5) {$Y_1$};
\node [scale=0.7](v7) at (-8.5,-2.5) {$Y_2$};
\node [scale=0.7](v8) at (-6.5,-2.5) {$Y_n$};
\draw (-8.5,-0.5) -- (v6)[postaction={decorate, decoration={markings,mark=at position 0.6 with {\arrow[black]{stealth}}}}];
\draw  (-8.5,-0.5) -- (v7)[postaction={decorate, decoration={markings,mark=at position 0.6 with {\arrow[black]{stealth}}}}];
\draw (-8.5,-0.5) -- (v8)[postaction={decorate, decoration={markings,mark=at position 0.6 with {\arrow[black]{stealth}}}}];
\node at (-7.8,-2) {$\cdots$};
\node [scale=0.7]at (-8.5,0) {$f_3:I_\mathcal{M}\rightarrow Y_1\otimes Y_2\otimes\cdots\otimes Y_n$};
\node [scale=0.7]at (-8.5,-3.5) {dom=$\emptyset$, cod=$Y_1Y_2\cdots Y_n$};
\end{tikzpicture}
\end{matrix}&
\begin{matrix}
\begin{tikzpicture}[scale=0.5]
\node (v2) at (-3.5,-0.5) {};
\draw[fill] (-3.5,-0.5) circle [radius=0.09];
\node  [scale=0.8]at (-3.5,-1) {$f_4:I_\mathcal{M}\rightarrow I_\mathcal{M}$};
\node [scale=0.7]at (-3.5,-2.5) {dom=$\emptyset$, cod=$\emptyset$};
\node at (-4,1) {};
\end{tikzpicture}
\end{matrix}
\end{matrix}
$$
\caption{}
\label{diagram1}
\end{figure}

The monoidal functor $\pi_{\mathcal{M}}$ sends any word $w$ in $Ob(\mathcal{M})$ to its equivalence class ( a "reduction" of $w$ obtained by removing all $I_\mathcal{M}$s in $w$), and sends any POP-diagram in $\mathcal{M}$ to its residue. The universal property of $\pi_{\mathcal{M}}$ is easy to check.

Clearly, the counit $\varepsilon_\mathcal{M}:\mathcal{F}\mathcal{U}(\mathcal{M}) \rightarrow \mathcal{M}$ is a monoidal functor satisfying $\varepsilon_\mathcal{M}(I_\mathcal{M})=I_\mathcal{M}$, so there is a unique monoidal functor $\epsilon_\mathcal{M}:\Pi(\mathcal{M}) \rightarrow \mathcal{M}$  such that the following diagram commutes
$$\xymatrix{\mathcal{F}\mathcal{U}(\mathcal{M})\ar[rr]^{\pi_{\mathcal{M}}}\ar[rd]^(0.6){\varepsilon_\mathcal{M}}&&\Pi(\mathcal{M})\ar[ld]^{\epsilon_\mathcal{M}}\\ &\mathcal{M}.&}$$

For any monoidal functor $\vartheta:\mathcal{M}_1\rightarrow \mathcal{M}_2$, we define $\Pi(\vartheta)$ to be the unique monoidal functor $\Pi(\vartheta):\Pi(\mathcal{M}_1)\rightarrow \Pi(\mathcal{M}_2)$ such that the following diagram commutes $$\xymatrix{\mathcal{F}\mathcal{U}(\mathcal{M}_1)\ar[rr]^{\mathcal{F}\mathcal{U}(\vartheta)}\ar[d]^{\pi_{\mathcal{M}_1}}&&\ar[d]^{\pi_{\mathcal{M}_2}}\mathcal{F}\mathcal{U}(\mathcal{M}_2)\\\Pi(\mathcal{M}_1)\ar[rr]^{\Pi(\vartheta)}&&\Pi(\mathcal{M}_2).}$$

The following result is easy to check.
\begin{thm}
The above construction  defines a (quotient) functor $\Pi:\mathbf{Mon.Cat}\rightarrow \mathbf{Mon.Cat}$ and two natural transformations $\pi: \mathcal{F}\mathcal{U}\rightarrow \Pi$, $\epsilon:\Pi\rightarrow Id_{\mathbf{Mon.Cat}}$.
\end{thm}

We can extend the above construction quite freely by considering more constraints on the classes of $\theta: \mathcal{F}\mathcal{U}(\mathcal{M})\rightarrow \mathcal{N}$, or equivalently, $\varphi: \mathcal{U}(\mathcal{M})\rightarrow \mathcal{U}(\mathcal{N})$. For example, there is a quotient monoidal category $\widetilde{\Pi}(\mathcal{M})$ of $\mathcal{F}\mathcal{U}(\mathcal{M})$ with the property that for any monoidal functor $\theta: \mathcal{F}\mathcal{U}(\mathcal{M})\rightarrow \mathcal{N}$ with
$\theta(
\begin{matrix}
\begin{tikzpicture}[scale=0.4]
\node (v2) at (-0.5,0.5) {};
\draw[fill] (-0.5,0.5) circle [radius=0.09];
\node [scale=0.5](v1) at (-0.5,2) {$I_\mathcal{M}$};
\node [scale=0.5](v3) at (-0.5,-1) {$I_\mathcal{M}$};
\draw  (v1) -- (-0.5,0.5)[postaction={decorate, decoration={markings,mark=at position 0.47 with {\arrow[black]{stealth}}}}];
\draw  (-0.5,0.5)-- (v3)[postaction={decorate, decoration={markings,mark=at position 0.66 with {\arrow[black]{stealth}}}}];
\node [scale=0.7]at (0.5,0.5) {$Id_{I_\mathcal{M}}$};
\end{tikzpicture}
\end{matrix})=Id_{I_\mathcal{N}}$, there exists a unique monoidal functor $\widehat{\theta}:\widetilde{\Pi}(\mathcal{M})\rightarrow\mathcal{N}$ such that the following diagram commutes
$$\xymatrix{\mathcal{F}\mathcal{U}(\mathcal{M})\ar[rr]^{\widetilde{\pi}_{\mathcal{M}}}\ar[rrd]^(0.6){\theta}&&\widetilde{\Pi}(\mathcal{M})\ar[d]^{\exists !\widehat{\theta}}\\ &&\mathcal{N},}$$
where $\widetilde{\pi}_{\mathcal{M}}$ is the quotient monoidal functor.

\begin{rem}
$\theta(
\begin{matrix}
\begin{tikzpicture}[scale=0.4]
\node (v2) at (-0.5,0.5) {};
\draw[fill] (-0.5,0.5) circle [radius=0.09];
\node [scale=0.5](v1) at (-0.5,2) {$I_\mathcal{M}$};
\node [scale=0.5](v3) at (-0.5,-1) {$I_\mathcal{M}$};
\draw  (v1) -- (-0.5,0.5)[postaction={decorate, decoration={markings,mark=at position 0.47 with {\arrow[black]{stealth}}}}];
\draw  (-0.5,0.5)-- (v3)[postaction={decorate, decoration={markings,mark=at position 0.66 with {\arrow[black]{stealth}}}}];
\node [scale=0.7]at (0.5,0.5) {$Id_{I_\mathcal{M}}$};
\end{tikzpicture}
\end{matrix})=Id_{I_\mathcal{N}}$ implies that  $\theta(
\begin{matrix}
\begin{tikzpicture}[scale=0.4]
\node (v2) at (-0.5,0.5) {};
\node [scale=0.5](v1) at (-0.5,2) {$I_\mathcal{M}$};
\node [scale=0.5](v3) at (-0.5,-1) {};
\draw  (v1) -- (v3)[postaction={decorate, decoration={markings,mark=at position 0.47 with {\arrow[black]{stealth}}}}];
\end{tikzpicture}
\end{matrix})=Id_{I_\mathcal{N}}$ (or equivalently, $\theta(I_\mathcal{M})=I_\mathcal{N}$). Clearly, as mentioned in Remark \ref{rem}, the converse is not true, namely, $\theta(
\begin{matrix}
\begin{tikzpicture}[scale=0.4]
\node (v2) at (-0.5,0.5) {};
\node [scale=0.5](v1) at (-0.5,2) {$I_\mathcal{M}$};
\node [scale=0.5](v3) at (-0.5,-1) {};
\draw  (v1) -- (v3)[postaction={decorate, decoration={markings,mark=at position 0.47 with {\arrow[black]{stealth}}}}];
\end{tikzpicture}
\end{matrix})=Id_{I_\mathcal{N}}$ does not imply that $\theta(
\begin{matrix}
\begin{tikzpicture}[scale=0.4]
\node (v2) at (-0.5,0.5) {};
\draw[fill] (-0.5,0.5) circle [radius=0.09];
\node [scale=0.5](v1) at (-0.5,2) {$I_\mathcal{M}$};
\node [scale=0.5](v3) at (-0.5,-1) {$I_\mathcal{M}$};
\draw  (v1) -- (-0.5,0.5)[postaction={decorate, decoration={markings,mark=at position 0.47 with {\arrow[black]{stealth}}}}];
\draw  (-0.5,0.5)-- (v3)[postaction={decorate, decoration={markings,mark=at position 0.66 with {\arrow[black]{stealth}}}}];
\node [scale=0.7]at (0.5,0.5) {$Id_{I_\mathcal{M}}$};
\end{tikzpicture}
\end{matrix})=Id_{I_\mathcal{N}}$.
\end{rem}
$\widetilde{\Pi}(\mathcal{M})$ is constructed as follows. $Ob(\widetilde{\Pi}(\mathcal{M}))$ is same as that of $\Pi(\mathcal{M})$. $Mor(\widetilde{\Pi}(\mathcal{M}))$ is the quotient set of $Diag(\mathcal{M})\sqcup \{\bigcirc\}$ by the equivalence relation $\approx_m$ that: $\Gamma_1\approx_m \Gamma_2$ if for any monoidal functor $\theta: \mathcal{F}\mathcal{U}(\mathcal{M})\rightarrow \mathcal{N}$ with $\theta(
\begin{matrix}
\begin{tikzpicture}[scale=0.4]
\node (v2) at (-0.5,0.5) {};
\draw[fill] (-0.5,0.5) circle [radius=0.09];
\node [scale=0.5](v1) at (-0.5,2) {$I_\mathcal{M}$};
\node [scale=0.5](v3) at (-0.5,-1) {$I_\mathcal{M}$};
\draw  (v1) -- (-0.5,0.5)[postaction={decorate, decoration={markings,mark=at position 0.47 with {\arrow[black]{stealth}}}}];
\draw  (-0.5,0.5)-- (v3)[postaction={decorate, decoration={markings,mark=at position 0.66 with {\arrow[black]{stealth}}}}];
\node [scale=0.7]at (0.5,0.5) {$Id_{I_\mathcal{M}}$};
\end{tikzpicture}
\end{matrix})=Id_{I_\mathcal{N}}$, $\theta(\Gamma_1)=\theta(\Gamma_2).$
In this case,  an equivalence class of POP-diagrams can be represented uniquely by a residue obtained from a POP-diagram in $\mathcal{M}$ by first removing all edges labelled by $I_\mathcal{M}$ (just as previous) and then removing all isolated vertices labelled by $Id_{I_\mathcal{M}}$ (see Fig \ref{change2}).  Clearly, $\bigcirc\approx_m \begin{matrix}
\begin{tikzpicture}[scale=0.4]
\node (v2) at (-0.5,0.5) {};
\node [scale=0.5](v1) at (-0.5,2) {$I_\mathcal{M}$};
\node [scale=0.5](v3) at (-0.5,-1) {};
\draw  (v1) -- (v3)[postaction={decorate, decoration={markings,mark=at position 0.47 with {\arrow[black]{stealth}}}}];
\end{tikzpicture}
\end{matrix} \approx_m \begin{matrix}
\begin{tikzpicture}[scale=0.4]
\node (v2) at (-0.5,0.5) {};
\draw[fill] (-0.5,0.5) circle [radius=0.09];
\node [scale=0.5](v1) at (-0.5,2) {$I_\mathcal{M}$};
\node [scale=0.5](v3) at (-0.5,-1) {$I_\mathcal{M}$};
\draw  (v1) -- (-0.5,0.5)[postaction={decorate, decoration={markings,mark=at position 0.47 with {\arrow[black]{stealth}}}}];
\draw  (-0.5,0.5)-- (v3)[postaction={decorate, decoration={markings,mark=at position 0.66 with {\arrow[black]{stealth}}}}];
\node [scale=0.7]at (0.5,0.5) {$Id_{I_\mathcal{M}}$};
\end{tikzpicture}
\end{matrix}$, and their equivalence class is represented by $\bigcirc$.
\begin{figure}[H]
\centering
$$
\begin{matrix}
\begin{matrix}
\begin{tikzpicture}[scale=0.7]
\draw [dashed] (-2.5,2) rectangle (-0.5,-0.5);
\draw[fill] (-1.5,0.75) circle [radius=0.07];
\node [scale=0.7]at (-2,0.7) {$Id_I$};
\end{tikzpicture}
\end{matrix}&\longrightarrow&\bigcirc&=&
\begin{matrix}
 \begin{tikzpicture}[scale=0.7]
\draw [dashed] (-2.5,2) rectangle (-0.5,-0.5);
\node [scale=0.7]at (-2,0.7) {};
\end{tikzpicture}
\end{matrix}
\end{matrix}
$$
\caption{Removing isolated vertices labelled by the identity morphism of the unit object.}
\label{change2}
\end{figure}
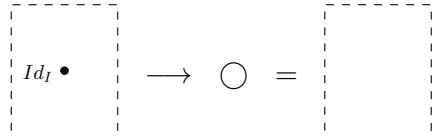

Note that this is exactly the unit convention reviewed in the introduction. Geometrically,  the residues in this case are exactly those irreducible progressive plane diagrams in $\mathcal{M}$ which have no isolated vertices labelled by $Id_{I_{\mathcal{M}}}$.
The source and target maps are, as that of $\Pi(\mathcal{M})$, given by the domains and codomains of residues, respectively.

The definition of $\widetilde{\pi}_{\mathcal{M}}$ is same as $\pi_{\mathcal{M}}$.  The universal property of $\widetilde{\pi}_{\mathcal{M}}$ is easy to check.
Similarly, $\varepsilon_\mathcal{M}(
\begin{matrix}
\begin{tikzpicture}[scale=0.4]
\node (v2) at (-0.5,0.5) {};
\draw[fill] (-0.5,0.5) circle [radius=0.09];
\node [scale=0.5](v1) at (-0.5,2) {$I_\mathcal{M}$};
\node [scale=0.5](v3) at (-0.5,-1) {$I_\mathcal{M}$};
\draw  (v1) -- (-0.5,0.5)[postaction={decorate, decoration={markings,mark=at position 0.47 with {\arrow[black]{stealth}}}}];
\draw  (-0.5,0.5)-- (v3)[postaction={decorate, decoration={markings,mark=at position 0.66 with {\arrow[black]{stealth}}}}];
\node [scale=0.7]at (0.5,0.5) {$Id_{I_\mathcal{M}}$};
\end{tikzpicture}
\end{matrix})=Id_{I_\mathcal{M}}$ implies that there is a unique monoidal functor $\widetilde{\epsilon}_\mathcal{M}:\widetilde{\Pi}(\mathcal{M}) \rightarrow \mathcal{M}$  such that the following diagram commutes
$$\xymatrix{\mathcal{F}\mathcal{U}(\mathcal{M})\ar[rr]^{\widetilde{\pi}_{\mathcal{M}}}\ar[rd]^(0.6){\varepsilon_\mathcal{M}}&&\widetilde{\Pi}(\mathcal{M})\ar[ld]^{\widetilde{\epsilon}_\mathcal{M}}\\ &\mathcal{M}.&}$$

For any monoidal functor $\vartheta:\mathcal{M}_1\rightarrow \mathcal{M}_2$, we define $\widetilde{\Pi}(\vartheta)$ to be the unique monoidal functor $\widetilde{\Pi}(\vartheta):\widetilde{\Pi}(\mathcal{M}_1)\rightarrow \widetilde{\Pi}(\mathcal{M}_2)$ such that the following diagram commutes $$\xymatrix{\mathcal{F}\mathcal{U}(\mathcal{M}_1)\ar[rr]^{\mathcal{F}\mathcal{U}(\vartheta)}\ar[d]^{\widetilde{\pi}_{\mathcal{M}_1}}&&\ar[d]^{\widetilde{\pi}_{\mathcal{M}_2}}\mathcal{F}\mathcal{U}(\mathcal{M}_2)\\\widetilde{\Pi}(\mathcal{M}_1)\ar[rr]^{\widetilde{\Pi}(\vartheta)}&&\widetilde{\Pi}(\mathcal{M}_2).}$$

As previous, we have the following result.

\begin{thm}
The above construction  defines a (quotient) functor $\widetilde{\Pi}:\mathbf{Mon.Cat}\rightarrow \mathbf{Mon.Cat}$ and two natural transformations $\widetilde{\pi}: \mathcal{F}\mathcal{U}\rightarrow \widetilde{\Pi}$, $\widetilde{\epsilon}:\widetilde{\Pi}\rightarrow Id_{\mathbf{Mon.Cat}}$. Moreover, there is a natural transformation $\omega:\Pi\rightarrow \widetilde{\Pi}$ such that for any $\mathcal{M}$, the following diagram commutes $$\xymatrix{\mathcal{F}\mathcal{U}(\mathcal{M})\ar[rr]^{\pi_{\mathcal{M}}}\ar[rrd]_(0.5){\widetilde{\pi}_\mathcal{M}}&&\Pi(\mathcal{M})\ar[d]^{\omega_\mathcal{M}}\\ &&\widetilde{\Pi}(\mathcal{M}),}$$
where $\omega_\mathcal{M}$ is an identity map on objects and acts on morphisms as an operation to remove all isolated vertices labelled by $Id_{I_{\mathcal{M}}}$.
\end{thm}

\begin{rem}
$(1)$ If we consider the constraints for $\theta: \mathcal{F}\mathcal{U}(\mathcal{M})\rightarrow \mathcal{N}$ that $\theta(\begin{matrix}
\begin{tikzpicture}[scale=0.4]
\node (v2) at (-0.5,0.5) {};
\draw[fill] (-0.5,0.5) circle [radius=0.09];
\node [scale=0.5](v1) at (-0.5,2) {$I_\mathcal{M}$};
\node [scale=0.5](v3) at (-0.5,-1) {$I_\mathcal{M}$};
\draw  (v1) -- (-0.5,0.5)[postaction={decorate, decoration={markings,mark=at position 0.47 with {\arrow[black]{stealth}}}}];
\draw  (-0.5,0.5)-- (v3)[postaction={decorate, decoration={markings,mark=at position 0.66 with {\arrow[black]{stealth}}}}];
\node [scale=0.9]at (0.3,0.5) {$\alpha$};
\end{tikzpicture}
\end{matrix})=Id_{I_\mathcal{N}}$ for all $\alpha\in Mor_{\mathcal{M}}(I_\mathcal{M},I_{\mathcal{M}})$, then the resulting residues will contain no isolated vertices.

$(2)$ If we consider the constraints for $\theta: \mathcal{F}\mathcal{U}(\mathcal{M})\rightarrow \mathcal{N}$ that $\theta(\Gamma)=Id_{I_\mathcal{N}}$ for all prime diagram $\Gamma$ with the unique internal vertex labelled by $Id_{I_\mathcal{M}}$, then the resulting residues will contain no vertices labelled by $Id_{I_\mathcal{M}}$.
\end{rem}

For more examples, we can arbitrarily choose a set $\Omega$ of POP-graphs, and consider \textbf{POP-diagrams of type} $\Omega$, namely, POP-diagrams with underlying POP-graphs in $\Omega$. Let $\Omega(\mathcal{M})$ denotes the set of POP-diagrams of type $\Omega$ in $\mathcal{\mathcal{M}}$. Just as previous, there should be a quotient monoidal category $\Pi_\Omega(\mathcal{M})$ of $\mathcal{F}\mathcal{U}(\mathcal{M})$ with the property that for any monoidal functor $\theta: \mathcal{F}\mathcal{U}(\mathcal{M})\rightarrow \mathcal{N}$  with $\theta(\Gamma)=Id_{I_\mathcal{N}}$ for any $\Gamma\in \Omega(M)$, there exists a unique monoidal functor $\overline{\theta}:\Pi_\Omega(\mathcal{M})\rightarrow\mathcal{N}$ such that the following diagram commutes
$$\xymatrix{\mathcal{F}\mathcal{U}(\mathcal{M})\ar[rr]^{\overline{\pi}_{\mathcal{M}}}\ar[rrd]^(0.6){\theta}&&\Pi_\Omega(\mathcal{M})\ar[d]^{\exists !\overline{\theta}}\\ &&\mathcal{N},}$$
where $\overline{\pi}_{\mathcal{M}}$ is the quotient monoidal functor. Clearly, $\Pi_\Omega$ is a functor and $\overline{\pi}$ is a natural transformation.

We call $\mathcal{M}$ an \textbf{$\Omega$-monoidal category} if for any $\Gamma\in \Omega(M)$, $\varepsilon_\mathcal{M}(\Gamma)=Id_{I_\mathcal{M}}$, where  $\varepsilon_\mathcal{M}:\mathcal{F}\mathcal{U}(\mathcal{M}) \rightarrow \mathcal{M}$ is given by the counit of $\mathcal{F}\dashv \mathcal{U}$. Clearly, if and only if  $\mathcal{M}$ is an $\Omega$-monoidal category,  there is a unique monoidal functor $\overline{\epsilon}_\mathcal{M}:\Pi_\Omega(\mathcal{M}) \rightarrow \mathcal{M}$  such that the following diagram commutes
$$\xymatrix{\mathcal{F}\mathcal{U}(\mathcal{M})\ar[rr]^{\overline{\pi}_{\mathcal{M}}}\ar[rd]^(0.6){\varepsilon_\mathcal{M}}&&\Pi_\Omega(\mathcal{M})\ar[ld]^{\overline{\epsilon}_\mathcal{M}}\\ &\mathcal{M}.&}$$

\begin{rem}
Note that the construction $\Omega(\mathcal{M})$ is functorial, that is, it can be pushed forward by any monoidal functor. The functorial property of $\Pi_\Omega$ follows from that of $\Omega(\mathcal{M})$. However, as mentioned previously, to produce a quotient monoidal category (under the assumption that $\mathbf{Mon.Cat}$ is cocomplete),  we can consider an arbitrary set of constraints on the class of $\theta: \mathcal{F}\mathcal{U}(\mathcal{M})\rightarrow \mathcal{N}$, even the constraints are not functorial with respect to monoidal functors. Nevertheless, we need the constraints to be functorial for constructing a quotient functor (the constraints in the constructions of $\Pi$ and $\widetilde{\Pi}$ are not of the form of $\Omega(\mathcal{M})$ but they are functorial).
\end{rem}

Let $\mathbf{\Omega.Mon}$ (called \textbf{$\Omega$-variety}) be the category of $\Omega$-monoidal categories, which is a full subcategory of $\mathbf{Mon.Cat}$. Then, restricted in $\mathbf{\Omega.Mon}$, $\overline{\epsilon}_\mathcal{M}$ should as previous  produce a natural transformation from $\Pi_\Omega$ to $Id_{\mathbf{\Omega.Mon}}$ $Id_{\Omega.Mon}$. The problem here is to give a concrete graphical convention describing $\Pi_\Omega$, just as those for $\Pi$ and $\widetilde{\Pi}$, which we call \textbf{$\Omega$-convention} as a generalization of the unit convention. Intuitively, such convention should be given by first removing all sub-diagrams of type $\Omega$ and then removing substrings of input edges and output edges whose labels form the domains or codomains of the diagrams of type $\Omega$.

\begin{rem}\label{pi}
It is reasonable to expect that there is a close relation between quotient monoidal categories of $\mathcal{POP}$ and $\Omega$-conventions above. Given a class $\Omega$ of POP-graphs, there should be a quotient monoidal category $\Pi_\Omega$ (abuse of notation) with the property that for any monoidal functors $\rho: \mathcal{POP}\rightarrow \mathcal{N}$ with $\rho\Big((G,\prec)\Big)=Id_{I_\mathcal{N}} $ for all $(G,\prec)\in\Omega$, there exists a unique monoidal functor $\widetilde{\rho}:\Pi_\Omega\rightarrow\mathcal{N}$ such that the following diagram commutes
$$\xymatrix{\mathcal{POP}\ar[rr]^{\varpi}\ar[rrd]^(0.6){\rho}&&\Pi_\Omega\ar[d]^{\exists !\widetilde{\rho}}\\ &&\mathcal{N},}$$
where $\varpi$ is the quotient monoidal functor. The  problem of $\Omega$-convention should relate with the  problem of finding a concrete graphical description of $\Pi_\Omega$. This situation is much similar as the theory of \textbf{PI-algebras} \cite{[DF04]}, where $\Omega$ and $\Omega$-monoidal categories play the roles just as the set of generators of a \textbf{$T$-ideal} and PI-algebras, respectively.
\end{rem}

\section{Revisit Joyal and Street's construction}
In this section, we show that the unit convention is naturally compatible with Joyal and Street's construction of a free monoidal category on a tensor scheme. More precisely, we construct two adjunctions which produce the functors  $\Pi$ and $\widetilde{\Pi}$ in previous section, respectively.

We first recall the definition of a tensor scheme and introduce morphisms for them.
\begin{defn}\label{T1}
A tensor scheme $\mathcal{T}$ consists of two (\textbf{possibly empty}) sets $Ob(\mathcal{T})$, $Mor(\mathcal{T})$  and two functions from $Mor(\mathcal{T})$ to $W(Ob(\mathcal{T}))$ $$s,t:Mor(\mathcal{T})\rightarrow W(Ob(\mathcal{T})),$$ which are called source and target maps, respectively.
\end{defn}
Note that, unlike the case of semi-tensor schemes, the emptiness of $Ob(\mathcal{T})$ does not imply the emptiness of $Mor(\mathcal{T})$.

\begin{defn}\label{mor1}
A \textbf{morphism} $\phi:\mathcal{T}_1\rightarrow\mathcal{T}_2$ of tensor schemes consists of two functions $\phi_o:Ob(\mathcal{T}_1)\sqcup\{\emptyset\}\rightarrow Ob(\mathcal{T}_2)\sqcup\{\emptyset\}$ and $\phi_m:Mor(\mathcal{T}_1)\rightarrow Mor(\mathcal{T}_2)$ such that $\phi_o(\emptyset)=\emptyset$ and the following diagram commutes
$$\xymatrix{W(Ob(\mathcal{T}_1))\ar[d]_{\widehat{\phi_o}}&\ar[d]^{\phi_m}Mor(\mathcal{T}_1)\ar[r]^{t_1}\ar[l]_{s_1}&\ar[d]^{\widehat{\phi_o}}W(Ob(\mathcal{T}_1))
\\W(Ob(\mathcal{T}_2))&\ar[l]_{s_2}\ar[r]^{t_2}Mor(\mathcal{T}_2)&W(Ob(\mathcal{T}_2)),}$$ where $\widehat{\phi_o}:W(Ob(\mathcal{T}_1))\rightarrow W(Ob(\mathcal{T}_2))$ is the unique morphism of monoids that naturally extends $\phi_o$. (For example, if $\phi_o(X_1)=Y_1$, $\phi_o(X_2)=\emptyset$, $\phi_o(X_3)=Y_3$, then $\widehat{\phi}_o(X_1X_2X_3)=Y_1Y_3$.)
\end{defn}

The category of tensor schemes and their morphisms is denoted by $\mathbf{Ten.Sch}$.

\begin{rem}
There are naturally two functors $\mathbf{T}: \mathbf{Semi.Ten} \rightarrow \mathbf{Ten.Sch}$ and  $\mathbf{S}: \mathbf{Ten.Sch} \rightarrow \mathbf{Semi.Ten}$, whose definitions are as follows.

$(1)$ Given a semi-tensor scheme $\mathcal{D}$, $\mathbf{T}(\mathcal{D}) (=\mathcal{D})$ is a tensor scheme with $Ob(\mathbf{T}(\mathcal{D}))=Ob(\mathcal{D})$,  $Mor(\mathbf{T}(\mathcal{D}))=Mor(\mathcal{D})$, and $s,t$ unchanging. Given a morphism $\varphi:\mathcal{D}_1\rightarrow \mathcal{D}_2$ of semi-tensor schemes, $\mathbf{T}(\varphi) :\mathbf{T}(\mathcal{D}_1)\rightarrow\mathbf{T}(\mathcal{D}_2)$ is a morphism of tensor schemes that extends $\varphi$ by $\mathbf{T}(\varphi)_o (\emptyset)=\emptyset$.

$(2)$ Given a tensor scheme $\mathcal{T}$, $\mathbf{S}(\mathcal{T})$ is a semi-tensor scheme with $Ob(\mathbf{S}(\mathcal{T}))=Ob(\mathcal{T})\sqcup\{\emptyset\}$,  $Mor(\mathbf{T}(\mathcal{T}))=Mor(\mathcal{T})$, and $s,t$ unchanging. Given a morphism $\phi:\mathcal{T}_1\rightarrow \mathcal{T}_2$ of tensor schemes, $\mathbf{S}(\phi) :\mathbf{S}(\mathcal{T}_1)\rightarrow\mathbf{S}(\mathcal{T}_2)$ is a morphism of semi-tensor schemes that extends $\phi$ by $\mathbf{S}(\phi)_o (\emptyset)=\emptyset$.

Both $\mathbf{T}$ and $\mathbf{S}$ are faithful, however, they \textbf{do not} form an adjunction.
\end{rem}

 Now we want to show that there is an adjunction $\mathbf{F}: \mathbf{Ten.Sch} \rightleftharpoons \mathbf{Mon.Cat}:\mathbf{U}$, whose definition is as follows.

$(1)$ $\mathbf{F}$ is given by Joyal and Street's construction of a free monoidal category on a tensor scheme. For any tensor scheme $\mathcal{T}$, $\mathbf{F}(\mathcal{T})$ is a monodial category with $Ob(\mathbf{F}(\mathcal{T}))=W(Ob(\mathcal{T}))$, and $Mor(\mathbf{F}(\mathcal{T}))$ being the set of empty diagram $\bigcirc$ and progressive plane diagrams in $\mathcal{T}$ (see Definition $34$ in \cite{[W13]} for the definition of a \textbf{progressive plane diagram in a tensor scheme}). The source and target maps are, as in previous section, given by the domain and codomain, respectively.

The unit object of $\mathbf{F}(\mathcal{T})$ is $\emptyset$  and its identity morphism is $\bigcirc$. For any non-unit object $w\in W^+(Ob(\mathcal{T}))$, its identity morphism is the invertible diagram with domain $w$ and codomain $w$. The tensor product of objects is given by concatenation of words. The tensor product and composition of morphisms are given by tensor product and composition of progressive plane diagrams, respectively.

Given a morphism $\phi:\mathcal{T}_1\rightarrow \mathcal{T}_2$ of tensor schemes, $\mathbf{F}(\phi):\mathbf{F}(\mathcal{T}_1)\rightarrow \mathbf{F}(\mathcal{T}_2)$ is a monoidal functor such that (\romannumeral1) on objects, $\mathbf{F}(\phi)$ is equal to $\widehat{\phi_o}$, especially $\mathbf{F}(\phi)(\emptyset)=\emptyset$; (\romannumeral2) on morphisms, $\mathbf{F}(\phi)(\bigcirc)=\bigcirc$ and it sends a progressive plane diagram $[G,\gamma_o,\gamma_m]$ in $\mathcal{T}_1$ (abuse of notation, $G$ here denotes a progressive plane graph) to \textbf{the residue of its "pushforward"}  $[G,\phi(\gamma_o),\phi(\gamma_m)]$ obtained by \textbf{removing all edges labelled by $\emptyset$}. Note that the "pushforward" $[G,\phi(\gamma_o),\phi(\gamma_m)]$ is in general \textbf{not} a progressive plane diagram in $\mathcal{T}_2$, as it may have some edges labelled by $\emptyset$, which do not appear in a progressive plane diagram in a tensor scheme according to its definition, see Remark below.

\begin{rem}
There is another way to understand the construction of $\mathbf{F}$. For any tensor scheme $\mathcal{T}$, $\mathbf{F}(\mathcal{T})$ is a quotient monoidal category of $\mathcal{F}(\mathbf{S}(\mathcal{T}))$ defined by the property that for any monoidal functor $\vartheta:\mathcal{F}(\mathbf{S}(\mathcal{T}))\rightarrow \mathcal{N}$ with $\vartheta(\emptyset)=I_{\mathcal{N}}$, there exists a unique monoidal functor $\widetilde{\vartheta}:\mathbf{F}(\mathcal{T})\rightarrow\mathcal{N}$ such that the following diagram commutes
$$\xymatrix{\mathcal{F}(\mathbf{S}(\mathcal{T}))\ar[rr]^{\pi_{\mathcal{T}}}\ar[rrd]^(0.6){\vartheta}&&\mathbf{F}(\mathcal{T})\ar[d]^{\exists !\widetilde{\vartheta}}\\ &&\mathcal{N},}$$
where $\pi_{\mathcal{T}}$ is the quotient monoidal functor. On objects, $\pi_{\mathcal{T}}$ is the identity map; on morphisms, $\pi_{\mathcal{T}}$ turns out to be given by the convention that removing all edges labelled by $\emptyset$.

The definition of a progressive plane diagram in a tensor scheme $\mathcal{T}$ coincides with the convention that removing all edges labelled by $\emptyset$. For any progressive plane diagram $[G,\gamma_o,\gamma_m]$ in $\mathcal{T}$, $(1)$ there is no edges of $G$ labelled by $\emptyset$, as $\emptyset\not\in Ob(\mathcal{T})$ and for any edge $e$ of $G$, by definition $\gamma_o(e)\in Ob(\mathcal{T})$; $(2)$ for an \textbf{inner node} $v$ of $G$, if it has no incoming edge, then $s(\gamma_m(v))=\emptyset$; if it has no outgoing edge, then $t(\gamma_m(v))=\emptyset$; if it is isolated, then $s(\gamma_m(v))=t(\gamma_m(v))=\emptyset$. See Fig \ref{diagram} for examples.

\begin{figure}[h]
\centering
$$
\begin{matrix}
\begin{matrix}
\begin{tikzpicture}[scale=0.5]
\node (v2) at (0,0) {};
\draw[fill] (v2) circle [radius=0.09];
\node[scale=0.7] at (2.7,0) {$f_1:XY\rightarrow UV$};

\node [scale=0.7](v4) at (0.5,1.5) {$X$};
\node [scale=0.7](v5) at (1.5,1.5) {$Y$};
\node [scale=0.7](v6) at (-1,-1.5) {$U$};
\node[scale=0.7] (v8) at (1,-1.5) {$V$};
\draw  (v4) -- (0,0)[postaction={decorate, decoration={markings,mark=at position 0.47 with {\arrow[black]{stealth}}}}];
\draw  (v5) -- (0,0)[postaction={decorate, decoration={markings,mark=at position 0.47 with {\arrow[black]{stealth}}}}];
\draw  (0,0)-- (v6)[postaction={decorate, decoration={markings,mark=at position 0.65 with {\arrow[black]{stealth}}}}];

\draw  (0,0) -- (v8)[postaction={decorate, decoration={markings,mark=at position 0.65 with {\arrow[black]{stealth}}}}];
\node[scale=0.7] at (1.3,-2.2) {dom=$XY$, cod=$UV$};
\end{tikzpicture}
\end{matrix}&
\begin{matrix}
\begin{tikzpicture}[scale=0.5]
\node (v2) at (-3.5,-0.5) {};
\draw[fill] (-3.5,-0.5) circle [radius=0.09];
\node [scale=0.7](v1) at (-5,1.5) {$X_1$};
\node [scale=0.7](v3) at (-4,1.5) {$X_2$};
\node [scale=0.7](v4) at (-2,1.5) {$X_m$};
\draw  (v1) --(-3.5,-0.5)[postaction={decorate, decoration={markings,mark=at position 0.47 with {\arrow[black]{stealth}}}}];
\draw  (v3) -- (-3.5,-0.5)[postaction={decorate, decoration={markings,mark=at position 0.47 with {\arrow[black]{stealth}}}}];
\draw  (v4) -- (-3.5,-0.5)[postaction={decorate, decoration={markings,mark=at position 0.47 with {\arrow[black]{stealth}}}}];
\node at (-3,1.0) {$\cdots$};
\node  [scale=0.7]at (-3.5,-1) {$f_2:X_1X_2\cdots  X_m\rightarrow \emptyset$};
\node [scale=0.7] at (-3.5,-2) {dom=$X_1X_2\cdots X_m$, cod=$\emptyset$};
\end{tikzpicture}
\end{matrix}&
\begin{matrix}
\begin{tikzpicture}[scale=0.5]
\node (v5) at (-8.5,-0.5) {};
\draw[fill] (-8.5,-0.5) circle [radius=0.09];
\node [scale=0.7](v6) at (-9.5,-2.5) {$Y_1$};
\node [scale=0.7](v7) at (-8.5,-2.5) {$Y_2$};
\node [scale=0.7](v8) at (-6.5,-2.5) {$Y_n$};
\draw (-8.5,-0.5) -- (v6)[postaction={decorate, decoration={markings,mark=at position 0.6 with {\arrow[black]{stealth}}}}];
\draw  (-8.5,-0.5) -- (v7)[postaction={decorate, decoration={markings,mark=at position 0.6 with {\arrow[black]{stealth}}}}];
\draw (-8.5,-0.5) -- (v8)[postaction={decorate, decoration={markings,mark=at position 0.6 with {\arrow[black]{stealth}}}}];
\node at (-7.8,-2) {$\cdots$};
\node [scale=0.7]at (-8.5,0) {$f_3:\emptyset\rightarrow Y_1 Y_2\cdots Y_n$};
\node [scale=0.7]at (-8.5,-3.5) {dom=$\emptyset$, cod=$Y_1Y_2\cdots Y_n$};
\end{tikzpicture}
\end{matrix}&
\begin{matrix}
\begin{tikzpicture}[scale=0.5]
\node (v2) at (-3.5,-0.5) {};
\draw[fill] (-3.5,-0.5) circle [radius=0.09];
\node  [scale=0.8]at (-3.5,-1.3) {$f_4:\emptyset\rightarrow \emptyset$};
\node [scale=0.7]at (-3.5,-2.5) {dom=$\emptyset$, cod=$\emptyset$};
\node at (-4,1) {};
\end{tikzpicture}
\end{matrix}
\end{matrix}
$$
\caption{}
\label{diagram}
\end{figure}

\end{rem}

$(2)$ Given a monoidal category $\mathcal{M}$, $\mathbf{U}(\mathcal{M})$ is a tensor scheme. $Ob(\mathbf{U}(\mathcal{M}))=Ob(\mathcal{M})-\{I_\mathcal{M}\}$. Clearly, when $\mathcal{M}$ has only one object, $Ob(\mathbf{U}(\mathcal{M}))$ is an empty set. In fact, $Ob(\mathbf{U}(\mathcal{M}))\sqcup\{\emptyset\}$ should be conceptually understood as the quotient set of $Ob(\mathcal{M})$ (as a subset of $W(Ob(\mathcal{M}))$) by the equivalence relation $\sim_o$ in the construction of $\Pi(\mathcal{M})$ in previous section, where $\emptyset$ represents the equivalence class of $I_{\mathcal{M}}$. Similarly, $Mor(\mathbf{U}(\mathcal{M}))$ is the quotient set of $Prim(\mathcal{M})$ (as a subset of $Diag(\mathcal{M})\sqcup\{\bigcirc\}$) by the equivalence relation $\sim_m$ in the construction of $\Pi(\mathcal{M})$. As in previous section, each equivalence class of prime POP-diagrams can be uniquely represented by a residue obtained under the convention that \textbf{removing all edges labelled by $I_\mathcal{M}$}, which is, geometrically, an irreducible prime progressive plane diagram in $\mathcal{M}$, see Fig \ref{diagram1} for examples. The source and target maps are, as in previous section, given by the domain and codomain, respectively.

Given a monoidal functor $\theta: \mathcal{M}_1\rightarrow \mathcal{M}_2$, $\mathbf{U}(\theta):\mathbf{U}(\mathcal{M}_1)\rightarrow\mathbf{U}(\mathcal{M}_2)$ is a morphism of tensor scheme such that (\romannumeral1) on objects, $\mathbf{U}(\theta)(\emptyset)=\emptyset$ and for any $X\in Ob(\mathcal{M})-\{I_\mathcal{M}\}$, if $\theta(X)\not=I_{\mathcal{M}}$, then $\mathbf{U}(\theta)(X)=\theta(X)$, otherwise, $\mathbf{U}(\theta)(X)=\emptyset$; (\romannumeral2) on morphisms, it sends an irreducible prime progressive plane diagram  $[G,\gamma_o,\gamma_m]$ in $\mathcal{M}_1$ to the residue of the "pushforward" $[G,\theta(\gamma_o),\theta(\gamma_m)]$ obtained by removing all edges labelled by $I_{\mathcal{M}_2}$. (The "pushforward" $[G,\theta(\gamma_o),\theta(\gamma_m)]$ is in general \textbf{not} an irreducible progressive plane diagram in $\mathcal{M}_2$.)

\begin{thm}\label{ad}
Defined as above, $\mathbf{F}: \mathbf{Ten.Sch} \rightleftharpoons \mathbf{Mon.Cat}:\mathbf{U}$ is an adjunction. Moreover, $\mathbf{F}\mathbf{U}=\Pi$ and the counit is given by $\epsilon:\Pi\rightarrow Id_{\mathbf{Mon.Cat}}$.
\end{thm}
\begin{proof}
The proof of the adjointness of $\mathbf{F}, \mathbf{U}$ is similar as that of $\mathcal{F}^+, \mathcal{U}^+$ in Theorem \ref{adjo}.
Given a tensor scheme $\mathcal{T}$ and a monoidal category $\mathcal{M}$, we want to show that there is a natural bijection between $Hom(\mathbf{F}(\mathcal{T}), \mathcal{M})$ and $Hom(\mathcal{T}, \mathbf{U}(\mathcal{M}))$. In fact, for any $\vartheta:\mathbf{F}(\mathcal{T})\rightarrow \mathcal{M}$, there is a $\widehat{\vartheta}:\mathcal{T}\rightarrow \mathbf{U}(\mathcal{M})$, whose definition is as follows. For any $X\in Ob(\mathcal{T})$, $\widehat{\vartheta}(X)$ is the equivalence class of $\vartheta(X)$, especially when $\vartheta(X)=I_{\mathcal{M}}$, $\widehat{\vartheta}(X)=\emptyset$; for any $f:X_1X_2\cdots X_m\rightarrow Y_1Y_2\cdots Y_n$ in $Mor(\mathcal{T})$, $\widehat{\vartheta}(f)$ is the residue or the equivalence class (under $\backsim_m$) of the following prime POP-diagram in $\mathcal{M}$
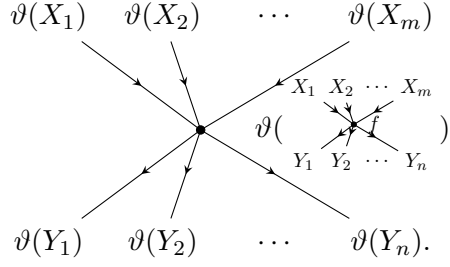
$$
\begin{matrix}
\begin{tikzpicture}[scale=1]
\node (v1) at (0,0) {};
\draw[fill] (0,0) circle [radius=0.055];
\node at (2,0){$\begin{matrix}\vartheta(\begin{matrix}\begin{tikzpicture}[scale=.33]
\node (v1) at (0,0) {};
\draw[fill] (0,0) circle [radius=0.1];
\node [scale=.7] at (0.8,0){$f$};
\node [scale=.7] (v2) at (-2,1.5) {$X_1$};
\node [scale=.7](v3) at (-0.5,1.5) {$X_2$};
\node [scale=.7](v4) at (1,1.5) {$\cdots$};
\node [scale=.7](v5) at (2.5,1.5) {$X_m$};
\node [scale=.7](v6) at (-2,-1.5) {$Y_1$};
\node [scale=.7](v7) at (-0.5,-1.5) {$Y_2$};
\node [scale=.7] at (1,-1.5) {$\cdots$};
\node [scale=.7](v8) at (2.5,-1.5) {$Y_n$};
\draw  (v2) -- (0,0)[postaction={decorate, decoration={markings,mark=at position .50 with {\arrow[black]{stealth}}}}];
\draw  (v3) -- (0,0)[postaction={decorate, decoration={markings,mark=at position .50 with {\arrow[black]{stealth}}}}];
\draw  (v5) -- (0,0)[postaction={decorate, decoration={markings,mark=at position .50 with {\arrow[black]{stealth}}}}];
\draw  (v6) -- (0,0)[postaction={decorate, decoration={markings,mark=at position .50 with {\arrowreversed[black]{stealth}}}}];
\draw  (v7) -- (0,0)[postaction={decorate, decoration={markings,mark=at position .50 with {\arrowreversed[black]{stealth}}}}];
\draw  (v8) -- (0,0)[postaction={decorate, decoration={markings,mark=at position .50 with {\arrowreversed[black]{stealth}}}}];
\end{tikzpicture}\end{matrix})\end{matrix}$};
\node (v2) at (-2,1.5) {$\vartheta(X_1)$};
\node (v3) at (-0.5,1.5) {$\vartheta(X_2)$};
\node (v4) at (1,1.5) {$\cdots$};
\node (v5) at (2.5,1.5) {$\vartheta(X_m)$};
\node (v6) at (-2,-1.5) {$\vartheta(Y_1)$};
\node (v7) at (-0.5,-1.5) {$\vartheta(Y_2)$};
\node at (1,-1.5) {$\cdots$};
\node (v8) at (2.5,-1.5) {$\vartheta(Y_n).$};
\draw  (v2) -- (0,0)[postaction={decorate, decoration={markings,mark=at position .50 with {\arrow[black]{stealth}}}}];
\draw  (v3) -- (0,0)[postaction={decorate, decoration={markings,mark=at position .50 with {\arrow[black]{stealth}}}}];
\draw  (v5) -- (0,0)[postaction={decorate, decoration={markings,mark=at position .50 with {\arrow[black]{stealth}}}}];
\draw  (v6) -- (0,0)[postaction={decorate, decoration={markings,mark=at position .50 with {\arrowreversed[black]{stealth}}}}];
\draw  (v7) -- (0,0)[postaction={decorate, decoration={markings,mark=at position .50 with {\arrowreversed[black]{stealth}}}}];
\draw  (v8) -- (0,0)[postaction={decorate, decoration={markings,mark=at position .50 with {\arrowreversed[black]{stealth}}}}];
\end{tikzpicture}
\end{matrix}
$$
In case that $m=0$, that is, $s(f)=\emptyset$, then $\widehat{\vartheta}(f)$ is the residue or the equivalence class (under $\backsim_m$) of the following prime POP-diagram in $\mathcal{M}$

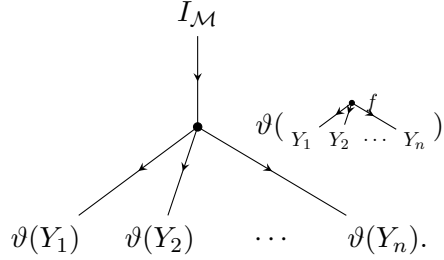
$$
\begin{matrix}
\begin{tikzpicture}[scale=1]
\node (v1) at (0,0) {};
\draw[fill] (0,0) circle [radius=0.055];
\node at (2,0){$\begin{matrix}\vartheta(\begin{matrix}\begin{tikzpicture}[scale=.33]
\node (v1) at (0,0) {};
\draw[fill] (0,0) circle [radius=0.1];
\node [scale=.7] at (0.8,0){$f$};

\node [scale=.7](v6) at (-2,-1.5) {$Y_1$};
\node [scale=.7](v7) at (-0.5,-1.5) {$Y_2$};
\node [scale=.7] at (1,-1.5) {$\cdots$};
\node [scale=.7](v8) at (2.5,-1.5) {$Y_n$};
\draw  (v6) -- (0,0)[postaction={decorate, decoration={markings,mark=at position .50 with {\arrowreversed[black]{stealth}}}}];
\draw  (v7) -- (0,0)[postaction={decorate, decoration={markings,mark=at position .50 with {\arrowreversed[black]{stealth}}}}];
\draw  (v8) -- (0,0)[postaction={decorate, decoration={markings,mark=at position .50 with {\arrowreversed[black]{stealth}}}}];
\end{tikzpicture}\end{matrix})\end{matrix}$};
\node (v2) at (0,1.5) {$I_{\mathcal{M}}$};

\node (v6) at (-2,-1.5) {$\vartheta(Y_1)$};
\node (v7) at (-0.5,-1.5) {$\vartheta(Y_2)$};
\node at (1,-1.5) {$\cdots$};
\node (v8) at (2.5,-1.5) {$\vartheta(Y_n).$};
\draw  (v2) -- (0,0)[postaction={decorate, decoration={markings,mark=at position .50 with {\arrow[black]{stealth}}}}];

\draw  (v6) -- (0,0)[postaction={decorate, decoration={markings,mark=at position .50 with {\arrowreversed[black]{stealth}}}}];
\draw  (v7) -- (0,0)[postaction={decorate, decoration={markings,mark=at position .50 with {\arrowreversed[black]{stealth}}}}];
\draw  (v8) -- (0,0)[postaction={decorate, decoration={markings,mark=at position .50 with {\arrowreversed[black]{stealth}}}}];
\end{tikzpicture}
\end{matrix}
$$
The case of $n=0$ is similar. In particular, in case that $m=n=0$, $\widehat{\vartheta}(f)=\begin{matrix}
\begin{tikzpicture}[scale=1]
\node (v1) at (0,0) {};
\draw[fill] (0,0) circle [radius=0.055];
\node [right]at (0,0){$\begin{matrix}\vartheta(\begin{matrix}\begin{tikzpicture}[scale=.33]
\draw[fill] (0,0) circle [radius=0.1];
\node [right,scale=.7] at (0,0){$f$};
\end{tikzpicture}\end{matrix})\end{matrix}$};
\end{tikzpicture}
\end{matrix}.$

Conversely, for any $\varphi:\mathcal{T}\rightarrow \mathbf{U}(\mathcal{M})$, there is a $\overline{\varphi}:\mathbf{F}(\mathcal{T})\rightarrow\mathcal{M}$, whose definition is as follows. For any $X_1X_2\cdots X_m\in Ob(\mathbf{F}(\mathcal{T}))$, $\overline{\varphi}(X_1X_2\cdots X_m)=\varphi(X_1)\otimes \varphi(X_2)\otimes \cdots\otimes \varphi(X_m)$, especially $\overline{\varphi}(\emptyset)=I_\mathcal{M}$;  for any progressive plane diagram $\Gamma=[G,\gamma_o,\gamma_m]$ in $\mathcal{T}$ , $\overline{\varphi}(\Gamma)$ is \textbf{the value of} the "pushforward" $[G,\gamma'_o,\gamma'_m]$, where for each edge $e$ of $G$, $\gamma'_o(e)=\varphi(\gamma(e))$, and for each inner node $v$ of $G$, $\gamma'_m(v)$ is \textbf{the value of} $\varphi(\gamma_m(v))$. (Clearly, in general the "pushforward" $[G,\gamma'_o,\gamma'_m]$ is not an irreducible progressive plane diagram in $\mathcal{M}$, however its value is well-defined.) In particular, $\overline{\varphi}(\bigcirc)=Id_{I_{\mathcal{M}}}$.

The naturality and bijectivity of this correspondence are easy to check, which imply that $\mathbf{F}: \mathbf{Ten.Sch} \rightleftharpoons \mathbf{Mon.Cat}:\mathbf{U}$ is an adjunction.  The other facts in this theorem are also easy to check.
\end{proof}

We have constructed an adjunction for $\Pi$. To do the same thing for $\widetilde{\Pi}$, we need the following notions.
\begin{defn}\label{T1}
An \textbf{$\bigcirc$-marked tensor scheme} $\mathcal{T}$ is a tensor scheme with a distinguished element of $ Mor(\mathcal{T})$, denoted as $\bigcirc$, such that  $s(\bigcirc)=t(\bigcirc)=\emptyset$.
\end{defn}

\begin{defn}\label{mor2}
A \textbf{morphism} $\phi:\mathcal{T}_1\rightarrow\mathcal{T}_2$ of $\bigcirc$-marked tensor schemes is a morphism of tensor schemes such that $\phi_m(\bigcirc)=\bigcirc$.
\end{defn}
The category of $\bigcirc$-marked tensor schemes and their morphisms is denoted as $\mathbf{\bigcirc.Ten}$. Now we want to show that there is an adjunction $\widetilde{\mathbf{F}}: \mathbf{\bigcirc.Ten} \rightleftharpoons \mathbf{Mon.Cat}:\widetilde{\mathbf{U}}$, whose definition is as follows.

$(1)$ Given an $\bigcirc$-marked tensor scheme $\mathcal{T}$, $\widetilde{\mathbf{F}}(\mathcal{T})$ is the quotient monoidal category of $\mathbf{F}(\mathcal{T})$ defined by  the property that for any monoidal functor $\vartheta:\mathcal{F}(\mathcal{T})\rightarrow \mathcal{N}$ with $\vartheta(\begin{matrix}
\begin{tikzpicture}[scale=1]
\node (v1) at (0,0) {};
\draw[fill] (0,0) circle [radius=0.055];
\node [right]at (0,0){$\bigcirc$};
\end{tikzpicture}
\end{matrix})=Id_{I_\mathcal{N}}$, there exists a unique monoidal functor $\widehat{\vartheta}:\mathbf{F}(\mathcal{T})\rightarrow\mathcal{N}$ such that the following diagram commutes
$$\xymatrix{\mathbf{F}(\mathcal{T})\ar[rr]^{\widetilde{\pi}_{\mathcal{T}}}\ar[rrd]^(0.6){\vartheta}&&\widetilde{\mathbf{F}}(\mathcal{T})\ar[d]^{\exists !\widehat{\vartheta}}\\ &&\mathcal{N},}$$
where $\widetilde{\pi}_{\mathcal{T}}$ is the quotient monoidal functor. On objects, $\widetilde{\pi}_{\mathcal{T}}$ is the identity map; on morphisms, it turns out that $\widetilde{\pi}_{\mathcal{T}}$ is given by the convention that removing all isolated vertices labelled by $\bigcirc$.

Given a morphism $\phi:\mathcal{T}_1\rightarrow\mathcal{T}_2$ of $\bigcirc$-marked tensor schemes,  $\widetilde{\mathbf{F}}(\phi)$ is defined to be the unique monoidal functor $\widetilde{\mathbf{F}}(\phi):\widetilde{\mathbf{F}}(\mathcal{T}_1)\rightarrow \widetilde{\mathbf{F}}(\mathcal{T}_2)$ such that the following diagram commutes $$\xymatrix{\mathbf{F}(\mathcal{T}_1)\ar[rr]^{\mathbf{F}(\phi)}\ar[d]^{\widetilde{\pi}_{\mathcal{T}_1}}&&\ar[d]^{\widetilde{\pi}_{\mathcal{T}_2}}\mathbf{F}(\mathcal{T}_2)\\ \widetilde{\mathbf{F}}(\mathcal{T}_1)\ar[rr]^{\widetilde{\mathbf{F}}(\phi)}&&\widetilde{\mathbf{F}}(\mathcal{T}_2).}$$

$(2)$ Given a monoidal category $\mathcal{M}$, $\widetilde{\mathbf{U}}(\mathcal{M})=\mathbf{U}(\mathcal{M})$ with $\begin{matrix}
\begin{tikzpicture}[scale=1]
\node (v1) at (0,0) {};
\draw[fill] (0,0) circle [radius=0.055];
\node [right]at (0,0){$Id_{I_{\mathcal{M}}}$};
\end{tikzpicture}
\end{matrix}$ as the distinguished element in $Mor(\widetilde{\mathbf{U}}(\mathcal{M}))$, that is, $\begin{matrix}
\begin{tikzpicture}[scale=1]
\node (v1) at (0,0) {};
\draw[fill] (0,0) circle [radius=0.055];
\node [right]at (0,0){$Id_{I_{\mathcal{M}}}$};
\end{tikzpicture}
\end{matrix}=\bigcirc$. For any monoidal functor $\theta: \mathcal{M}_1\rightarrow \mathcal{M}_2$, $\widetilde{\mathbf{U}}(\theta)=\mathbf{U}(\theta)$.

The following result can be directly checked.
\begin{thm}\label{ad2}
Defined as above, $\widetilde{\mathbf{F}}: \mathbf{\bigcirc.Ten} \rightleftharpoons \mathbf{Mon.Cat}:\widetilde{\mathbf{U}}$ is an adjunction. Moreover, $\widetilde{\mathbf{F}}\widetilde{\mathbf{U}}=\widetilde{\Pi}$ and the counit is given by $\widetilde{\epsilon}:\widetilde{\Pi}\rightarrow Id_{\mathbf{Mon.Cat}}$.
\end{thm}

As in previous section, we can go further.   Fixing an $\Omega$, we can consider an adjunction  $\mathbf{F}_\Omega: \mathbf{Ten.Sch} \rightleftharpoons \mathbf{\Omega.Mon}:\mathbf{U}_\Omega$, where both $\mathbf{F}_\Omega$ and $\mathbf{U}_\Omega$ should be given by $\Omega$-convention. Applying on a tensor scheme, $\mathbf{F}_\Omega$ will produce an $\Omega$ monoidal category, which is relatively free with respect to the full subcategory $\mathbf{\Omega.Mon}$ of $\mathbf{Mon.Cat}$.

\section*{Acknowledgement}
Xuexing Lu would like to thank Ross Street for comments and encouragement and John Power for encouragement and a lot of concrete advice for improving this paper.  He also would like to thank  Maurizio Patrignani for careful check of the proofs in an earlier version of this paper. He also thanks Jinsong Wu for reminding him to change the terminologies and pointing out some mistakes in an earlier version of this paper.
This work is partly supported by the National Scientific Foundation of China No.11431010 and 11571329 and  "the Fundamental Research Funds for the Central Universities".


\textbf{Xuexing Lu}\hfill \\  Email: xxlu@mail.ustc.edu.cn

\textbf{Yu Ye} \hfill \\ Email: yeyu@ustc.edu.cn

\textbf{Sen Hu}\hfill \\  Email: shu@ustc.edu.cn

\end{document}